\documentclass[reqno]{amsart}

\usepackage{hyperref}
\usepackage{amsmath}
\usepackage{amssymb}
\usepackage{amsfonts}
\usepackage{amsthm}
\usepackage{latexsym}
\usepackage{esint}
\usepackage{mathtools}
\usepackage{mathrsfs}
\usepackage{graphicx}
\usepackage{wrapfig}
\usepackage{bbm}
\usepackage{color}
\usepackage{bm}
\usepackage[top=1in, bottom=1.25in, left=1.10in, right=1.10in]{geometry}
\usepackage{todonotes}
\usepackage{stackengine}
\allowdisplaybreaks

\newcommand\ocirc[1]{\ensurestackMath{\stackon[1pt]{#1}{\mkern2mu\circ}}}

\newtheorem{theorem}{Theorem}[section]
\newtheorem{lemma}[theorem]{Lemma}
\newtheorem{proposition}[theorem]{Proposition}
\newtheorem{corollary}[theorem]{Corollary}

\theoremstyle{definition}
\newtheorem{definition}[theorem]{Definition}

\theoremstyle{remark}
\newtheorem{remark}[theorem]{Remark}

\numberwithin{equation}{section}

\newcommand{\bx}{\mathbf{x}}
\newcommand{\by}{\mathbf{y}}

\newcommand{\bq}{\mathbf{q}}
\newcommand{\bu}{\mathbf{u}}
\newcommand{\bfu}{\mathbf{u}}

\newcommand{\bff}{\mathbf{f}}

\newcommand{\bfn}{\mathbf{n}}
\newcommand{\vu}{\mathbf{u}}

\newcommand{\bn}{\mathbf{n}}

\newcommand{\dy}{\, \mathrm{d}\mathbf{y}}
\newcommand{\dd}{\,\mathrm{d}}
\newcommand{\dq}{\, \mathrm{d} \mathbf{q}}
\newcommand{\dx}{\, \mathrm{d} \mathbf{x}}

\newcommand{\dt}{\, \mathrm{d}t}

\newcommand{\Div}{\mathrm{div}_{\mathbf{x}}}
\newcommand{\divx}{\mathrm{div}_{\mathbf{x}}}
\newcommand{\divq}{\mathrm{div}_{\mathbf{q}}}

\newcommand{\nabx}{\nabla_{\mathbf{x}}}
\newcommand{\naby}{\nabla_{\mathbf{y}}}
\newcommand{\nabq}{\nabla_{\mathbf{q}}}

\newcommand{\Delx}{\Delta_{\mathbf{x}}}
\newcommand{\Dely}{\Delta_{\mathbf{y}}}

\newcommand{\R}{\mathbb{R}}

\newcommand{\Oeta}{\Omega_{\eta(t)}}

\newcommand{\Ozeta}{\Omega_{\zeta}}

\begin{document}

\title[Local strong solution to the beam-polymeric fluid interaction system]
{Existence of a local strong solution to the beam-polymeric fluid interaction system}

\author{Dominic Breit and Prince Romeo Mensah}
\address{Institute of Mathematics, TU Clausthal, Erzstra\ss e 1, 38678 Clausthal-Zellerfeld, Germany}
\email{dominic.breit@tu-clausthal.de, prince.romeo.mensah@tu-clausthal.de}

\thanks{The authors would like to thank S. Schwarzacher  for valuable suggestions concerning the compability condition in fluid-structure interaction.}

%\author{Prince Romeo Mensah}
%\address{Department of Mathematics, Imperial College, London SW7 2AZ, United Kingdom}
%\email{p.mensah@imperial.ac.uk}

%    General info

\date{\today}



\begin{abstract}

  We construct  a unique local strong solution to the  finitely extensible nonlinear elastic (FENE) dumbbell model of Warner-type for an incompressible polymer fluid (described by the Navier--Stokes--Fokker--Planck equations)
interacting with a flexible elastic shell. The latter occupies the flexible boundary of the polymer fluid domain and is modeled by a beam equation coupled through kinematic boundary conditions and the balance of forces.  In the 2D case for the co-rotational Fokker-Planck model we obtain global-in-time strong solutions.

A main step in our approach is the proof of local well-posedness for just  the solvent-structure system in higher-order topologies which is  of independent interest. Different from most of the previous results in the literature, the reference spatial domain is an arbitrary smooth subset of $\mathbb{R}^3$, rather than a flat one. That is, we cover viscoelastic shells rather than  elastic plates.
 Our results also supplement the existing literature on the Navier--Stokes--Fokker--Planck equations posed on a fixed bounded domain. %\textcolor{green}{delete note: Periodic can be done exactly as in Masmoudi, whole space probably too}
\end{abstract}

\maketitle
%\tableofcontents

\section{Introduction}
\label{sec:intro}

\subsection{Motivation}
The mathematical theory of fluid-structure interactions has seen vast progress in the last two decades. This has largely  been motivated by the variety of applications ranging from hydroelasticity and aeroelasticity  to biomechanics and hemodynamics. Many results in the literature are concerned with the existence of solutions as well as the qualitative properties of the underlying systems of nonlinear partial differential equations (PDEs). See Section \ref{ref:biblio} below for an overview. Most of these results are focused on incompressible Newtonian fluids. Clearly, only simple fluids such as water can be realistically described in such a way. Complex fluids, on the other hand, require more complicated models. Nevertheless, it is also common to work with Newton's rheological law for the viscous stress tensor even in the context of complex fluids. 
A particular instance is hemodynamics where one studies the flow of blood in vessels, which deform elastically as a response. Blood has a very complex behaviour and the incompressible Navier--Stokes equations fail to capture all of it. In fact, there  only exists a few results on the mathematical analysis of non-Newtonian fluids (where Newton's rheological law is replaced by a nonlinear stress-strain relation) interacting with elastic structures, see \cite{HLN} and \cite{Len}. A different Ansatz to model the behaviour of complex fluids is to consider polymeric fluid models. Here, an additional stress tensor is obtained which describes the prolongation vector of polymer chains arising from a micro- or mesoscopic model, see the next subsection. The mathematical theory for such models (in fixed domains) is in a mature state (we give an overview of the literature below).
Although they arise naturally in many applications,
mathematical results concerning the interaction of a polymeric fluid with a flexible structure are virtually missing in the literature. The only result we are aware of is our previous paper \cite{breit2021incompressible} in which we construct a weak solution to a polymer fluid interacting with an elastic shell. We continue in this direction and construct a strong solution to the corresponding polymer fluid-structure problem which exists locally in time.

\subsection{The model}
We consider a solute-solvent-structure mutually coupled problem describing the interaction between a polymeric fluid and a flexible structure. Here, the polymeric fluid consists of a mixture of a solvent, say, water, and a solute made up of a pair of monomers linked by a finitely extensible nonlinear elastic (FENE) spring described by the FENE dumbbell model of Warner-type \cite{warner1972kinetic}.  More precisely, our system is described by the three-dimensional incompressible Navier--Stokes--Fokker--Planck system of equations defined on $I \times \Oeta \times B$ coupled with a two-dimensional viscous beam equation defined on $I \times \omega$. Here, $I:=(0,T)$ is the time interval, $ \Oeta\subset \mathbb{R}^3$ is the configuration of the moving spatial domain at a time $t\in I$ (which arises by deforming the reference domain $\Omega\subset\R^3$ in the normal direction with amplitude $\eta(t)$, see Section \ref{sec:prelims} for the set-up),  and the domain for the elongation vector $\bq=(q_1,q_2,q_3)$ of the monomer molecules is taken as the ball
 $B\subset \R^3$ centered at the origin with radius $\sqrt{b}$.
 Finally, $\omega$ represent $\partial\Omega$; the boundary of the reference domain $\Omega$. For technical simplification, we identify $\omega$ with the two-dimensional
torus. The normal vectors on $\partial\Omega$ and $\partial\Oeta$ are denoted by $\bn$ and $\bn_\eta$, respectively.

We wish to find the structure displacement $\eta:(t, \by)\in I \times \omega \mapsto   \eta(t,\by)\in \mathbb{R}$, the fluid's velocity field $\mathbf{u}:(t, \mathbf{x})\in I \times \Oeta \mapsto  \mathbf{u}(t, \mathbf{x}) \in \mathbb{R}^3$, the fluid's pressure $\pi:(t, \mathbf{x})\in I \times \Oeta \mapsto  \pi(t, \mathbf{x}) \in \mathbb{R}$ and the probability density function $f :(t, \mathbf{x}, \mathbf{q})\in I \times \Oeta \times B \mapsto f (t, \mathbf{x}, \mathbf{q}) \in [0,\infty)$
such that for $\widehat{f}:=f/M$, where  $M=M(\bq)>0$ is given in \eqref{maxwellianPartial} below, the equations
\begin{align}
\varrho_s\partial_t^2\eta -\gamma\partial_t\Dely \eta + \alpha\Dely^2\eta&=g-\bn^\top\mathbb{T}\circ\bm{\varphi}_\eta\bn_\eta  \det(\naby\bm{\varphi}_\eta) ,
\label{shellEq}
\\
\varrho_f\big(\partial_t \bu  + (\mathbf{u}\cdot \nabx)\mathbf{u} \big)
&= 
\mu\Delx \bu -\nabx\pi+ \bff+
\divx   \mathbb{S}_\bq(\widehat{f}),
\label{momEq}
\\
\divx \bu&=0,
\label{contEq}
\\
M\big(\partial_t \widehat{f} + (\mathbf{u}\cdot \nabx) \widehat{f})
+
 \divq  \big( (\nabx \bu) \bq M\widehat{f} \big) 
&=
\varepsilon\Delx(M \widehat{f})
+
\kappa\,   \divq  \big( M \nabq  \widehat{f}
\big)
\label{fokkerPlank}
\end{align}
are satisfied almost everywhere in $I \times \Oeta \times B$. Here, the tensor field $\mathbb{T}$ is given by
\begin{align}
\mathbb{T}:=\mathbb{T}(\bu, \pi,\widehat{f})=\mu(\nabx\bu+\nabx\bu^\top)-\pi\mathbb{I}_{3\times3}+\mathbb{S}_\bq(\widehat{f})
\end{align}
and $\varrho_s, \gamma, \alpha, \varrho_f, \mu, \varepsilon,\kappa$ are positive   parameters all of which we henceforth set to $1$ for simplicity. There are two external forcing terms given by $\bff:(t, \mathbf{x})\in I \times \Oeta \mapsto  \bff(t, \mathbf{x}) \in \mathbb{R}^3$ and $g:(t, \by)\in I \times \omega \mapsto  g(t,\by)\in   \mathbb{R}$.
The elastic stress tensor $\mathbb{S}_\bq$ in the momentum equation \eqref{momEq} is given by
\begin{align} 
\label{extraStreeTensor}
\mathbb{S}_\bq(\widehat{f} ) 
=
 \int_B M(\bq) \widehat{f} (t, \mathbf{x},\mathbf{q})\nabq U(\tfrac{1}{2}\vert\mathbf{q}\vert^2 ) \otimes\mathbf{q} \dq, 
%= 
% \int_B  M(\bq) \widehat{f}  (t, \mathbf{x},\mathbf{q})\, \mathbf{F}(\mathbf{q}) \otimes \mathbf{q} \dq.
\end{align}
where the spring potential $U$ is given by
\begin{align}
\label{elasticSpringForce}
%\mathbf{F} (\mathbf{q}) = \nabla_{\bq} U\Big(\frac{1}{2}\vert \mathbf{q} \vert^2 \Big) = U'\Big(\frac{1}{2}\vert \mathbf{q} \vert^2 \Big) \mathbf{q}, \qquad
U(s )
=-\frac{b}{2}\ln\left(1-\frac{2s}{b}\right),
\qquad
s\in[0,b/2), \quad b>2.
\end{align}
The spring potential is also related to the associated  Maxwellian $M$ via the relation
\begin{align}
\label{maxwellianPartial}
M(\mathbf{q}) = \frac{e^{-U(\tfrac{1}{2}\vert\mathbf{q}\vert^2 ) }}{\int_{B}e^{-U (\tfrac{1}{2}\vert\mathbf{q}\vert^2 )}\,\mathrm{d}\bq}.
\end{align}
%Note that
%\begin{align}
%   \divq  \big( M \nabq (f/M)
%\big)
%=
%\Delta_\bq f + \divq(f\nabq U).
%\end{align}
%Also, if we set
%\begin{align*}
%f(t,\bx,\bq)=e^{-\frac{1}{2}U(\frac{1}{2}\vert \bq \vert^2)}g(t,\bx,\bq),
%\end{align*}
%then $g$ satisfies
%\begin{equation}
%\begin{aligned}
%\partial_t g + (\mathbf{u}\cdot \nabx) g
%+
% \divq  \big( (\nabx   \mathbf{u}) \bq g \big) 
%&=
%\varepsilon \Delx g +\kappa \Delta_\bq g
%+
%\frac{1}{2}(\nabq U)\cdot((\nabx \bu)\bq g)
% \\&
% +
%\kappa\frac{1}{2} \bigg( \Delta_\bq U -\frac{1}{2} \vert \nabq U \vert^2\bigg)g.
%\label{fokkerPlankg}
%\end{aligned}
%\end{equation}
%
%\subsection{Initial and Boundary conditions}
%\label{sec:initialBoundaryConds}
The initial conditions for the  polymer fluid  are
\begin{align}
&\eta(0, \cdot) = \eta_0, \quad \partial_t\eta(0, \cdot) = \eta_\star
&\quad \text{in }  \omega,
\label{initialStructure}
\\
&\mathbf{u}(0, \cdot) = \mathbf{u}_0
&\quad \text{in }  \Omega_{\eta_0},
\label{initialDensityVelo}
\\
&\widehat{f}(0, \cdot, \cdot) =\widehat{f}_0 \geq 0
& \quad \text{in }\Omega_{\eta_0} \times B
\label{fokkerPlankIintialx}
\end{align}
with given functions $\eta_0,\eta_\star:\by\in\omega\mapsto \eta_0(\by),\eta_\star(\by)\in\R$, $\mathbf u_0:\bx\in\Omega_{\eta_0} \mapsto \mathbf u_0(\bx)\in\R^3$ and $\widehat f_0:(\bx,\bq)\in\Omega_{\eta_0}\times B \mapsto \widehat{f}_0(\bx,\bq)\in[0,\infty)$.
With respect to boundary conditions, we supplement  the viscous beam equation \eqref{shellEq} with periodic boundary conditions and we impose  
\begin{align}
\label{noSlip}
&\bu  \circ \bm{\varphi}_{\eta} =(\partial_t\eta)\bn
&\quad \text{on }I \times \omega 
\end{align}
at the polymer fluid-structure interface with a  normal vector $\bn$ on $\partial\Omega$. Finally, for the solute, we have
%\todo{normal at $\partial B$ is $\bq/|\bq|$}
\begin{align}
&
\nabx\widehat{f}\cdot \bn_\eta =0
&\quad \text{on }I \times \partial\Omega_\eta \times B,
\label{fokkerPlankBoundaryxnEta}
\\
&M\big(\nabq\widehat{f}   -  (\nabx \bu) \bq \widehat{f}
 \big) \cdot \frac{\bq}{\vert\bq\vert} =0
&\quad \text{on }I \times \Omega_\eta \times \partial \overline{B}.
\label{fokkerPlankBoundaryx}
\end{align}
\noindent When the probability density function $f$ is identically zero, the system \eqref{contEq}--\eqref{fokkerPlank} reduces to a normal fluid-structure problem for an unsteady three-dimensional viscous incompressible fluid interacting with an elastic structure.  Let us point out that the reference spatial domain in our set-up is an arbitrary smooth subset of $\mathbb{R}^3$ (such as cylinders or spheres), rather than a flat one. That is, we cover viscoelastic shells rather than simple elastic plates.

\subsection{Bibliographical overview}
\label{ref:biblio}
We may broadly classify analytic works on fluid-structure interaction problems into the construction of weak and strong solutions. In this paper, we are only interested in strong solutions but let us refer to \cite{barbu2007existence, barbu2008smoothness, boulakia2007existence, chambolle2005existence,  coutand2005motion, grandmont2008existence, lengeler2014weak, muha2019existence} 
for some important works on the construction of weak solutions.
\\
When it comes to strong solutions, the short time existence and uniqueness of solutions in Sobolev spaces are studied in \cite{cheng2007navier, cheng2010interaction} for a viscous incompressible fluid interacting with a nonlinear thin elastic  shell. The shell equation, for the former \cite{cheng2007navier}, is modeled by the nonlinear Saint-Venant-Kirchhoff constitutive law whereas that of the latter \cite{cheng2010interaction}  is modeled by the nonlinear Koiter shell model.  In \cite{coutand2006interaction}, however, the authors prove the existence of a unique local strong solution, without restriction on the size of the data, when the elastic structure is now governed by quasilinear elastodynamics. 
%A simplified model for describing blood flow through viscoelastic arteries is explored in \cite{GraHil}. Here a global-in-time strong solution is constructed for the coupling of a one-dimensional viscoelastic beam equation and the two-dimensional incompressible Navier--Stokes equation. 
%In particular, the authors show that contact between the viscoelastic wall and the bottom of the fluid cavity does not occur in finite time.
In \cite{ignatova2014well}, the elastic structure is modeled by a damped wave equation  with  additional boundary stabilization terms. For sufficiently small initial data, subject to said boundary stabilization terms,
global-in-time existence of   strong solutions and exponential decay of the solutions are shown. 
The free boundary fluid-structure interaction problem consisting of a Navier--Stokes equation and a wave equation defined in two different but adjacent domains is studied in \cite{kukavica2012solutions}. A local strong solution is constructed under suitable compatibility conditions for the data.
Another local-in-time strong existence result is \cite{maity2020maximal}, where the viscous Newtonian fluid is now interacting with an elastic structure modeled by a nonlinear damped shell equation.
Finally, a local strong solution is constructed for the motion of a linearly elastic Lam\'e solid moving in a viscous fluid in \cite{raymond2014fluid}.
\\
For a fixed geometry and an identically zero solution of the structure equation, the system \eqref{contEq}--\eqref{fokkerPlank} reduces to an incompressible Navier--Stokes--Fokker--Planck system  for a polymeric fluid with center-of-mass diffusion. Weak solutions for such a system have been studied in,  for example,  \cite{barrett2005existence, barrett2007existence, barrett2008existence, barrettSuli2011existence, barrett2012existenceMMMAS, barrett2012existenceJDE, el1989existence, gwiazda2018existence, lukacova2017global}. 
%When we further assume that the solution of the Navier--Stokes equation is identically zero, the authors in \cite{el1989existence} also construct solutions that are independent of the Deborah number (a function of $\kappa>0$ in \eqref{fokkerPlank}). 
\\
For strong solutions, however, a unique local-in-time strong solution   was first shown to exist in \cite{renardy1991existence}, which unfortunately  excludes the physically relevant FENE dumbbell models. The local theory was then revisited in \cite{jourdain2004existence} for the stochastic FENE model for the simple Couette flow
and in \cite{e2004well} where the authors analysed the  incompressible Navier--Stokes equation coupled with a system of SDEs describing the configuration of the spring. The corresponding deterministic system  was then studied in \cite{li2004local, zhang2006local}.  The existence of Lyapunov functionals and smooth solutions was shown to exist in \cite{constantin2005nonlinear}. Finally, a global-in-time strong solution for the 2D system is shown in \cite{masmoudi2008well}(see also \cite{constantin2007regularity} and \cite{lin20082d}).
\\
The only result on the fully coupled system  \eqref{contEq}--\eqref{fokkerPlank}
is our previous paper \cite{breit2021incompressible}, in which we prove the existence of a weak solution (allowing even the fully nonlinear original Koiter model for the structure displacement).

\subsection{Main result and novelties}
\label{sec:nov}
Our main result is the existence of a unique local-in-time strong solution to \eqref{shellEq}--\eqref{fokkerPlank}. The precise statement can be found in Theorem \ref{thm:main}. The proof consists of constructing a fixed point of the following solution map in a suitable topology (see Section  \ref{sec:fullyCoupled} for details): Given a probability density function $\widehat f$, we solve the solvent-structure problem \eqref{shellEq}--\eqref{contEq} leading to a solution $(\eta,\bu,\pi)$. Eventually, we solve the Fokker-Planck equation \eqref{fokkerPlank} in a given moving domain $\Omega_\eta$ with a given velocity field $\bu$ yielding a solution $\widehat h$. Then we consider the map $$\mathtt{T}:X\rightarrow X,\quad \widehat f\mapsto \widehat h,$$
in (a subset of) a function space $X$. Such a strategy is also applied in \cite{masmoudi2008well} and other papers and it turns out that the velocity field needs to belong at least to $W^{1,\infty}$ with respect to the spatial variable to close the argument.
In \cite{masmoudi2008well}, the Navier--Stokes--Fokker--Planck system
is considered without centre-of-mass diffusion (in a fixed domain). At first glance, it will seem easier to do the same in the case $\varepsilon>0$ (leaving the difficulty of a moving boundary beside for the moment). This is certainly true if the Navier--Stokes--Fokker--Planck equations are studied on the whole space or with respect to periodic boundary conditions. However, in the case of bounded domains, where \eqref{fokkerPlank} must be complemented with Neumann boundary conditions as in \eqref{fokkerPlankBoundaryxnEta}--\eqref{fokkerPlankBoundaryx}, it is not clear if one can obtain higher order spatial derivatives for the probability density function even for smooth velocity fields. As a consequence, a result similar to \cite{masmoudi2008well} for the problem with centre-of-mass diffusion and a non-trivial boundary does not seem to exist in the literature.
%\footnote{Related results for the problem on the whole space are obtained in \cite{Ya} and \cite{La} and we studied the compressible Navier--Stokes--Fokker-Planck equations with periodic boundary conditions in \cite{breit2021local}.} 
As a by-product of our theory, we close this gap via the following idea: We first differentiate \eqref{fokkerPlank}  once in space (formerly testing by $\Delx \widehat f$). As just explained, this is not sufficient  to close the fixed point argument but does not create problems with the boundary conditions either. Eventually, we differentiate in time and obtain the same estimate
for the time derivative of the probability density function. Details can be found in Section \ref{sec:FP}, where for a given velocity field and moving geometry, we construct a strong solution to the Fokker--Planck equation. Here, due to the linear structure of the equation, we rely on an approximation procedure similar to \cite{breit2021local, masmoudi2008well}. The analysis here is, however, more complicated due to the flexible nature of the given geometry. 

 Let us now comment on the fluid-structure system  \eqref{shellEq}--\eqref{contEq}. Its solvability, for a given $\widehat f$, is an intermediate step for the fixed point problem just described but it is also of independent interest. 
It is worth pointing out that, different from most of the previous results on strong solutions such as \cite{cheng2010interaction}, \cite{GHL}, \cite{Leq}, \cite{maity2020maximal} and \cite{schwarzacherSu}, we consider a general non-flat geometry. The first results in this direction were only provided  very recently in  \cite{breit2022regularity}, where the existence of a unique global-in-time strong solution was proved in the 2D case.
 The existence of a local strong solution to the 3D fluid-structure problem has been recently shown in \cite{BMSS}. However, this strong solution is not regular enough to couple the solvent-structure system with the Fokker--Planck equation. For this reason we devote Section \ref{sec:solveSolventStructure} to obtaining higher space-time regularity for the strong solution constructed in \cite{BMSS}  by way of a fixed-point argument. 
This is of independent interest, and it is the first result of the higher regularity of the strong solution to the incompressible fluid-structure problem  in the case of shells.
Although one might expect that taking higher order derivatives will be easy, the problem of compatibility conditions of the data occurs (typical for parabolic equations in bounded domains, see the classical works \cite{So1,So2}). Such a condition is needed to control the initial pressure (see the proof of Proposition \ref{prop:bigData}), a problem that is absent in  \cite{BMSS}.

In two dimensions, if the co-rotational model is considered (i.e. $\nabx \bu$ is replaced by $\mathcal W(\bu)=\frac{1}{2}(\nabla\mathbf u-\nabla\mathbf u^\top)$ in the drag term in \eqref{fokkerPlank}), we prove the existence of a unique global strong solution, cf. Theorem \ref{thm:main2D}. It is a consequence of a novel estimate for the Fokker-Planck equation derived in 
Section \ref{sec:FP} combined with the recent results from \cite{breit2022regularity}
on the fluid-structure problem.
Again, the result for the Navier--Stokes--Fokker--Planck system seems new even for fixed domains (the counterpart without centre-of-mass diffusion in the Fokker-Planck equation is given in \cite{masmoudi2008well}).

\section{Preliminaries and main results}
\label{sec:prelims}
\noindent Without loss of generality, henceforth, we set all the parameters ($\rho_s$, \ldots, $\kappa$) in \eqref{shellEq}--\eqref{fokkerPlankBoundaryx} to 1. For two non-negative quantities $F$ and $G$, we write $F \lesssim G$  if there is a $c>0$  such that $F \leq c\,G$. If $F \lesssim G$ and $G\lesssim F$ both hold, we use the notation $F\sim G$. The symbol $\vert \cdot \vert$ may be used in four different contexts. For a scalar function $f\in \mathbb{R}$, $\vert f\vert$ denotes the absolute value of $f$. For a vector $\bff\in \mathbb{R}^d$, $\vert \bff \vert$ denotes the Euclidean norm of $\bff$. For a square matrix $\mathbb{F}\in \mathbb{R}^{d\times d}$, $\vert \mathbb{F} \vert$ shall denote the Frobenius norm $\sqrt{\mathrm{trace}(\mathbb{F}^T\mathbb{F})}$. Finally, if $S\subseteq  \mathbb{R}^d$ is  a (sub)set, then $\vert S \vert$ is the $d$-dimensional Lebesgue measure of $S$.
\\ 
The spatial domain $\Omega$ is assumed to be an open bounded subset of $\mathbb{R}^3$, with a smooth boundary and an outer unit normal ${\bfn}$. We assume that
 $\partial\Omega$ can be parametrised by an injective mapping ${\bm{\varphi}}\in C^k(\omega;\R^3)$ for some sufficiently large $k\in\mathbb N$. We suppose for all points $\by=(y_1,y_2)\in \omega$ that the pair of vectors  
%$\mathbf{a}_i(\by):= 
$\partial_i {\bm{\varphi}}(\by)$, $i=1,2,$ are linearly independent.
 For a point $\bx$ in the neighbourhood
of $\partial\Omega$ we can define the functions $\by$ and $s$ by
\begin{align*}
 \by(\bx)=\arg\min_{\by\in\omega}|\bx-\bm{\varphi}(\by)|,\quad s(\bx)=(\bx-\by(\bx))\cdot\bfn(\by(\bx)).
 \end{align*}
Moreover, we define the projection $\mathbf p(\bx)=\bm{\varphi}(\by(\bx))$. We define $L>0$ to be the largest number such that $s,\by$ and $\mathbf p$ are well-defined on $S_L$, where
\begin{align}
\label{eq:boundary1}
S_L=\{\bx\in\R^3:\,\mathrm{dist}(\bx,\partial\Omega)<L\}.
\end{align}
Due to the smoothness of $\partial\Omega$ for $L$ small enough we have $|s(\bx)|=\min_{\by\in\omega}|\bx-\bm{\varphi}(\by)|$ for all $\bx\in S_L$. This implies that $S_L=\{s\bfn(\by)+\by:(s,\by)\in (-L,L)\times \omega\}$.
%
%Moreover, it has been shown, that if $\overline{\gamma}(\eta_0)\neq 0$, then there exists a $L>0$ such that for all $\eta\in H^{2,2}(\omega)
%\begin{align}
%\label{eq:boundary2}
%S_L=\{x\in\R^3:\,\mathrm{dist}(x,\partial\Omega)<L\}.
%\end{align}
For a given function $\eta : I \times \omega \rightarrow\R$ we parametrise the deformed boundary by 
\begin{align}\label{eq:0108}
{\bm{\varphi}}_\eta(t,\by)={\bm{\varphi}}(\by) + \eta(t,\by){\bfn}(\by), \quad \,\by \in \omega,\,t\in I.
\end{align}
%By possibly decreasing $L$, one easily deduces from this formula that $\Omega_{\eta}$ does not degenerate, that is\footnote{\textcolor{blue}{first condition not needed?}}
%\begin{align}\label{eq:1705}
%\begin{aligned}
%\partial_1\bm{\varphi}_\eta\times\partial_2\bm{\varphi}_\eta(t,y)\neq0,\quad
% \bfn(y)\cdot\bfn_{\eta(t)}(y)&>0,\quad \,y \in \omega,\,t\in I,
% \end{aligned}
%\end{align}
%provided $\|\eta\|_{L^\infty_{t,x}}<L$. Here $\bfn_{\eta(t)}$
%is the normal of the domain $\Omega_{\eta(t)}$
% defined through
%\begin{align}\label{eq:2612}
%\partial\Omega_{\eta(t)}=\{{\bm{\varphi}}(y) + \eta(t,y){\bfn}(y):y\in \omega\}.
%\end{align}
With some abuse of notation, we define the deformed space-time cylinder as $$I\times\Omega_\eta=\bigcup_{t\in I}\{t\}\times\Omega_{\eta(t)}\subset\R^{4}.$$
The corresponding function spaces for variable domains are defined as follows.
\begin{definition}{(Function spaces)}
Let $M\in C(B)$ be the Maxwellian \eqref{maxwellianPartial}. For $1\leq q\leq \infty$, we denote by
\begin{align*}
&L^q_M(B)=\{f\in L^q_{\mathrm{loc}}(B)\,:\, \Vert f\Vert_{L^q_M(B)}^q<\infty\}, 
\\&
W^{1,q}_M(B)=\{ f\in W^{1,q}_{\mathrm{loc}}(B)\,:\, \Vert f \Vert_{W^{1,q}_M(B)}^q<\infty \},
\end{align*}
the Maxwellian-weighted $L^q$ and $W^{1,q}$ spaces over $B$ with respective norms
\begin{align*}
\Vert f\Vert_{L^q_M(B)}^q:=\int_BM(\bq)\vert f(\bq)\vert^q\dq, \qquad
\Vert f\Vert_{W^{1,q}_M(B)}^q:=\int_BM(\bq)\big(\vert f(\bq)\vert^q+ \vert \nabq f(\bq)\vert^q\big)\dq.
\end{align*}
For $I=(0,T)$, $T>0$, and $\eta\in C(\overline{I}\times\omega)$ with $\|\eta\|_{L^\infty(I\times\omega)}< L$, we define for $1\leq p,r\leq\infty$,
\begin{align*}
L^p(I;L^r(\Omega_\eta))&:=\Big\{v\in L^1(I\times\Omega_\eta):\substack{v(t,\cdot)\in L^r(\Omega_{\eta(t)})\,\,\text{for a.e. }t,\\\|v(t,\cdot)\|_{L^r(\Omega_{\eta(t)})}\in L^p(I)}\Big\},\\
L^p(I;W^{1,r}(\Omega_\eta))&:=\big\{v\in L^p(I;L^r(\Omega_\eta)):\,\,\nabx v\in L^p(I;L^r(\Omega_\eta))\big\}.
\end{align*}
\end{definition}
Higher order Sobolev spaces can be defined accordingly. For $k>0$ with $k\notin\mathbb N$, we define the fractional Sobolev space $L^p(I;W^{k,r}(\Oeta))$ as the class of $L^p(I;L^r(\Omega_\eta))$-functions $v$ for which the norm 
\begin{align*}
\|v\|_{L^p(I;W^{k,r}(\Oeta))}^p
&=\int_I\bigg(\int_{\Oeta} \vert v\vert^r\dx
+\int_{\Oeta}\int_{\Oeta}\frac{|v(\bx)-v(\bx')|^r}{|\bx-\bx'|^{3+k r}}\dx\dx'\bigg)^{\frac{p}{r}}\dt
\end{align*}
is finite. Accordingly, we can also introduce fractional differentiability in time for the spaces on moving domains.
\noindent When we combine the function spaces defined on $B$ and on space-time, we obtain spaces of the form
\begin{align*}
L^p(I;W^{k,r}(\Omega_\eta;W^{l,q}(B)))\quad k\geq0,\quad
1\leq p,r,q\leq \infty, \quad l\in\{0,1\}
\end{align*}
and more generally
\begin{align*}
W^{s,p}(I;W^{k,r}(\Omega_\eta;W^{l,q}(B)))\quad s,k\geq0,\quad 1\leq p,r,q\leq \infty, \quad l\in\{0,1\}.
\end{align*}
For various purposes, it is useful to relate the time-dependent domain and the fixed domain. This can be done by means of the Hanzawa transform. Its construction can be found in
\cite[pages 210, 211]{lengeler2014weak}. Note that variable domains in \cite{lengeler2014weak} are defined via functions $\zeta:\partial\Omega\rightarrow\R$ rather than functions $\eta:\omega\rightarrow\R$ (clearly, one can link them by setting
$\zeta=\eta\circ\bm{\varphi}^{-1}$). For any $\eta(t):\omega\rightarrow (-L,L)$ at time point $t\in I$, we let $\bm{\Psi}_{\eta(t)} : \Omega \rightarrow\Omega_{\eta(t)}$ be the Hanzawa transform   defined by 
\begin{equation}
\bm{\Psi}_{\eta(t)}(\bx)
=
 \left\{
  \begin{array}{lr}
    \mathbf p(\bx)+\big(s(\bx)+\eta(t,\by(\bx))\phi(\by(\bx))\big)\bn(\by(\bx)) &\text{if dist}(\bx,\partial\Omega)<L,\\
    \bx &\text{elsewhere}.
  \end{array}
\right.
\end{equation}
and with inverse $\bm{\Psi}^{-1}_{\eta(t)} : \Omega_{\eta(t)} \rightarrow\Omega$. Here, $\phi\in C^\infty(\mathbb{R})$ is such that $\phi\equiv0$ in a neighbourhood of $-L$ and $\phi\equiv1$ in a neighbourhood of $1$. %Also, $s(\bx)=(\bx-\by(\bx))\cdot \bn(\by(\bx)))$ where $\bn$ is the outer unit normal to the boundary of the reference domain $\omega$.
It is shown in \cite{breit2022regularity} that if for some $\alpha,R>0$, we assume that
\begin{align*}
\Vert\eta(t)\Vert_{L^\infty_\by}
+
\Vert\zeta(t)\Vert_{L^\infty_\by}
< \alpha <L \qquad\text{and}\qquad
\Vert\naby\eta(t)\Vert_{L^\infty_\by}
+
\Vert\naby\zeta(t)\Vert_{L^\infty_\by}
<R
\end{align*}
holds, then for any  $s>0$, $p\in[1,\infty]$, $k\in\{0,1,2\}$ and for any $\eta,\zeta \in W^{k,1}(I;W^{s,p}(\omega))$, we have that
\begin{align}
\label{210and212}
&\Vert \partial_t^k\bm{\Psi}_\eta \Vert_{W^{s,p}_\bx}
+
\Vert \partial_t^k\bm{\Psi}_\eta^{-1} \Vert_{W^{s,p}_\bx}
 \lesssim
1+ \Vert \partial_t^k\eta \Vert_{W^{s,p}_\by},
\\
\label{211and213}
&\Vert \partial_t^k(\bm{\Psi}_\eta - \bm{\Psi}_\zeta)  \Vert_{W^{s,p}_\bx} 
+
\Vert \partial_t^k(\bm{\Psi}_\eta^{-1} - \bm{\Psi}_\zeta^{-1}  )\Vert_{W^{s,p}_\bx} 
\lesssim
 \Vert \partial_t^k(\eta - \zeta) \Vert_{W^{s,p}_\by}
\end{align}
%and
%\begin{align}
%\label{218}
%&\Vert \partial_t\bm{\Psi}_\eta \Vert_{W^{s,p}_\bx}
%\lesssim
%1+ \Vert \partial_t\eta \Vert_{W^{s,p}_\by},
%\qquad
%\eta
%\in W^{1,1}(I;W^{s,p}(\omega))
%\end{align}
holds uniformly in time with the hidden constants depending only on the reference geometry, on $L-\alpha$ and $R$. The estimate \eqref{210and212} holds without the $1$ on the right-hand side when in addition, $k\neq0$.
\\
Our interest is to construct a strong and regular solution to the system \eqref{shellEq}--\eqref{fokkerPlank} (i.e. a solution that satisfies \eqref{shellEq}--\eqref{fokkerPlank} pointwise almost everywhere in spacetime with additional regularity properties which will soon be made precise) emanating from the initial conditions \eqref{initialStructure}--\eqref{fokkerPlankIintialx}. 
To make the notion of \textit{a strong solution} precise, we first present the following notion of a \textit{a weak solution}.
\begin{definition}[Weak solution] \label{def:weakSolution}
Let $(\bff, g, \eta_0, \eta_\star, \bu_0, \widehat{f}_0)$ be a dataset such that
\begin{equation}
\begin{aligned}
\label{dataset} 
&\bff \in L^2\big(I; L^2_{\mathrm{loc}}(\mathbb{R}^3 )\big),\quad
g \in L^2\big(I; L^{2}(\omega)\big), \quad
\eta_0 \in W^{2,2}(\omega) \text{ with } \Vert \eta_0 \Vert_{L^\infty( \omega)} < L,
\\
&\eta_\star \in L^{2}(\omega), \quad  \widehat{f}_0\in L^2\big( \Omega_{\eta_0} ;L^2_M(B)\big), \quad
\bu_0\in L^{2}_{\mathrm{\divx}}(\Omega_{\eta_0} ) \text{ is such that }\bu_0 \circ \bm{\varphi}_{\eta_0} =\eta_\star \bn \text{ on } \omega.
\end{aligned}
\end{equation} 
We call  
$( \eta, \bu,  \widehat{f} )$
a \textit{weak solution} to the system \eqref{shellEq}--\eqref{fokkerPlank} with data $(\bff, g, \eta_0, \eta_\star, \bu_0, \widehat{f}_0)$ provided that the following holds:
\begin{itemize}
\item[(a)]  the shell displacement $\eta$ satisfies $\Vert \eta \Vert_{L^\infty(I \times \omega)} <L$ and
\begin{align*}
\eta \in W^{1,\infty} \big(I; L^{2}(\omega) \big)\cap  L^\infty \big(I; W^{2,2}(\omega) \big) \cap W^{1,2}\big(I;W^{1,2}(\omega) \big);% \cap W^{2,2}\big(I;L^2(\omega) \big)  ;
\end{align*}
\item[(b)] the velocity $\bu$ is such that $\bu  \circ \bm{\varphi}_{\eta} =(\partial_t\eta)\bn$ on $I\times \omega$ and
\begin{align*}
 \vu \in L^{\infty} \big(I; L^2_{\divx}(\Omega_{\eta(t)} ) \big)\cap  L^2 \big(I; W^{1,2}(\Omega_{\eta(t)} \big)
;
\end{align*}
\item[(c)] the probability density function $ \widehat{f}$ satisfies 
\begin{align*}
%&\widehat{f}\geq 0 \text{ a.e. in }  I \times \Omega_{\eta(t)}\times B,
\widehat{f} \in  L^\infty \big( I; L^2(\Omega_{\eta(t)}; L^2_M(B))\big)
\cap  L^2 \big( I; W^{1,2}(\Omega_{\eta(t)} ; L^2_M(B))\big)
\cap  L^2 \big( I; L^2(\Omega_{\eta(t)} ; H^1_M(B))\big);
\end{align*}
\item[(d)]  for all  $(\phi, \bm{\phi}, \varphi ) \in C^\infty(\overline{I}\times\omega) \times C^\infty(\overline{I}; C^\infty_{\divx}(\R^3))
\times C^\infty (\overline{I}\times \R^3 \times \overline{B} )$ with $\phi(T,\cdot)=0$, $\bm{\phi}(T,\cdot)=0$  and $\bm{\phi}\circ \bm{\varphi}_{\eta}= \phi\bn$, we have
\begin{align*}
\int_I  \frac{\mathrm{d}}{\dt}\bigg(\int_\omega \partial_t \eta \, \phi \dy
&+
\int_{\Oeta}
\bu  \cdot \bm{\phi}\dx
+
\int_{\Omega_{\eta(t)} \times B}M \widehat{f} \, \varphi \dq \dx
\bigg)\dt 
\\
&=
\int_I \int_\omega \big(\partial_t \eta\, \partial_t\phi
-
\partial_t \naby \eta\cdot \naby\phi
+
 g\, \phi
 - 
 \Dely \eta\, \Dely\phi \big)\dy\dt
 \\&+
 \int_I  \int_{\Oeta}\big(  \vu\cdot \partial_t  \bm{\phi} + \vu \otimes\vu: \nabx \bm{\phi} 
  \big) \dx\dt
\\&
-\int_I  \int_{\Oeta}\big(   
 \nabx \bu:\nabx \bm{\phi}  + \mathbb{S}_\bq(\widehat{f} )  :\nabx \bm{\phi}-\bff\cdot\bm{\phi} \big) \dx\dt
 \\&+\int_I\int_{\Omega_{\eta(t)} \times B}\big(M \widehat{f} \,\partial_t \varphi 
+
M\bu  \widehat{f} \cdot \nabx \varphi \big) \dq \dx \dt
\\&
+ \int_I\int_{ \Omega_{\eta(t)} \times B}
 \big( M(\nabx \bu)  \bq\widehat{f}-
   M \nabq   \widehat{f} \big) \cdot \nabq\varphi \dq \dx \dt,
\end{align*}
with $\eta(0)=\eta_0$ and $\partial_t\eta=\eta_\star$ a.e. in $\omega$, $\bfu(0)=\bfu_0$ a.e. in $\Omega_{\eta_0}$ as well as $\widehat f(0)=\widehat f_0$ a.e. in $\Omega_{\eta_0}\times B$.
\end{itemize}
\end{definition}
\noindent The existence of a weak solution \eqref{shellEq}--\eqref{fokkerPlank} in the sense of Definition \ref{def:weakSolution} is shown in \cite{breit2021incompressible}.\footnote{Note that there is no dissipation in the shell equation in \cite{breit2021incompressible}. It makes, however, the analysis easier and can thus be incorporated without any problems.}
For this solution to be regular, we impose below, additional regularity assumptions on the initial conditions and the forcing terms in \eqref{shellEq}--\eqref{fokkerPlank}. More precisely, we suppose that the dataset
$(\bff, g, \eta_0, \eta_\star, \bu_0, \widehat{f}_0)$
satisfies
\begin{equation}
\begin{aligned}
\label{mainDataForAll}
&\bff \in  W^{1,2}(I;L^{2}(\Oeta)) ,\quad  \bff(0)\in W^{1,2}(\Omega_{\eta_0}),
\\
&g\in  W^{1,2}(I;W^{1,2}(\omega)) ,
%\quad  g(0)\in W^{1,2}(\omega)
\\&\eta_0 \in W^{5,2}(\omega) \text{ with } \Vert \eta_0 \Vert_{L^\infty( \omega)} < L, \quad \eta_\star \in W^{3,2}(\omega),
\\&\bu_0 \in W^{3,2}_{\divx}(\Omega_{\eta_0} )\text{ is such that }\bu_0 \circ \bm{\varphi}_{\eta_0} =\eta_\star \bn \text{ on } \omega,
\\&
\widehat{f}_0\in W^{1,2}(\Omega_{\eta_0};L^2_M(B))    
\end{aligned}
\end{equation}
with the compatibility condition 
\begin{equation}
\begin{aligned}
\label{compatibilityConditionMain}
\big[\Dely  \eta_\star &-  \Dely^2 \eta_0
+
g(0) 
-\bn^\top\mathbb{T}(0)\circ\bm{\varphi}_{\eta_0}\bn_{\eta_0} \det(\naby\bm{\varphi}_{\eta_0}) \big]\bn
\\&=
\big[ \Delx \bu_0 -\nabx\pi_0+ \bff(0)+
\divx   \mathbb{S}_\bq(\widehat{f}_0)  \big] \circ\bm{\varphi}_{\eta_0}
\end{aligned}
\end{equation}
on $\omega$, where $\pi_0$ is the solution to
%\begin{align*}
%\begin{cases}
%\Delx \tilde \pi_\star=\Delta_{\Omega_\eta}^{-1}\Div((\bfu\cdot\nabx)\bfu -\Delx\bu)&\text{ in }\Omega_\eta,\\
%\nabx\tilde\pi_\star\cdot\bfn_\eta\circ\bm{\varphi}_\eta^{-1}=(\bfu\cdot\nabx)\bfu -\Delx\bu\big)\bfn_{\eta}\circ\bm{\varphi}_\eta^{-1}+\partial_t^2\eta\circ\bm{\varphi}_\eta^{-1}&\text{ on }\partial\Omega_\eta,
%\end{cases}
%\end{align*} 
\begin{align*}
\begin{cases}
\Delx  \pi_0
=
\Div(\Delx\bu_0 +\bff(0)+\divx\mathbb{S}_\bq(\widehat{f}_0) -(\bfu_0\cdot\nabx)\bfu_0 )&\text{ in } \Omega_{\eta_0},
\\
\nabla
\pi_0\cdot\bn\circ\bm{\varphi}_\eta^{-1}
+
\pi_0\bn^\top\circ\bm{\varphi}_{\eta_0}^{-1}\cdot\bfn_\eta\circ\bm{\varphi}_{\eta_0}^{-1}\det(\naby\bm{\varphi}_{\eta_0})\circ\bm{\varphi}_{\eta_0}^{-1}
\\
\qquad=
[\Delx\bu_0+\bff(0)+\divx\mathbb{S}_\bq(\widehat{f}_0)]\cdot\bn\circ\bm{\varphi}_{\eta_0}^{-1} 
-
[\Dely\eta_\star-\Dely^2\eta_0+g(0)]\circ \bm{\varphi}_{\eta_0}^{-1}
\\
\qquad
+\bn^\top\circ\bm{\varphi}_{\eta_0}^{-1}(\nabx\bu_0+\nabx\bu^\top_0+\mathbb{S}_\bq(\widehat{f}_0)) \bfn_{\eta_0}\circ\bm{\varphi}_{\eta_0}^{-1}\det(\naby\bm{\varphi}_{\eta_0})\circ\bm{\varphi}_{\eta_0}^{-1} &\text{ on } \partial\Omega_{\eta_0}.
\end{cases}
\end{align*}
Note that \eqref{compatibilityConditionMain} is in line with the compatibility condition for the fluid-structure interaction problems studied in \cite{cheng2010interaction}.
% and in particular, in \cite{schwarzacherSu}.

As far as the Fokker-Planck equation is concerned, we require that the function $\widetilde{f}_0$ defined by 
\begin{equation}
\begin{aligned}
\label{compatibilityFokker}
M   \widetilde{f}_0 
&=
\Delx(M  \widehat{f}_0)
+
 \divq  \big( M \nabq    \widehat{f}_0
\big)
-
M  (\bu_0\cdot \nabx)   \widehat{f}_0
-
 \divq  \big(   (  \nabx\bu_0) \bq M\widehat{f}_0 \big) 
\end{aligned}
\end{equation} 
in $\Omega_{\eta_0}\times B$ is such that
\begin{align}
\label{tildeFinitial}
\widetilde{f}_0\in L^{2}(\Omega_{\eta_0};L^2_M(B)).
\end{align}
\noindent 
With this regularised dataset, we can now make precise, what we mean by a strong solution of \eqref{shellEq}--\eqref{fokkerPlank}. 
\begin{definition}[Strong solution] \label{def:strongSolution}
Let $(\bff, g, \eta_0, \eta_\star, \bu_0, \widehat{f}_0)$ be a dataset satisfying \eqref{mainDataForAll}-\eqref{tildeFinitial}.
We call $( \eta, \bu,  \pi, \widehat{f} )$ a \textit{strong solution} of \eqref{shellEq}--\eqref{fokkerPlank} with data $(\bff, g, \eta_0, \eta_\star, \bu_0, \widehat{f}_0)$ provided that:
\begin{enumerate}
\item[(a)] $( \eta, \bu,  \widehat{f} )$ is a weak solution of \eqref{shellEq}--\eqref{fokkerPlank} in the sense of Definition \ref{def:weakSolution};
\item[(b)] $( \eta, \bu,  \pi, \widehat{f} )$ satisfies
\begin{align*}
\eta&\in W^{1,\infty}\big(I_*;W^{3,2}(\omega)  \big) \cap W^{2,2}\big(I_*;W^{1,2}(\omega)  \big)
\cap W^{3,2}\big(I_*;L^2(\omega)  \big)
\\
&\qquad\qquad\cap W^{1,2}\big(I_*;W^{3,2}(\omega)  \big)\cap L^{\infty}\big(I_*;W^{4,2}(\omega)  \big)\cap L^{2}\big(I_*;W^{5,2}(\omega)  \big),\\
\bu &\in
W^{1,\infty} \big(I_*; W^{1,2}(\Oeta) \big)\cap  W^{2,2}\big(I_*;L^2(\Oeta)  \big)
\\
&\qquad\qquad\cap  W^{1,2}\big(I_*;W^{2,2}(\Oeta)  \big)
\cap L^2\big(I_*;W^{3,2}(\Oeta)  \big),
\\
\pi &\in
W^{1,2}\big(I_*;W^{1,2}(\Oeta)  \big)
\cap
L^2\big(I_*;W^{2,2}(\Oeta)  \big),
\\ 
\widehat{f} & 
\in   W^{1,\infty}\big(I_*;W^{1,2}(\Oeta;L^2_M(B))  \big)
\cap 
W^{1,2}\big(I_*;W^{2,2}(\Oeta;L^2_M(B))  \big)
\\
&\qquad\qquad\cap  W^{1,2}\big(I_*;W^{1,2}(\Oeta;H^1_M(B))  \big).
\end{align*}
\item[(c)] $(\bu,\pi)$ satisfies
%\todo{This is needed, otherwise the role of the pressure is not clear. It is not included in Def. \ref{def:weakSolution}}
$$\partial_t \bu  + (\mathbf{u}\cdot \nabx)\mathbf{u}
= \Delx \bu -\nabx\pi+ \bff+
\divx   \mathbb{S}_\bq(\widehat{f})$$
a.e. in $I\times\Omega_\eta$.
\end{enumerate}
\end{definition}
Our main result now reads as follows. 
\begin{theorem}\label{thm:main}
Let $(\bff, g, \eta_0, \eta_\star, \bu_0, \widehat{f}_0)$ be a dataset satisfying \eqref{mainDataForAll}--\eqref{tildeFinitial}. 
There is a time $T_*>0$ such that a unique  strong solution $( \eta, \bu,  \pi, \widehat{f} )$ of \eqref{shellEq}--\eqref{fokkerPlank}, in the sense of Definition \ref{def:strongSolution}, exists.
\end{theorem}
%\textcolor{red}{\begin{remark}
%The main result still holds true when the regularity of the forces $\bff \in L^2(I_*;W^{2,2}(\Oeta)) $ and $g\in L^2(I_*;W^{2,2}(\omega))$ are spatially reduced by one to $\bff \in L^2(I_*;W^{1,2}(\Oeta)) $ and $g\in L^2(I_*;W^{1,2}(\omega))$, respectively. In the case, $\bu\in L^2\big(I_*;W^{4,2}(\Oeta)  \big)$, $\pi\in L^2\big(I_*;W^{3,2}(\Oeta)  \big)$ and 
%\begin{align*}
%\eta\in W^{1,2}\big(I_*;W^{4,2}(\omega)  \big)\cap L^{\infty}\big(I_*;W^{5,2}(\omega)  \big)\cap L^{2}\big(I_*;W^{6,2}(\omega)  \big)
%\end{align*}
%accordingly reduces by one spatial order. All other regularities remain the same even with the reduction in the spatial regularity of the two forces.
%\end{remark}}
\begin{remark}
We note that the Fokker--Planck equation is conservative and its solution is an actual probability density function (meaning that $\widehat{f}\geq0$) within the flexible geometry under consideration. More precisely, we note that if we integrate \eqref{fokkerPlank} over $\Oeta\times B$ and use \eqref{noSlip}--\eqref{fokkerPlankBoundaryx} together with Reynold's transport theorem, we obtain
\begin{align*}
\frac{\dd}{\dt}\int_{\Oeta\times B}M\widehat{f}\dq\dx=0
\end{align*}
and thus it is conservative at all times.
Furthermore, the solution $\widehat{f}$ of \eqref{fokkerPlank} advected by the velocity field $\bu\in L^2(I;W^{1,\infty}(\Oeta))$ remains nonnegative if it were initially nonnegative. Indeed,  if we test \eqref{fokkerPlank} with the nonpositive part $\widehat{f}_-=\min\{0,\widehat{f}\}$ of $\widehat{f}$, integrate over $\Oeta\times B$ and use \eqref{noSlip}--\eqref{fokkerPlankBoundaryx} together with Reynold's transport theorem, we obtain
\begin{align*}
\frac{1}{2}\frac{\dd}{\dt}\int_{\Oeta\times B}M&\vert \widehat{f}_-\vert^2\dq\dx
+\int_{\Oeta\times B}M\vert \nabx \widehat{f}_-\vert^2\dq\dx
+\int_{\Oeta\times B}M\vert \nabq \widehat{f}_-\vert^2\dq\dx
\\&\leq
\frac{1}{2}\int_{\Oeta\times B}M\vert \nabq \widehat{f}_-\vert^2\dq\dx
+
c(\bq)
\frac{1}{2}\int_{\Oeta}\vert\nabx\bu\vert^2\int_{B}M\vert \widehat{f}_-\vert^2\dq\dx.
\end{align*} 
If we now apply Gr\"onwall's lemma, then for a nonnegative initial data $\widehat{f}_0\geq0$, it follows that
\begin{align*}
\frac{1}{2}\frac{\dd}{\dt}\int_{\Oeta\times B}M\vert \widehat{f}_-\vert^2\dq\dx
&+\int_{\Oeta\times B}M\vert \nabx \widehat{f}_-\vert^2\dq\dx
+\frac{1}{2}\int_{\Oeta\times B}M\vert \nabq \widehat{f}_-\vert^2\dq\dx
=0.
\end{align*} 
Therefore, $ \widehat{f}_-=0$ a.e. in $\Oeta\times B$ and thus, $ \widehat{f}=\widehat{f}_+=\max\{0,\widehat{f}\}$. See \cite{debiec2023corotational} for the corresponding argument for the fixed geometry.
\end{remark} 
\begin{remark}\label{rem:reg}
The reason for the choice of the topology in Definition \ref{def:strongSolution} comes from the coupling of the Navier--Stokes equations and the Fokker-Planck equation with centre-of-mass diffusion in a bounded domain: The velocity field must be Lipschitz in space, we cannot allow more than two spatial derivatives for the probability density function (this is related to \eqref{fokkerPlankBoundaryxnEta} and \eqref{fokkerPlankBoundaryx}). Moreover, we need additional temporal regularity to close the fixed point argument for the fully coupled system in Section \ref{sec:fullyCoupled}. As we have explained in Section \ref{sec:nov}
these difficulties are not related to the moving boundary and our result is even new for fixed boundaries (referring to $\eta_0=\eta_\star=g=0$ in \eqref{shellEq} and \eqref{initialStructure}).
\end{remark}
\begin{remark}
It remains open if a result similar to Theorem \ref{thm:main} holds if the Fokker-Planck equation \emph{without} centre-of-mass diffusion, that is $\varepsilon=0$ in \eqref{fokkerPlank}, is considered. This is related to the non-trivial boundary conditions for the fluid as well as the moving domain. For the fixed point argument for the fully coupled system in Section \ref{sec:fullyCoupled}, we need to prove a stability estimate for two different solutions of \eqref{fokkerPlank} being posed in two different moving domains. This requires to transform them to the reference domain, which eventually creates several boundary terms. They can only be controlled with the help of the additional regularity coming form the centre-of-mass diffusion.
\end{remark}

\section{Solving the Solvent-Structure problem}
\label{sec:solveSolventStructure}
\noindent In this section, we assume that a solution $\hbar$ for the equation of the solute described by the Fokker--Planck equation is known and that $ \hbar$ and its associated elastic stress tensor $ \mathbb{S}_\bq( \hbar)$ has sufficient regularity. For given body forces $\bff$ and $g$, our goal now is to construct a local-in-time strong solution of the solvent-structure coupled system given by 
\begin{align}
\partial_t^2\eta - \partial_t\Dely \eta + \Dely^2\eta&=g-\bn^\top\mathbb{T}\circ\bm{\varphi}_\eta\bn_\eta\det(\naby\bm{\varphi}_\eta) ,
\label{shellEqAlone}
\\
 \partial_t \bu  + (\mathbf{u}\cdot \nabx)\mathbf{u}  
&= 
\Delx \bu -\nabx\pi+ \bff+
\divx   \mathbb{S}_\bq( \hbar),
\label{momEqAlone}\\
\divx \bu&=0,
\label{contEqAlone}
\end{align}
where $\mathbb{T}$ is given by
\begin{align*}
\mathbb{T}:=\mathbb{T}(\bu, \pi,\hbar)=(\nabx\bu+\nabx\bu^\top)-\pi\mathbb{I}_{3\times3}+\mathbb{S}_\bq(\hbar).
\end{align*}
In the weak formulation, one considers a pair of test-functions  $(\phi, \bm{\phi}) \in C^\infty(\overline{I}\times\omega) \times C^\infty(\overline{I}; C^\infty_{\divx}(\R^3))
$ with $\phi(T,\cdot)=0$, $\bm{\phi}(T,\cdot)=0$  and $\bm{\phi}\circ \bm{\varphi}_{\eta}= \phi\bn$, obtaining
\begin{align*}
\int_I  \frac{\mathrm{d}}{\dt}\bigg(\int_\omega \partial_t \eta \, \phi \dy
&+
\int_{\Oeta}
\bu  \cdot \bm{\phi}\dx
\bigg)\dt 
\\
&=
\int_I \int_\omega \big(\partial_t \eta\, \partial_t\phi
-
\partial_t \naby \eta\cdot \naby\phi
+
 g\, \phi
 - 
 \Dely \eta\, \Dely\phi \big)\dy\dt
 \\&+
 \int_I  \int_{\Oeta}\big(  \vu\cdot \partial_t  \bm{\phi} + \vu \otimes \vu: \nabx \bm{\phi} 
  \big) \dx\dt
\\&
-\int_I  \int_{\Oeta}\big(   
\nabx \bu:\nabx \bm{\phi}  + \mathbb{S}_\bq( \hbar) :\nabx \bm{\phi}-\bff\cdot\bm{\phi} \big) \dx\dt.
\end{align*}
Note that this formulation is pressure-free.
 The pressure can be recovered by solving a.e. in time
%\begin{align*}
%\begin{cases}
%\Delx \tilde \pi_\star=\Delta_{\Omega_\eta}^{-1}\Div((\bfu\cdot\nabx)\bfu -\Delx\bu)&\text{ in }\Omega_\eta,\\
%\nabx\tilde\pi_\star\cdot\bfn_\eta\circ\bm{\varphi}_\eta^{-1}=(\bfu\cdot\nabx)\bfu -\Delx\bu\big)\bfn_{\eta}\circ\bm{\varphi}_\eta^{-1}+\partial_t^2\eta\circ\bm{\varphi}_\eta^{-1}&\text{ on }\partial\Omega_\eta,
%\end{cases}
%\end{align*} 
\begin{align*}
\begin{cases}
\Delx \tilde \pi_\star
=
\Div(\Delx\bu +\bff+\divx\mathbb{S}_\bq(\hbar) -(\bfu\cdot\nabx)\bfu )&\text{ in }I\times\Omega_\eta,
\\
\nabla\tilde
\pi_\star\cdot\bn\circ\bm{\varphi}_\eta^{-1}
+
\pi_\star\bn^\top\circ\bm{\varphi}_\eta^{-1}\cdot\bfn_\eta\circ\bm{\varphi}_\eta^{-1}\det(\naby\bm{\varphi}_\eta)\circ\bm{\varphi}_\eta^{-1}
\\
\qquad=
[\Delx\bu+\bff+\divx\mathbb{S}_\bq(\hbar)]\cdot\bn\circ\bm{\varphi}_\eta^{-1} 
-
[\partial_t\Dely\eta-\Dely^2\eta+g]\circ \bm{\varphi}_\eta^{-1}
\\
\qquad
+\bn^\top\circ\bm{\varphi}_\eta^{-1}(\nabx\bu+\nabx\bu^\top+\mathbb{S}_\bq(\hbar)) \bfn_\eta\circ\bm{\varphi}_\eta^{-1}\det(\naby\bm{\varphi}_\eta)\circ\bm{\varphi}_\eta^{-1} &\text{ on }I \times\partial\Omega_\eta,
\end{cases}
\end{align*}
and setting
$\pi_\star:=\tilde\pi_\star-(\tilde\pi_\star)_{\Omega_{\eta}}.$
 If $\Omega_\eta$ is Lipschitz uniformly in time (which follows from Definition \ref{def:strongSolutionAlone} (a) below) the solution operator to the Robin problem above has the usual properties, i.e. the solution belongs to $W^{1,2}$ if the right-hand side belongs to $W^{-1,2}$ and the boundary datum is in $W^{1/2,2}$.
 %\todo{what is meant by "usual properties"?}
%For $\mathcal O\subset\R^3$ open and bounded with normal $\bn_{\mathcal O}$ 
%we denote by $\Delta^{-1}_{\mathcal O}\Div$ the solution operator to the equation
%\begin{align*}
%\Delx h=\Div\bfg\quad\text{in}\quad\mathcal O,\quad \bfn_{\mathcal O}\cdot(\nabx h-\bfg)=0\quad\text{on}\quad\partial\mathcal O.
%\end{align*}
%where $\nu_{\mathcal O}$ denotes the normal of $\partial\mathcal O$.
%where $\Delta_\eta^{-1}$ is the solution operator to the Laplace equation on $\Omega_\eta$ with respect to homogeneous Neumann boundary conditions. 
We must complement $\pi_\star$ by a function $c_\pi(t)$ depending on only time which is uniquely determined by the structure equation. Setting
$\pi(t)=\pi_\star(t)+c_\pi(t)$ and 
testing the structure equation with 1 we obtain
\begin{align}\label{eq:pressure}
\begin{aligned}
c_\pi(t)\int_{\omega}\bfn\cdot\bfn_\eta\det(\naby\bm{\varphi}_\eta)\dy&=\int_{\omega}\bfn(\nabx\bu+\nabx\bu^\top-\pi_\star\mathbb{I}_{3\times3}+\mathbb{S}_\bq(\hbar))\circ \bm{\varphi}_\eta\bfn_\eta\det(\naby\bm{\varphi}_\eta)\dy\\&+\int_\omega\partial_t^2\eta\dy-\int_\omega g\dy.
\end{aligned}
\end{align}
This equation can be solved for $c_\pi(t)$ provided the integral on the left-hand side is strictly positive (which certainly holds if the $W^{1,\infty}_\by$-norm of $\eta$ is not too large, cf.~\eqref{eq:0108}).

\begin{definition}[Strong solution]
\label{def:strongSolutionAlone}
Let $(\bff, g, \eta_0, \eta_\star, \bu_0, \mathbb{S}_\bq( \hbar))$ be a dataset such that
\begin{equation}
\begin{aligned}
\label{datasetAlone}
&\bff \in L^2\big(I; L^2_{\mathrm{loc}}(\mathbb{R}^3 )\big),\quad
g \in L^2\big(I; L^{2}(\omega)\big), \quad
\eta_0 \in W^{3,2}(\omega) \text{ with } \Vert \eta_0 \Vert_{L^\infty( \omega)} < L, \quad
\eta_\star \in W^{1,2}(\omega), 
\\
& \mathbb{S}_\bq( \hbar)\in L^2(I;W^{1,2}_{\mathrm{loc}}(\R^3)), \quad
\bu_0\in W^{1,2}_{\mathrm{\divx}}(\Omega_{\eta_0} ) \text{ is such that }\bu_0 \circ \bm{\varphi}_{\eta_0} =\eta_\star \bn \text{ on } \omega.
\end{aligned}
\end{equation} 
We call 
$( \eta, \bu,  \pi )$
a \textit{strong solution}  of  \eqref{shellEqAlone}--\eqref{contEqAlone} with data $(\bff, g, \eta_0, \eta_\star, \bu_0, \mathbb{S}_\bq( \hbar))$ provided that the following holds:
\begin{itemize}
\item[(a)]  the structure-function $\eta$ is such that $
\Vert \eta \Vert_{L^\infty(I \times \omega)} <L$ and
\begin{align*}
\eta \in W^{1,\infty}\big(I;W^{1,2}(\omega)  \big)\cap L^{\infty}\big(I;W^{3,2}(\omega)  \big) \cap  W^{2,2}\big(I;L^2(\omega)  \big)\cap L^{2}\big(I;W^{4,2}(\omega)  \big) ;
\end{align*}
\item[(b)] the velocity $\bu$ is such that $\bu  \circ \bm{\varphi}_{\eta} =(\partial_t\eta)\bn$ on $I\times \omega$ and
\begin{align*}
\bu\in  W^{1,2} \big(I; L^2_{\divx}(\Omega_{\eta(t)} ) \big)\cap L^2\big(I;W^{2,2}(\Oeta)  \big);
\end{align*}
\item[(c)] the pressure $\pi$ is such that
% $\nabx\pi=\nabx\Delta^{-1}_{\Omega_\eta}\divx[(\bu\cdot\nabx)\bu-\Delx\bu]$ and
\begin{align*}
\pi\in L^2\big(I;W^{1,2}(\Oeta)  \big);
\end{align*}
\item[(d)] equations \eqref{shellEqAlone}--\eqref{contEqAlone} are satisfied a.e. in space-time with $\eta(0)=\eta_0$ and $\partial_t\eta=\eta_\star$ a.e. in $\omega$ as well as $\bfu(0)=\bfu_0$ a.e. in $\Omega_{\eta_0}$.
\end{itemize}
\end{definition}
\noindent
The existence of a unique local-in-time strong solution to \eqref{shellEqAlone}--\eqref{contEqAlone} in the sense of Definition \ref{def:strongSolutionAlone} is shown in \cite{BMSS}. The regularity obtained is, however, not sufficient for the coupling with the Fokker--Planck equation. Hence we are going to prove a corresponding result in higher-order Sobolev spaces.
Our main theorem is the following:
\begin{theorem}
\label{thm:fluidStructureWithoutFK}
Suppose that the dataset
$(\bff, g, \eta_0, \eta_\star, \bu_0, \mathbb{S}_\bq( \hbar))$
satisfies
 \eqref{mainDataForAll}--\eqref{compatibilityConditionMain} (which is stronger than \eqref{datasetAlone}) and
\begin{align*}
\mathbb{S}_\bq( \hbar) \in L^2(I;W^{2,2}(\Oeta)) \cap W^{1,2}(I;W^{1,2}(\Oeta)),
 \quad
 \mathbb{S}_\bq(\overline{\hbar} (0))\in  W^{2,2}(\Omega_{\eta_0}).
\end{align*}
% and in a addition
%\begin{align*}
%g\in L^2(I;W^{1,2}(\omega)),\quad \eta_0 \in W^{3,2}(\omega), \quad \eta_\star \in W^{1,2}(\omega), \quad \bu_0 \in W^{1,2}_{\divx}(\Omega_{\eta_0}), \quad \overline{\mathbb{S}}_\bq\in L^2\ldots\ldots.
%\end{align*}
There is a time $T_*>0$ such that  \eqref{shellEqAlone}--\eqref{contEqAlone} admits a unique strong solution $( \eta, \bu,  \pi )$, in the sense of Definition \ref{def:strongSolutionAlone}, that further satisfies
\begin{align*}
\eta&\in W^{1,\infty}\big(I_*;W^{3,2}(\omega)  \big) \cap W^{2,2}\big(I_*;W^{1,2}(\omega)  \big)
\cap W^{3,2}\big(I_*;L^2(\omega)  \big)
\\
&\qquad\qquad\qquad \cap W^{1,2}\big(I_*;W^{3,2}(\omega)  \big)\cap L^{\infty}\big(I_*;W^{4,2}(\omega)  \big)\cap L^{2}\big(I_*;W^{5,2}(\omega)  \big),\\
\bu &\in
W^{1,\infty} \big(I_*; W^{1,2}(\Oeta) \big)\cap  W^{2,2}\big(I_*;L^2(\Oeta)  \big)\cap W^{1,2}\big(I_*;W^{2,2}(\Oeta)  \big)
\cap L^2\big(I_*;W^{3,2}(\Oeta)  \big),
\\
\pi &\in
W^{1,2}\big(I_*;W^{1,2}(\Oeta)  \big)
\cap
L^2\big(I_*;W^{2,2}(\Oeta)  \big).
\end{align*}
\end{theorem}
\begin{remark}\label{rem:fsi1/2}
The choice of the rather unusual topology for the solution in Theorem \ref{thm:fluidStructureWithoutFK} is due to the coupling with the Fokker-Planck equation which is our main aim (see Section \ref{sec:nov} and Remark \ref{rem:reg} for the explanation). For instance, one can  also construct solutions provided it only holds
\begin{equation}
\begin{aligned}
\label{mainDataForAll'}
&\bff \in  W^{1/2,2}(I;L^{2}(\Oeta))\cap L^2(I;W^{1,2}(\Oeta)) ,\quad  \bff(0)\in W^{1/2,2}(\Omega_{\eta_0}),
\\
&g\in  L^{2}(I;W^{1,2}(\omega))\cap  W^{1/2,2}(I;W^{1/2,2}(\omega)),
\quad  g(0)\in W^{1/2,2}(\omega)
\\&\eta_0 \in W^{4,2}(\omega) \text{ with } \Vert \eta_0 \Vert_{L^\infty( \omega)} < L, \quad \eta_\star \in W^{2,2}(\omega),
\\&\bu_0 \in W^{2,2}_{\divx}(\Omega_{\eta_0} )\text{ is such that }\bu_0 \circ \bm{\varphi}_{\eta_0} =\eta_\star \bn \text{ on } \omega,
\\&
\mathbb{S}_\bq( \hbar) \in L^2(I;W^{2,2}(\Oeta)) \cap W^{1/2,2}(I;W^{1,2}(\Oeta)),
 \quad
 \mathbb{S}_\bq(\overline{\hbar} (0))\in  W^{1,2}(\Omega_{\eta_0})
\end{aligned}
\end{equation}
rather than  \eqref{mainDataForAll}. In this case the solution belongs to the following regularity class
\begin{align*}
\eta&\in W^{1,\infty}\big(I_*;W^{3,2}(\omega)  \big)
\cap W^{5/2,2}\big(I_*;L^2(\omega)  \big)
\\
&\qquad\qquad\qquad \cap W^{3/2,2}\big(I_*;W^{2,2}(\omega)  \big)\cap L^{\infty}\big(I_*;W^{4,2}(\omega)  \big)\cap L^{2}\big(I_*;W^{5,2}(\omega)  \big),\\
\bu &\in
W^{3/2,2}\big(I_*;L^2(\Oeta)  \big)\cap W^{1/2,2}\big(I_*;W^{2,2}(\Oeta)  \big)
\cap L^2\big(I_*;W^{3,2}(\Oeta)  \big),
\\
\pi &\in
W^{1/2,2}\big(I_*;W^{1,2}(\Oeta)  \big)
\cap
L^2\big(I_*;W^{2,2}(\Oeta)  \big).
\end{align*}
Similarly, if we have
\begin{equation}
\begin{aligned}
\label{mainDataForAll''}
&\bff \in  W^{1,2}(I;L^{2}(\Oeta))\cap L^2(I;W^{2,2}(\Oeta)) ,\quad  \bff(0)\in W^{1,2}(\Omega_{\eta_0}),
\\
&g\in  L^{2}(I;W^{2,2}(\omega))\cap  W^{1,2}(I;W^{1,2}(\omega)),
%\quad  g(0)\in W^{1,2}(\omega)
\\&\eta_0 \in W^{5,2}(\omega) \text{ with } \Vert \eta_0 \Vert_{L^\infty( \omega)} < L, \quad \eta_\star \in W^{3,2}(\omega),
\\&\bu_0 \in W^{3,2}_{\divx}(\Omega_{\eta_0} )\text{ is such that }\bu_0 \circ \bm{\varphi}_{\eta_0} =\eta_\star \bn \text{ on } \omega,
\\&
\mathbb{S}_\bq( \hbar) \in L^2(I;W^{3,2}(\Oeta)) \cap W^{1,2}(I;W^{1,2}(\Oeta)),
 \quad
 \mathbb{S}_\bq(\overline{\hbar} (0))\in  W^{2,2}(\Omega_{\eta_0})
\end{aligned}
\end{equation}
then the solution satisfies
\begin{align*}
\eta&\in W^{1,\infty}\big(I_*;W^{3,2}(\omega)  \big)
\cap W^{3,2}\big(I_*;L^2(\omega)  \big)
\\
&\qquad\qquad\qquad \cap W^{2,2}\big(I_*;W^{2,2}(\omega)  \big)\cap L^{\infty}\big(I_*;W^{5,2}(\omega)  \big)\cap L^{2}\big(I_*;W^{6,2}(\omega)  \big),\\
\bu &\in
W^{2,2}\big(I_*;L^2(\Oeta)  \big)\cap W^{1,2}\big(I_*;W^{2,2}(\Oeta)  \big)
\cap L^2\big(I_*;W^{4,2}(\Oeta)  \big),
\\
\pi &\in
W^{1,2}\big(I_*;W^{1,2}(\Oeta)  \big)
\cap
L^2\big(I_*;W^{3,2}(\Oeta)  \big).
\end{align*}
\end{remark}
\noindent
In order to prove Theorem \ref{thm:fluidStructureWithoutFK}, we follow the following strategy which has been successfully implemented before, for instance in \cite{breit2022regularity,BMSS,GH,GHL,Leq}.
\begin{itemize}
\item We transform the solvent-structure system to its reference domain.
\item We then linearise the resulting system on the reference domain and obtain estimates for the linearised system.
\item We construct a contraction map for the linearised problem (by choosing the end time
small enough) which gives the local solution to the system on its original/actual domain.
\end{itemize}
\subsection{Transformation to reference domain}\label{sec:ref}
For a solution $( \eta, \bu,  \pi )$  of \eqref{shellEqAlone}--\eqref{contEqAlone}, we define $\overline{\pi}=\pi\circ \bm{\Psi}_\eta$ and 
$\overline{\bu}=\bu\circ \bm{\Psi}_\eta$
as well as define
\begin{align*}
\mathbf{A}_\eta=J_\eta\big( \nabx \bm{\Psi}_\eta^{-1}\circ \bm{\Psi}_\eta \big)^\mathtt{T}\nabx \bm{\Psi}_\eta^{-1}\circ \bm{\Psi}_\eta,\\
\mathbf{B}_\eta=J_\eta \nabx \bm{\Psi}_\eta^{-1}\circ \bm{\Psi}_\eta,\\
h_\eta(\overline{\bu})=\big( \mathbf{B}_{\eta_0}-\mathbf{B}_\eta\big):\nabx \overline{\bu},
\\
\mathbf{H}_\eta(\overline{\bu}, \overline{\pi})
=\
\big( \mathbf{A}_{\eta_0}-\mathbf{A}_\eta\big)\nabx \overline{\bu}
-
\big( \mathbf{B}_{\eta_0}-\mathbf{B}_\eta\big) [\overline{\pi}\,\mathbb I_{3\times 3}-\mathbb{S}_\bq(\overline{\hbar}) ],
\\
\mathbf{h}_\eta(\overline{\bu})
=
(J_{\eta_0}-J_\eta)\partial_t \overline{\bu}
%-
%J_\eta\big( \nabx \bm{\Psi}_\eta^{-1}\circ \bm{\Psi}_\eta \nabx \overline{\bu} \big)\big(\partial_t \bm{\Psi}_\eta^{-1}\circ \bm{\Psi}_\eta + \overline{\bu}\big)
-
J_\eta  \nabx \overline{\bu} \big(\partial_t \bm{\Psi}_\eta^{-1}\circ \bm{\Psi}_\eta  \big)
-
J_\eta\big( \nabx \bm{\Psi}_\eta^{-1}\circ \bm{\Psi}_\eta  \big)( \overline{\bu}\cdot\nabx) \overline{\bu}
+
J_\eta  \bff \circ \bm{\Psi}_\eta 
\end{align*}
where $\overline{\hbar}=\hbar\circ \bm{\Psi}_\eta$.
The following result holds true and can be found in \cite[Lemma 4.2]{breit2022regularity}. 
\begin{theorem}
\label{thm:transformedSystem}
Suppose that the dataset
$(\bff, g, \eta_0, \eta_\star, \bu_0, \mathbb{S}_\bq(\hbar))$
satisfies \eqref{datasetAlone}.
% and in a addition
%\begin{align*}
%g\in L^2(I;W^{1,2}(\omega)),\quad \eta_0 \in W^{3,2}(\omega), \quad \eta_\star \in W^{1,2}(\omega), \quad \bu_0 \in W^{1,2}_{\divx}(\Omega_{\eta_0}), \quad \overline{\mathbb{S}}_\bq\in L^2\ldots\ldots.
%\end{align*}
Then $( \eta, \bu,  \pi )$ is strong solution to \eqref{shellEqAlone}--\eqref{contEqAlone}, in the sense of Definition \ref{def:strongSolutionAlone}, if and only if $( \eta, \overline{\bu},  \overline{\pi} )$ is a strong solution of
\begin{align}
\partial_t^2\eta - \partial_t\Dely \eta + \Dely^2\eta
&=
g+\bn^\top \big[\mathbf{H}_\eta(\overline{\bu}, \overline{\pi})-\mathbf{A}_{\eta_0}  \nabx\overline{\bu} +\mathbf{B}_{\eta_0}(\overline{\pi}\,\mathbb I_{3\times 3}-
\mathbb{S}_\bq(\overline{\hbar}) )\big]\circ\bm{\varphi} \bn ,
\label{shellEqAloneBar}
\\
J_{\eta_0}\partial_t \overline{\bu}  -\divx(\mathbf{A}_{\eta_0}  \nabx\overline{\bu}-\mathbf{B}_{\eta_0}\overline{\pi}) 
& = 
 \divx(\mathbf{B}_{\eta_0}\mathbb{S}_\bq(\overline{\hbar}) ) +
\mathbf{h}_\eta(\overline{\bu})-
\divx  \mathbf{H}_\eta(\overline{\bu}, \overline{\pi}),
\label{momEqAloneBar}\\
\label{contEqAloneBar}
\mathbf{B}_{\eta_0}:\nabx \overline{\bu}&= h_\eta(\overline{\bu})
\end{align}
in $I\times \Omega$ with  $\overline{\bu}  \circ \bm{\varphi}  =(\partial_t\eta)\bn$ on $I\times \omega$.
\end{theorem}

\subsection{The linearised problem}
In this section, we let $(g, \eta_0, \eta_\star)$ be as before in Theorem \ref{thm:transformedSystem}. In addition, we take $(h,\mathbf{h}, \mathbf{H}, \overline{\bu}_0, \mathbb{S}_\bq(\overline{\hbar}))$ that satisfies
\begin{align*}
&h\in L^2(I;W^{1,2}(\Omega))\cap W^{1,2}(I;W^{-1,2}(\Omega)),
\\ 
&\mathbf{h}\in L^2(I;L^2(\Omega)), \quad \mathbf{H},\,\mathbb{S}_\bq(\overline{\hbar})\in L^2(I;W^{1,2}(\Omega)),
\\&\overline{\bu}_0\in W^{1,2}(\Omega),
\quad \overline{\bu}_0\circ \bm{\varphi} =\eta_\star\bn, \quad \mathbf{B}_{\eta_0}:\nabx \overline{\bu}_0=h,
\end{align*}
and consider the following linear system
\begin{align}
 \partial_t^2\eta - \partial_t\Dely \eta +  \Dely^2\eta
&=
g+\bn^\top \big[\mathbf{H} -\mathbf{A}_{\eta_0}  \nabx\overline{\bu} +\mathbf{B}_{\eta_0}(\overline{\pi}\,\mathbb I_{3\times 3} -  \mathbb{S}_\bq(\overline{\hbar}))\big]\circ\bm{\varphi} \bn ,
\label{shellEqAloneBarLinear}
\\
J_{\eta_0}\partial_t \overline{\bu}  -\divx(\mathbf{A}_{\eta_0}  \nabx\overline{\bu}
-
\mathbf{B}_{\eta_0}\overline{\pi}) 
&=
 \divx(\mathbf{B}_{\eta_0}\mathbb{S}_\bq(\overline{\hbar}) ) 
 +
\mathbf{h}-
\divx  \mathbf{H},
\label{momEqAloneBarLinear}\\
\label{contEqAloneBarLinear}
\mathbf{B}_{\eta_0}:\nabx \overline{\bu}&= h
\end{align}
in $I\times \Omega$ with  $\overline{\bu}  \circ \bm{\varphi}  =(\partial_t\eta)\bn$ on $I\times \omega$ and with$\eta(0)=\eta_0$ and $\partial_t\eta=\eta_\star$ a.e. in $\omega$ as well as $\overline \bfu(0)=\overline\bfu_0$ a.e. in $\Omega$. As shown in \cite[Proposition 3.3]{BMSS}, we obtain
\begin{equation}
\begin{aligned}
\label{energyEstLinear}
&\sup_I\int_\omega
\big(\vert \partial_t\naby \eta\vert^2 
+
\vert \naby\Dely \eta\vert^2
\big)
\dy
+
\sup_I\int_\Omega\vert\nabx \overline{\bu}\vert^2\dx
\\&+
\int_I\int_\omega
\big(\vert \partial_t\Dely \eta \vert^2 + \vert \partial_t^2 \eta\vert^2
+ \vert \Dely^2 \eta\vert^2
  \big)\dy\dt
 +
\int_I\int_\Omega\big( \vert \nabx^2\overline{\bu}\vert^2 +\vert \partial_t\overline{\bu} \vert^2  +\vert \overline{\pi}\vert^2 + \vert \nabx \overline{\pi}\vert^2
 \big)\dx\dt
 \\&\lesssim
 \int_\omega\big( 
%\vert \eta_\star\vert^2
% +
 \vert \naby\eta_\star\vert^2
% +
% \vert \Dely\eta_0\vert^2
 +
  \vert \naby\Dely\eta_0\vert^2
  \big)\dy
  +
  \int_\Omega \big(\vert
  \overline{\bu}_0\vert^2
  +
   \vert\nabx\overline{\bu}_0\vert^2 \big)\dx
   \\&+
   \int_I\Vert \partial_t h \Vert_{W^{-1,2}(\Omega)}^2\dt
   +
 \int_I\int_\omega  \vert g\vert^2 \dy\dt
  \\&+
  \int_I\int_\Omega\big(
  \vert  h\vert^2
  +
  \vert \nabx h\vert^2
  +
   \vert \mathbf{h}\vert^2 +
  \vert \mathbf{H}\vert^2
  +
  \vert \nabx\mathbf{H}\vert^2
  +
  \vert \mathbb{S}_\bq(\overline{\hbar}) ) \vert^2
  +
  \vert \nabx\mathbb{S}_\bq(\overline{\hbar}) ) \vert^2
  \big)
 \dx\dt.
\end{aligned}
\end{equation}
For a dataset that is more regular in time and space, our goal now is to obtain higher-in-time (and then in space) regularity for the strong solution above.
This requires assuming the compatibility condition \footnote{Setting $h=h_\eta(\overline{\bu}), \mathbf h=\mathbf h_\eta(\overline{\bu})$ and $\mathbf H=\mathbf H_\eta(\overline{\bu},\overline{\pi})$ with the definitions from Section \ref{sec:ref}, this is equivalent to \eqref{compatibilityConditionMain}.}
\begin{equation}
\begin{aligned}
\label{compatibilityCondition}
\big[\Dely  \eta_\star &-  \Dely^2 \eta_0
+
g(0)
+
\bn^\top \big[ \mathbf{H}(0) -\mathbf{A}_{\eta_0}  \nabx\overline{\bu}_0 +\mathbf{B}_{\eta_0}(\overline{\pi}_0\,\mathbb I_{3\times 3}
-
\mathbb{S}_\bq(\overline{\hbar}(0) ))\big]\circ\bm{\varphi}  \bn \big]\bn
\\&=
J_{\eta_0}^{-1}\big[ \divx(\mathbf{A}_{\eta_0}  \nabx\overline{\bu}_0
 -\mathbf{B}_{\eta_0}\overline{\pi}_0+
 \mathbf{B}_{\eta_0}\mathbb{S}_\bq(\overline{\hbar}(0) ))
 + 
\mathbf{h}(0)-
\divx  \mathbf{H}(0)  \big] \circ\bm{\varphi}
\end{aligned}
\end{equation}
on $\omega$ for the data. Here the initial pressure $\pi_0$ is the unique solution to the elliptic problem
\begin{equation}
\begin{aligned}
\divx(\mathbf{A}_{\eta_0}\nabx\overline{\pi}_0) 
% &
%=
% -\divx(\mathbf{B}_{\eta_0}^\top
%\partial_t \overline{\bu} ) 
% + \divx\big(\mathbf{B}_{\eta_0}^\top J_{\eta_0}^{-1}\big\{ \divx(\mathbf{A}_{\eta_0}  \nabx\overline{\bu}
% +
% \mathbf{B}_{\eta_0}\mathbb{S}_\bq(\overline{\hbar}
% ) ) 
% +
%\mathbf{h} -
%\divx  \mathbf{H} 
%\}\big)
%\\
&
=
 - 
\partial_th(0)
 + \divx\big(\mathbf{B}_{\eta_0}^\top J_{\eta_0}^{-1}\big\{ \divx(\mathbf{A}_{\eta_0}  \nabx\overline{\bu}_0
 +
 \mathbf{B}_{\eta_0}\mathbb{S}_\bq(\overline{\hbar}
 (0)) )\}\big) \\
& + \divx\big(\mathbf{B}_{\eta_0}^\top J_{\eta_0}^{-1}\big\{ \mathbf{h}(0) -
\divx  \mathbf{H}(0) 
\}\big)
\label{4'a}
\end{aligned}
\end{equation}
in $\Omega$ with the Neumann boundary condition
\begin{equation}
\begin{aligned}
\mathbf{A}_{\eta_0}\nabx\overline{\pi}_0&\cdot \bn \circ \bm{\varphi}^{-1} 
 +
J_{\eta_0} 
\bn^\top\circ \bm{\varphi}^{-1}   \mathbf{A}_{\eta_0}\overline{\pi}
_0\cdot \bn\circ \bm{\varphi}^{-1} 
\\&=
\mathbf{B}_{\eta_0}^\top J_{\eta_0}^{-1}\big\{ \divx(\mathbf{A}_{\eta_0}  \nabx\overline{\bu}_0
+
\mathbf{B}_{\eta_0}\mathbb{S}_\bq(\overline{\hbar}
 (0))) 
 +
\mathbf{h}(0) -
\divx  \mathbf{H}(0)
\}\cdot \bn\circ \bm{\varphi}^{-1} 
\\
&
-
 \mathbf{B}_{\eta_0}^\top
\big[\big\{\Dely\eta_\star - \Dely^2\eta_0+g(0)\big\}\bfn\big]\circ \bm{\varphi}^{-1}\bfn\circ \bm{\varphi}^{-1} 
\\&
-
\mathbf{B}_{\eta_0}^\top
\bn^\top\circ \bm{\varphi}^{-1} \big[\mathbf{H}(0) -\mathbf{A}_{\eta_0}  \nabx\overline{\bu}_0  \big]
\bn\circ \bm{\varphi}^{-1}
+
J_{\eta_0} 
\bn^\top\circ \bm{\varphi}^{-1}  \mathbf{A}_{\eta_0}
\mathbb{S}_\bq(\overline{\hbar}
(0)) \bn\circ \bm{\varphi}^{-1}  
\label{6a}
\end{aligned}
\end{equation}
on $\partial\Omega$.
Our main result in this subsection is the following.
\begin{proposition}
\label{prop:bigData}
Suppose that the dataset
$(g, \eta_0, \eta_\star, \overline{\bu}_0,\overline{\hbar}, h, \mathbf{h},\mathbf{H})$
satisfies \eqref{datasetAlone} and in addition
\begin{equation}
\begin{aligned}
\label{bigData}
&g\in L^2(I;W^{2,2}(\omega)) \cap W^{1,2}(I;W^{1,2}(\omega)) 
%,\quad  g(0)\in W^{1,2}(\omega),
\\&\eta_0 \in W^{5,2}(\omega) \text{ with } \Vert \eta_0 \Vert_{L^\infty( \omega)} < L, \quad \eta_\star \in W^{3,2}(\omega),
\\&\overline{\bu}_0 \in W^{3,2}(\Omega),
\quad \overline{\bu}_0\circ \bm{\varphi} =\eta_\star\bn, \quad \mathbf{B}_{\eta_0}:\nabx \overline{\bu}_0=h,
\\&
h\in L^2(I;W^{3,2}(\Omega)) \cap W^{1,2}(I;W^{1,2}(\Omega))  \cap W^{2,2}(I;W^{-1,2}(\Omega)) 
%\cap \{h(0,\cdot)=0\}
, 
\\&
\mathbf{h} \in L^2(I;W^{2,2}(\Omega))
 \cap W^{1,2}(I;L^2(\Omega)) 
 ,\quad 
 \mathbf{h}(0)\in  W^{1,2}(\Omega),
 \\&
 \mathbf{H}\in L^2(I;W^{3,2}(\Omega)) \cap W^{1,2}(I;W^{1,2}(\Omega)),
 \quad
 \mathbf{H}(0)\in  W^{2,2}(\Omega),
 \\&
\mathbb{S}_\bq(\overline{\hbar} )\in L^2(I;W^{3,2}(\Omega)) \cap W^{1,2}(I;W^{1,2}(\Omega)),
 \quad
 \mathbb{S}_\bq(\overline{\hbar} (0))\in  W^{2,2}(\Omega),
\end{aligned}
\end{equation}
with the compatibility condition
\eqref{compatibilityCondition}.
 Then a strong solution $( \eta,  \overline{\bu},  \overline{\pi})$ of \eqref{shellEqAloneBarLinear}--\eqref{contEqAloneBarLinear} satisfies
 \begin{equation}
\begin{aligned}
\label{timeEstimate}
&\sup_I\int_\omega
\big(\vert \partial_t^2\naby \eta\vert^2 
+
\vert \partial_t\naby^3 \eta\vert^2+
\vert \naby^5 \eta\vert^2
\big)
\dy
+
\sup_I\int_\Omega\vert \partial_t\nabx \overline{\bu}\vert^2\dx
\\&+
\int_I\big(\Vert\overline{\bu}
 \Vert_{W^{4,2}(\Omega )}^2
 +
 \Vert\overline{\pi} \Vert_{W^{3,2}(\Omega )}^2\big) \dt
+
\int_I\int_\omega
\big(\vert \partial_t^2\naby^2 \eta \vert^2+\vert \partial_t\naby^4 \eta \vert^2 + \vert \partial_t^3 \eta\vert^2
 +\vert\naby^6\eta\vert^2\big)\dy\dt
 \\&+
\int_I\int_\Omega\big( \vert \partial_t \nabx^2\overline{\bu}\vert^2 +\vert \partial_t^2\overline{\bu} \vert^2  +\vert \partial_t\overline{\pi}\vert^2 + \vert \partial_t\nabx \overline{\pi}\vert^2
 \big)\dx\dt
\lesssim
 \mathcal{D}(g, \eta_0, \eta_\star, \bu_0, h, \mathbf{h},\mathbf{H}),
\end{aligned}
\end{equation}
where
\begin{equation}
\begin{aligned}
\label{bigDataNorm}
\mathcal{D}&(g, \eta_0, \eta_\star, \bu_0, h, \mathbf{h},\mathbf{H}):=
\Vert \eta_\star\Vert_{W^{3,2}(\omega)}^2
 +
 \Vert \eta_0\Vert_{W^{5,2}(\omega)}^2
  +
 \Vert
  \overline{\bu}_0\Vert_{W^{3,2}(\Omega)}^2
  +
  \Vert \mathbb{S}_\bq(\overline{\hbar}(0))
  \Vert_{W^{2,2}(\Omega)}^2
   \\&
    +
   \Vert \mathbf{h}(0)\Vert_{W^{1,2}(\Omega)}^2 
   +
  \Vert \mathbf{H}(0)\Vert_{W^{2,2}(\Omega)}^2
   +
   \int_I\Vert \partial_t^2 h \Vert_{W^{-1,2}(\Omega)}^2\dt
    +
 \Vert
  g(0)\Vert_{W^{1,2}(\omega)}^2
   \\&+
   \int_I \big(\Vert \partial_t g \Vert_{W^{1,2}(\omega)}^2
   + 
 \Vert g\Vert_{W^{2,2}(\omega)}^2
 %\todo{******we may not need this. Check******}
 +
 \Vert \partial_t h \Vert_{W^{1,2}(\Omega)}^2   +
 \Vert   h \Vert_{W^{3,2}(\Omega)}^2  
 +
 \Vert \partial_t\mathbf{h}\Vert_{L^2(\Omega)}^2\big)\dt
  \\&+
  \int_I\big(
   \Vert \mathbf{h}\Vert_{W^{2,2}(\Omega)}^2
  +
\Vert
\mathbb{S}_\bq(\partial_t\overline{\hbar} )
\Vert_{W^{1,2}(\Omega)}^2
 +
\Vert
\mathbb{S}_\bq(\overline{\hbar} )
\Vert_{W^{3,2}(\Omega)}^2 +
  \Vert \partial_t\mathbf{H}\Vert_{W^{1,2}(\Omega)}^2
+
  \Vert \mathbf{H}\Vert_{W^{3,2}(\Omega)}^2
  \big)
\dt.
\end{aligned}
\end{equation}
\end{proposition} 
\begin{proof}
The proof of Proposition \ref{prop:bigData} will be obtained in three steps. Firstly, we differentiate in time, each of the equations in \eqref{shellEqAloneBarLinear}--\eqref{contEqAloneBarLinear} as well as the interface condition $\overline{\bu}  \circ \bm{\varphi}  =(\partial_t\eta)\bn$. Since the system \eqref{shellEqAloneBarLinear}--\eqref{contEqAloneBarLinear} is linear, the resulting system after differentiating in time will be of the same form except for the extra time derivative applied to the individual terms in the system. Also, the initial conditions are no longer given but now  solve PDEs as well. Consequently, in the first instant, our new system will also satisfy the inequality \eqref{energyEstLinear} (for the time derivatives of each term). However, since the initial conditions solve  PDEs, we will have to estimate them as well. The estimate for these initial conditions is our second step. Finally, our third step will consist of obtaining estimates for the remaining terms on the left-hand side of  \eqref{timeEstimate} (which happens to be the highest spatial regularity for the velocity field and the pressure) in terms of  $\mathcal{D}(\cdot)$  as defined in \eqref{bigDataNorm}.
\\
Let us now give further details. We argue formerly, a rigorous proof can be obtained by working with a Galerkin approximation (see also \cite[Section 3]{BMSS} and \cite[Section 4]{schwarzacherSu}). This is also where the compatibility condition \eqref{compatibilityCondition} comes into play to obtain sufficient temporal regularity.
 First, in order to simplify notation, let us  set
\begin{align}
\label{tildePartialT}
\tilde{\bu}=\partial_t\overline{\bu},
\quad \tilde{\pi}=\partial_t\overline{\pi}, \quad
\tilde{\hbar}=\partial_t\overline{\hbar}, \quad
 \tilde{\eta}=\partial_t\eta,
\quad \tilde{g}=\partial_t g,
\quad \tilde{h}=\partial_t h, \quad\tilde{\mathbf{h}}=\partial_t \mathbf{h}, \quad
\tilde{\mathbf{H}}=\partial_t \mathbf{H}.
\end{align}
We now obtain the system
\begin{align}
 \partial_t^2\tilde{\eta} - \partial_t\Dely\tilde{ \eta} +  \Dely^2\tilde{\eta}
&=
\tilde{g} 
+
\bn^\top \big[\tilde{\mathbf{H}} -\mathbf{A}_{\eta_0}  \nabx\tilde{\bu} +\mathbf{B}_{\eta_0}(\tilde{\pi}\,\mathbb I_{3\times 3}
-
\mathbb{S}_\bq(\tilde{\hbar} ))\big]\circ\bm{\varphi} \bn ,
\label{shellEqAloneBarLinearTime}
\\
J_{\eta_0}\partial_t \tilde{\bu}  -\divx(\mathbf{A}_{\eta_0}  \nabx\tilde{\bu}
-\mathbf{B}_{\eta_0}\tilde{\pi}) 
 &=\divx(\mathbf{B}_{\eta_0}\mathbb{S}_\bq(\tilde{\hbar} ))
 +
\tilde{\mathbf{h}}-
\divx  \tilde{\mathbf{H}},
\label{momEqAloneBarLinearTime}\\
\label{contEqAloneBarLinearTime}
\mathbf{B}_{\eta_0}:\nabx \tilde{\bu}&= \tilde{h}
\end{align}
with  $\tilde{\bu}  \circ \bm{\varphi}  =(\partial_t\tilde{\eta})\bn$ on $I\times \omega$ and with the initial conditions 
\begin{align}
&\tilde{\eta}(0, \cdot) = \tilde{\eta}_0, \quad \partial_t\tilde{\eta}(0, \cdot) = \tilde{\eta}_\star
&\quad \text{in }  \omega,
\\
&\tilde{\mathbf{u}}(0, \cdot) = \tilde{\mathbf{u}}_0
&\quad \text{in }  \Omega.
\end{align}
Here, the initial data   $( \tilde{\eta}_0, \tilde{\eta}_1, \tilde{\bu}_0)$  satisfies
\begin{align}
&\tilde{\eta}_0=\eta_\star,
\label{eta0eta1}
\\
&\tilde{\eta}_\star=
\Dely  \eta_\star -  \Dely^2 \eta_0
+
g(0) 
+
\bn^\top \big[ \mathbf{H}(0) -\mathbf{A}_{\eta_0}  \nabx\overline{\bu}_0 +\mathbf{B}_{\eta_0}(\overline{\pi}_0\,\mathbb I_{3\times 3}
-
\mathbb{S}_\bq(\overline{\hbar}(0) ))\big]\circ\bm{\varphi} \bn, 
\label{tiledEtaNoTiledEta}
\\
&\tilde{\bu}_0  
=
J_{\eta_0}^{-1}\big[\divx(\mathbf{A}_{\eta_0}  \nabx\overline{\bu}_0
-\mathbf{B}_{\eta_0}\overline{\pi}_0
+
\mathbf{B}_{\eta_0}\mathbb{S}_\bq(\overline{\hbar}(0) ))
 + 
\mathbf{h}(0)-
\divx  \mathbf{H}(0) \big].
\label{tiledUNoTiledU}
\end{align}
The initial pressure $\overline\pi_0$ is prescribed by the data via \eqref{4'a}--\eqref{6a}.
Because the system \eqref{shellEqAloneBarLinearTime}--\eqref{contEqAloneBarLinearTime} is of the same form as  \eqref{shellEqAloneBarLinear}--\eqref{contEqAloneBarLinear}, we can directly infer from \eqref{energyEstLinear} that
\begin{equation}
\begin{aligned}
\label{energyEstLinearTilde}
&\sup_I\int_\omega
\big(\vert \partial_t\naby \tilde{\eta}\vert^2 
+
\vert \naby\Dely \tilde{\eta}\vert^2
\big)
\dy
+
\sup_I\int_\Omega\vert\nabx \tilde{\bu}\vert^2\dx
\\&+
\int_I\int_\omega
\big(\vert \partial_t\Dely \tilde{\eta} \vert^2 + \vert \partial_t^2 \tilde{\eta}\vert^2+ \vert \Dely^2 \tilde{\eta}\vert^2
 \big)\dy\dt
 +
\int_I\int_\Omega\big( \vert \nabx^2\tilde{\bu}\vert^2 +\vert \partial_t\tilde{\bu} \vert^2  +\vert \tilde{\pi}\vert^2 + \vert \nabx \tilde{\pi}\vert^2
 \big)\dx\dt
 \\&\lesssim
 \int_\omega\big( \vert \tilde{\eta}_\star\vert^2
 +
 \vert \naby\tilde{\eta}_\star\vert^2
 +
 \vert \Dely\tilde{\eta}_0\vert^2
 +
  \vert \naby\Dely\tilde{\eta}_0\vert^2
  \big)\dy
  +
  \int_\Omega \big(\vert
  \tilde{\bu}_0\vert^2
  +
   \vert\nabx\tilde{\bu}_0\vert^2 \big)\dx
   \\&
+
   \int_I\Vert \partial_t \tilde{h} \Vert_{W^{-1,2}(\Omega)}^2\dt
   +
 \int_I\int_\omega  \vert \tilde{g}\vert^2 \dy\dt
  \\&
  +
  \int_I\int_\Omega\big(
  \vert \tilde{h}\vert^2
  +
  \vert \nabx \tilde{h}\vert^2
  +
   \vert \tilde{\mathbf{h}}\vert^2 +
  \vert \tilde{\mathbf{H}}\vert^2
  +
  \vert \nabx\tilde{\mathbf{H}}\vert^2
  +
  \vert \mathbb{S}_\bq(\tilde{\hbar}) ) \vert^2
  +
  \vert \nabx\mathbb{S}_\bq(\tilde{\hbar}) ) \vert^2
  \big)
 \dx\dt.
\end{aligned}
\end{equation}
Now, since the initial data solve the equations \eqref{eta0eta1}--\eqref{tiledUNoTiledU}, we have to further estimate the right-hand side of \eqref{energyEstLinearTilde} above to get the right-hand side \eqref{timeEstimate}. 
%We give this in the following lemma.
%\begin{lemma}
%\label{lem:bigData}
%Suppose that the dataset
%$(g, \eta_0, \eta_\star, \bu_0, h, \mathbf{h},\mathbf{H})$
%satisfies the assumptions of Proposition \ref{prop:bigData}.
% Then
%\begin{equation}
%\begin{aligned}
%&\int_\omega\big( \vert \tilde{\eta}_1\vert^2
% +
% \vert \naby\tilde{\eta}_1\vert^2
% +
% \vert \Dely\tilde{\eta}_0\vert^2
% +
%  \vert \naby\Dely\tilde{\eta}_0\vert^2
%  \big)\dy
%  +
%  \int_\Omega \big(\vert
%  \tilde{\bu}_0\vert^2
%  +
%   \vert\nabx\tilde{\bu}_0\vert^2 \big)\dx
%   \\&
%%{\color{blue}
%%  +
%%  \int_\Omega \big(\vert
%%  \overline{\pi}(0)\vert^2
%%  +
%%   \vert\nabx \overline{\pi}(0)\vert^2 
%%     +
%%   \vert\nabx^2 \overline{\pi}(0)\vert^2 \big)\dx
%%}  
%+
%   \int_I\Vert \partial_t \tilde{h} \Vert_{W^{-1,2}(\Omega)}^2\dt
%   +
% \int_I\int_\omega\big( \vert \tilde{g}\vert^2 + \vert \naby \tilde{g}\vert^2 
%  \big)\dy\dt
%  \\&
%  +
%  \int_I\int_\Omega\big(
%  \vert \tilde{h}\vert^2
%  +
%  \vert \nabx \tilde{h}\vert^2
%  +
%   \vert \tilde{\mathbf{h}}\vert^2 +
%  \vert \tilde{\mathbf{H}}\vert^2
%  +
%  \vert \nabx\tilde{\mathbf{H}}\vert^2
%  \big)
% \dx\dt
% \\&\lesssim \mathcal{D}(g, \eta_0, \eta_\star, \bu_0, h, \mathbf{h},\mathbf{H})
%\end{aligned}
%\end{equation}
%where $\mathcal{D}(\cdot)$ is as defined in \eqref{bigDataNorm}.
%\end{lemma}
%%\todo{TO DOMINIC: The blue terms on the left are not needed and should be ideally deleted. 'Ideally' means that we find a RIGHT-HAND SIDE estimate for $ \Vert \overline{\pi}(0)\Vert_{W^{2,2}(\Omega)}^2$ in \eqref{estTildeEta1} below.
%%}
%\begin{proof}
\\
Since $\tilde{\eta}_0=\eta_\star$ and \eqref{tildePartialT} holds, clearly,
\begin{equation}
\begin{aligned}
\label{initialCondEstimate1}
&\int_\omega\big( 
 \vert \Dely\tilde{\eta}_0\vert^2
 +
  \vert \naby\Dely\tilde{\eta}_0\vert^2
  \big)\dy
  +
   \int_I\Vert \partial_t \tilde{h} \Vert_{W^{-1,2}(\Omega)}^2\dt
   +
 \int_I\int_\omega \vert \tilde{g}\vert^2 \dy\dt
  \\&+
  \int_I\int_\Omega\big(
  \vert \tilde{h}\vert^2
  +
  \vert \nabx \tilde{h}\vert^2
  +
   \vert \tilde{\mathbf{h}}\vert^2 +
  \vert \tilde{\mathbf{H}}\vert^2
  +
  \vert \nabx\tilde{\mathbf{H}}\vert^2
  +
  \vert \mathbb{S}_\bq(\tilde{\hbar}) ) \vert^2
  +
  \vert \nabx\mathbb{S}_\bq(\tilde{\hbar}) ) \vert^2
  \big)
 \dx\dt
 \\&\lesssim
\Vert \eta_\star\Vert_{W^{3,2}(\omega)}^2
  +
   \int_I\Vert \partial_t^2 h\Vert_{W^{-1,2}(\Omega)}^2\dt
   +
   \int_I 
   \Vert \partial_t g \Vert_{W^{1,2}(\omega)}^2\dt
%   +
% \Vert g\Vert_{W^{3,2}(\omega)}^2
  \\&+
  \int_I\big( 
 \Vert \partial_t h \Vert_{W^{1,2}(\Omega)}^2   
 +
   \Vert \partial_t\mathbf{h}\Vert_{L^2(\Omega)}^2 
    +
  \Vert \partial_t\mathbf{H}\Vert_{W^{1,2}(\Omega)}^2 
  +
  \Vert
  \mathbb{S}_\bq(\partial_t\overline{\hbar}) ) \Vert_{W^{1,2}(\Omega)}^2 \big)
\dt.
\end{aligned}
\end{equation}
It remains to estimate
\begin{equation}
\begin{aligned}
&\Vert \tilde{\eta}_\star\Vert_{W^{1,2}(\omega)}^2
  +
  \Vert
  \tilde{\bu}_0\Vert_{W^{1,2}(\Omega)}^2.
%+
%   \int_I\Vert \partial_t \tilde{h} \Vert_{W^{-1,2}(\Omega)}\dt
\end{aligned}
\end{equation}
%Now note that $g$ is continuous at time zero and so $\tilde{g}=\partial_t g$ belongs in $W^{1,2}(\omega)$ at $t=0$.
From \eqref{tiledEtaNoTiledEta}--\eqref{tiledUNoTiledU},
\begin{equation}
\begin{aligned}
\label{estTildeEta}
&\Vert \tilde{\eta}_\star\Vert_{W^{1,2}(\omega)}^2
  +
  \Vert
  \tilde{\bu}_0\Vert_{W^{1,2}(\Omega)}^2
 \\&\lesssim
\Vert \eta_\star\Vert_{W^{3,2}(\omega)}^2
 +
 \Vert \eta_0\Vert_{W^{5,2}(\omega)}^2
   +
 \Vert
  g(0)\Vert_{W^{1,2}(\omega)}^2
   +
 \Vert
  \overline{\bu}_0\Vert_{W^{3,2}(\Omega)}^2
  \\&+
   \Vert \overline{\pi}_0\Vert_{W^{2,2}(\Omega)}^2 +
   \Vert \mathbf{h}(0)\Vert_{W^{1,2}(\Omega)}^2 
   +
  \Vert \mathbf{H}(0)\Vert_{W^{2,2}(\Omega)}^2
  +
  \Vert \mathbb{S}_\bq(\overline{\hbar}(0))
  \Vert_{W^{2,2}(\Omega)}^2 
  \\&+
   \Vert \mathbf{H}(0)\Vert_{W^{1,2}(\partial\Omega)}^2 +
   \Vert \overline{\bu}_0\Vert_{W^{2,2}(\partial\Omega)}^2 +
  \Vert \overline{\pi}_0\Vert_{W^{1,2}(\partial\Omega)}^2
  +
  \Vert
  \mathbb{S}_\bq(\overline{\hbar}(0) )
  \Vert_{W^{1,2}(\partial\Omega)}^2
  .
\end{aligned}
\end{equation}
The last boundary terms can be estimated using the trace theorem to obtain
\begin{equation}
\begin{aligned}
\label{estPress222}
   \Vert \mathbf{H}(0)\Vert_{W^{1,2}(\partial\Omega)}^2 &+
   \Vert \overline{\bu}_0\Vert_{W^{2,2}(\partial\Omega)}^2 +
  \Vert \overline{\pi}(0)\Vert_{W^{1,2}(\partial\Omega)}^2
  +
  \Vert \mathbb{S}_\bq(\overline{\hbar}(0))
  \Vert_{W^{1,2}(\partial\Omega)}^2
  \\&\lesssim
   \Vert \mathbf{H}(0)\Vert_{W^{2,2}(\Omega)}^2 +
   \Vert \overline{\bu}_0\Vert_{W^{3,2}(\Omega)}^2 +
  \Vert \overline{\pi}(0)\Vert_{W^{2,2}(\Omega)}^2
  +
  \Vert \mathbb{S}_\bq(\overline{\hbar}(0))
  \Vert_{W^{2,2}(\Omega)}^2
%  \\&\lesssim
%   \Vert \mathbf{H}(0)\Vert_{W^{2,2}(\Omega)}^2 +
%   \Vert \overline{\bu}_0\Vert_{W^{3,2}(\Omega)}^2 +
%  \Vert \nabx^2\overline{\pi}(0)\Vert_{L^2(\Omega)}^2+
%  \Vert \nabx\overline{\pi}(0)\Vert_{L^2(\Omega)}^2
%  +(c_{\overline{\pi}}(0))^2
  .
\end{aligned}
\end{equation}
If we now combine \eqref{estTildeEta} and \eqref{estPress222}, we obtain
\begin{equation}
\begin{aligned}
\label{estTildeEta1}
&\Vert \tilde{\eta}_\star\Vert_{W^{1,2}(\omega)}^2
  +
  \Vert
  \tilde{\bu}_0\Vert_{W^{1,2}(\Omega)}^2
 \\&\lesssim
\Vert \eta_\star\Vert_{W^{3,2}(\omega)}^2
 +
 \Vert \eta_0\Vert_{W^{5,2}(\omega)}^2
   +
 \Vert
  g(0)\Vert_{W^{1,2}(\omega)}^2
   +
 \Vert
  \overline{\bu}_0\Vert_{W^{3,2}(\Omega)}^2
  \\&+
   \Vert \overline{\pi}_0\Vert_{W^{2,2}(\Omega)}^2 +
   \Vert \mathbf{h}(0)\Vert_{W^{1,2}(\Omega)}^2 
   +
  \Vert \mathbf{H}(0)\Vert_{W^{2,2}(\Omega)}^2
  +
  \Vert \mathbb{S}_\bq(\overline{\hbar}(0))
  \Vert_{W^{2,2}(\Omega)}^2.
\end{aligned}
\end{equation}
 To get suitable bounds for the pressure $\overline\pi_0$, we study \eqref{4'a}. In particular, to get $L^2$-bound for $\nabx\pi$, we consider
\begin{equation}
\begin{aligned}
\int_\Omega\mathbf{A}_{\eta_0}\nabx\overline{\pi}_0\cdot \nabx\overline{\pi} _0\dx
 &
=
\int_\Omega\divx(\mathbf{A}_{\eta_0}\nabx\overline{\pi}_0\,\overline{\pi}_0) \dx
-
\int_\Omega\divx(\mathbf{A}_{\eta_0}\nabx\overline{\pi}_0) \overline{\pi}_0 \dx,
\label{7}
\end{aligned}
\end{equation}
where, by ellipticity of $\mathbf A_{\eta_0}$,
\begin{equation}
\begin{aligned}
\int_\Omega \vert\nabx\overline{\pi}_0\vert^2 \dx
\lesssim
\int_\Omega\mathbf{A}_{\eta_0}\nabx\overline{\pi}_0\cdot \nabx\overline{\pi}_0 \dx
\label{7'}
\end{aligned}
\end{equation}
and
\begin{equation}
\begin{aligned}
\int_\Omega&\divx(\mathbf{A}_{\eta_0}\nabx\overline{\pi}_0\,\overline{\pi}_0) \dx
=
\int_{\partial\Omega} \mathbf{A}_{\eta_0}\nabx\overline{\pi}_0\,\overline{\pi}_0\cdot\bn\circ \bm{\varphi}^{-1}  \,\mathrm{d}\mathcal H^2
\\&=
-
\int_{\partial\Omega} J_{\eta_0} 
\bn^\top\circ \bm{\varphi}^{-1}  \mathbf{A}_{\eta_0}\overline{\pi}_0^2\cdot\bn\circ \bm{\varphi}^{-1}  \,\mathrm{d}\mathcal H^2
+
\int_{\partial\Omega}\overline\pi_0 J_{\eta_0} 
\bn^\top\circ \bm{\varphi}^{-1}  \mathbf{A}_{\eta_0}\mathbb{S}_\bq(\overline{\hbar}
 (0))\cdot\bn\circ \bm{\varphi}^{-1}  \,\mathrm{d}\mathcal H^2
\\&
+
\int_{\partial\Omega} \overline\pi_0
\mathbf{B}_{\eta_0}^\top J_{\eta_0}^{-1}[ \divx(\mathbf{A}_{\eta_0}  \nabx\overline{\bu}_0
+\mathbf{B}_{\eta_0}
\mathbb{S}_\bq(\overline{\hbar}
 (0)))
 +
\mathbf{h}(0) -
\divx  \mathbf{H}(0) 
]\cdot \bn\circ \bm{\varphi}^{-1}  \,\mathrm{d}\mathcal H^2
\\&
-
\int_{\partial\Omega} \overline\pi_0
 \mathbf{B}_{\eta_0}^\top
 \big[
\big(\{ \Dely\eta_\star - \Dely^2\eta_0+g(0)\}\bn\big)\circ \bm{\varphi}^{-1} 
+
\bn^\top\circ \bm{\varphi}^{-1} (\mathbf{H}(0) -\mathbf{A}_{\eta_0}  \nabx\overline{\bu}_0 ) \big]\cdot \bn\circ \bm{\varphi}^{-1}   \,\mathrm{d}\mathcal H^2
\label{8}
\end{aligned}
\end{equation}
with
\begin{align*}
-
\int_{\partial\Omega} J_{\eta_0} 
\bn^\top\circ \bm{\varphi}^{-1}  \mathbf{A}_{\eta_0}\overline{\pi}^2_0\bn\circ \bm{\varphi}^{-1}  \dy\leq 0
\end{align*}
and
\begin{equation}
\begin{aligned}
-
\int_\Omega&\divx(\mathbf{A}_{\eta_0}\nabx\overline{\pi}_0) \overline{\pi}_0 \dx
=
\int_\Omega\partial_t h(0)  \overline{\pi}_0 \dx
\\&
-
\int_{\partial\Omega}\overline\pi_0\mathbf{B}_{\eta_0}^\top J_{\eta_0}^{-1}\big[ \divx(\mathbf{A}_{\eta_0}  \nabx\overline{\bu}_0
+
\mathbf{B}_{\eta_0}\mathbb{S}_\bq(\overline{\hbar}
(0) ))
 +
\mathbf{h}(0) -
\divx  \mathbf{H}(0)
\big]\cdot \bn\circ \bm{\varphi}^{-1} \dy
\\
&
+
\int_\Omega \big(\mathbf{B}_{\eta_0}^\top J_{\eta_0}^{-1}[ \divx(\mathbf{A}_{\eta_0}  \nabx\overline{\bu}_0
+
\mathbf{B}_{\eta_0}\mathbb{S}_\bq(\overline{\hbar}(0)
 ))
 +
\mathbf{h}(0) -
\divx  \mathbf{H}(0)
]\big) \cdot\nabx\overline{\pi}_0 \dx.
\label{9}
\end{aligned}
\end{equation}
Therefore,
%by the compability condition \eqref{compatibilityCondition},
\begin{equation}
\begin{aligned}
\int_\Omega \vert\nabx\overline{\pi}_0\vert^2 \dx
&\leq
\int_\Omega \overline{\pi}_0\,\partial_t h(0)   \dx
+
\int_{\partial\Omega}\overline\pi_0 J_{\eta_0} 
\bn^\top\circ \bm{\varphi}^{-1}   \mathbf{A}_{\eta_0}\mathbb{S}_\bq(\overline{\hbar}(0))
 )\cdot\bn\circ \bm{\varphi}^{-1}  \,\mathrm{d}\mathcal H^2
 \\&
+
\int_\Omega \big(\mathbf{B}_{\eta_0}^\top J_{\eta_0}^{-1}[ \divx(\mathbf{A}_{\eta_0}  \nabx\overline{\bu}_0
+
\mathbf{B}_{\eta_0}\mathbb{S}_\bq(\overline{\hbar}
 (0)) 
 +
\mathbf{h}(0) -
\divx  \mathbf{H}(0) 
]\big) \nabx\overline{\pi}_0 \dx
\\&
-
\int_{\partial\Omega} \overline\pi_0
 \mathbf{B}_{\eta_0}^\top
\big[\{( \Dely\eta_\star - \Dely^2\eta_0+g(0))\bn\}\circ \bm{\varphi}^{-1} \big]
\cdot\bn\circ \bm{\varphi}^{-1}   \,\mathrm{d}\mathcal H^2\\
&+\int_{\partial\Omega} \overline\pi_0
 \mathbf{B}_{\eta_0}^\top
\big[
\bn^\top\circ \bm{\varphi}^{-1}  (\mathbf{H}(0) -\mathbf{A}_{\eta_0}  \nabx\overline{\bu}_0  ) \big]
\cdot\bn\circ \bm{\varphi}^{-1}   \,\mathrm{d}\mathcal H^2
\\&
\leq 
\delta
\int_\Omega \vert\nabx\overline{\pi}_0\vert^2 \dx
+
c(\delta)
  \Big(
  \Vert \partial_t h(0) \Vert_{W^{-1,2}(\Omega)}^2
 +
  \Vert \overline{\bu}_0\Vert^2_{W^{2,2}(\Omega)}
  +
   \Vert \mathbf{h}(0)\Vert^2_{L^2(\Omega)}  
   \Big)\\&+
  c(\delta)
  \Big(\Vert  \mathbf{H}(0)\Vert^2_{W^{1,2}(\Omega)}
  +
  \Vert
  \mathbb{S}_\bq(\overline{\hbar}(0) )\Vert^2_{W^{1,2}(\Omega)}
+
  \Vert\Dely\eta_\star\Vert^2_{L^2(\omega)}
 \Big) \\&+c(\delta)
  \Big(
   \Vert \Dely\eta_\star\Vert^2_{L^2(\omega)}  
  +\Vert \Dely^2\eta_0\Vert^2_{L^2(\omega)}  
  +
  \Vert g(0)\Vert^2_{L^2(\omega)}
  \Big),
\label{11}
\end{aligned}
\end{equation}
with a constant depending on the $W^{2,\infty}_\by$-norm of $\eta_0$. Note that we used again the trace theorem to estimate the boundary terms. Choosing $\delta$ small enough and using  Sobolev embedding yields
\begin{equation}
\begin{aligned}
\int_\Omega \vert\nabx\overline{\pi}_0\vert^2 \dx&
\lesssim  \Vert \partial_t h(0) \Vert_{W^{-1,2}(\Omega)}^2
 +
  \Vert \overline{\bu}_0\Vert^2_{W^{2,2}(\Omega)}
  +
   \Vert \mathbf{h}(0)\Vert^2_{L^2(\Omega)}\\
&  +\Vert  \mathbf{H}(0)\Vert^2_{W^{1,2}(\Omega)}
  +
  \Vert
  \mathbb{S}_\bq(\overline{\hbar}(0) )\Vert^2_{W^{1,2}(\Omega)}
   +
   \Vert \Dely\eta_\star\Vert^2_{L^2(\omega)}  
   \\&+
   \Vert \Dely^2\eta_0\Vert^2_{L^2(\omega)}
  +
  \Vert g(0)\Vert^2_{L^2(\omega)}\\
&\lesssim \mathcal{D}_*(g, \eta_0, \eta_\star, \bu_0, h, \mathbf{h},\mathbf{H}),
\label{11b}
\end{aligned}
\end{equation}
where
\begin{equation}
\begin{aligned}
\label{bigDataNorm*}
\mathcal{D}_*&(g, \eta_0, \eta_\star, \bu_0, h, \mathbf{h},\mathbf{H}):=
\Vert \eta_\star\Vert_{W^{3,2}(\omega)}^2
 +
 \Vert \eta_0\Vert_{W^{5,2}(\omega)}^2
  +
 \Vert
  \overline{\bu}_0\Vert_{W^{3,2}(\Omega)}^2
  +
  \Vert \mathbb{S}_\bq(\overline{\hbar}(0))
  \Vert_{W^{2,2}(\Omega)}^2
   \\&
    +
   \Vert \mathbf{h}(0)\Vert_{W^{1,2}(\Omega)}^2 
   +
  \Vert \mathbf{H}(0)\Vert_{W^{2,2}(\Omega)}^2
   +
   \int_I\Vert \partial_t^2 h \Vert_{W^{-1,2}(\Omega)}^2\dt 
   +
\Vert \partial_t h(0) \Vert_{L^2(\Omega )}^2
    +
 \Vert
  g(0)\Vert_{W^{1,2}(\omega)}^2
   \\&+
   \int_I \big(\Vert \partial_t g \Vert_{W^{1,2}(\omega)}^2
 +
 \Vert \partial_t h \Vert_{W^{1,2}(\Omega)}^2  
 +
 \Vert \partial_t\mathbf{h}\Vert_{L^2(\Omega)}^2
 +
\Vert
\mathbb{S}_\bq(\partial_t\overline{\hbar} )
\Vert_{W^{1,2}(\Omega)}^2
 +
  \Vert \partial_t\mathbf{H}\Vert_{W^{1,2}(\Omega)}^2
  \big)
\dt.
\end{aligned}
\end{equation}
In order to control the $W^{1,2}_\bx$-norm of $\overline{\pi}_0$, we must also control its $L^2_\bx$-norm. 
Let us write $\overline\pi_0=\overline{\pi}_{0}^\perp+c_0$, where $(\overline\pi_{0}^\perp)_{\Omega}=0$. We obtain from \eqref{compatibilityCondition} (multiplying it by $\bfn$)
\begin{align*}
|c_0|&\lesssim \int_{\partial\Omega}\bfn^\top\mathbf B_{\eta_0}\mathbf B_{\eta_0}^\top\bfn |c_0|\dd\mathcal{H}^2
\\&\lesssim\|\Dely  \eta_\star\|_{L^2(\omega)} +  \|\Dely^2 \eta_0\|_{L^2(\omega)}
+
\|g(0) \|_{L^2(\omega)}+
\|\overline{\bu}_0\|_{W^{2,2}(\partial\Omega)}\\&
+\|\overline{\pi}^\perp_0\|_{W^{1,2}(\partial\Omega)}+
 \|\mathbb{S}_\bq(\overline{\hbar}(0) )\|_{W^{1,2}(\partial\Omega)}
 + 
\|\mathbf{h}(0)\|_{L^2(\partial\Omega)}+\|\mathbf{H}(0) \|_{W^{1,2}(\partial\Omega)}
\end{align*}
using ellipticity of $\mathbf B_{\eta_0}\mathbf B_{\eta_0}^\top$ as well as $\Div\mathbf B_{\eta_0}^\top=0$. By the trace theorem and interpolation, we finally get
\begin{align}\label{11b'}\begin{aligned}
|c_0|\
&\lesssim\|\Dely  \eta_\star\|_{L^2(\omega)} +  \|\Dely^2 \eta_0\|_{L^2(\omega)}
+
\|g(0) \|_{L^2(\omega)}+
\|\overline{\bu}_0\|_{W^{2,2}(\partial\Omega)}\\&+
 \|\mathbb{S}_\bq(\overline{\hbar}(0) )\|_{W^{2,2}(\Omega)}
 + 
\|\mathbf{h}(0)\|_{W^{1,2}(\Omega)}+\|\mathbf{H}(0) \|_{W^{2,2}(\Omega)}\\
&+c(\delta)\|\nabx\overline{\pi}_0\|_{L^{2}(\Omega)}+\delta \|\nabx^2\overline{\pi}_0\|_{L^{2}(\Omega)}\\
&\lesssim c(\delta)\mathcal{D}_*(g, \eta_0, \eta_\star, \bu_0, h, \mathbf{h},\mathbf{H})+\delta \|\nabx^2\overline{\pi}_0\|_{L^{2}(\Omega)}.
\end{aligned}
\end{align}
On the other hand, setting
$\underline{\pi}_0=\overline{\pi}_0\circ\bm{\Psi}_{\eta_0}^{-1}$ %and $\underline{\varphi}=\varphi\circ\bm{\Psi}_{\eta_0}^{-1}$
we obtain
from \eqref{4'a} the elliptic problem
\begin{equation}
\begin{aligned}
 \Delx\underline{\pi}_0
 &
=
J_{\eta_0}f_0\circ\bm{\Psi}_{\eta_0}^{-1}\quad\text{in}\quad \Omega_{\eta_0}
\end{aligned}
\end{equation}
subject to the boundary condition
\begin{equation}
\begin{aligned}
 \partial_{\bn_{\eta_0}}\underline{\pi}_0=J_{\eta_0}f_1\circ\bm{\Psi}_{\eta_0}^{-1}\quad\text{on}\quad \partial\Omega_{\eta_0},
\end{aligned}
\end{equation}
where
\begin{equation}
\begin{aligned}
f_0:=& - 
\partial_th(0)
 + \divx\big(\mathbf{B}_{\eta_0}^\top J_{\eta_0}^{-1}\big\{ \divx(\mathbf{A}_{\eta_0}  \nabx\overline{\bu}_0
 +
 \mathbf{B}_{\eta_0}\mathbb{S}_\bq(\overline{\hbar}(0) )) 
 +
\mathbf{h}(0) -
\divx  \mathbf{H}(0) 
\}\big),
\\
f_1
 :=&-
J_{\eta_0} 
\bn^\top\circ \bm{\varphi}^{-1}    \mathbf{A}_{\eta_0}(\overline{\pi}_0
+
\mathbb{S}_\bq(\overline{\hbar}(0) ))
\cdot \bn \circ \bm{\varphi}^{-1}  
\\&+
\mathbf{B}_{\eta_0}^\top J_{\eta_0}^{-1}\big\{ \divx(\mathbf{A}_{\eta_0}  \nabx\overline{\bu}_0
+
\mathbf{B}_{\eta_0}\mathbb{S}_\bq(\overline{\hbar}(0) )) 
 +
\mathbf{h}(0) -
\divx  \mathbf{H}(0) 
\}\cdot \bn\circ \bm{\varphi}^{-1}  
\\
&
-
 \mathbf{B}_{\eta_0}^\top\big[
\big\{\Dely\eta_\star - \Dely^2\eta_0+g(0)\big\}\bn\big]\circ \bm{\varphi}^{-1} \cdot \bn\circ \bm{\varphi}^{-1}   
\\&-
\mathbf{B}_{\eta_0}^\top
\bn^\top\circ \bm{\varphi}^{-1}   \big[\mathbf{H}(0) -\mathbf{A}_{\eta_0}  \nabx\overline{\bu}_0  \big]
\cdot \bn\circ \bm{\varphi}^{-1}  .
\end{aligned}
\end{equation}
%which is equivalent to
%\begin{equation}
%\begin{aligned}
%-\int_{\Omega_{\eta_0}} \nabx\underline{\pi} \cdot \nabx\underline{\varphi}\dx
% &
%=
% \int_{\Omega_{\eta_0}} J_{\eta_0}f_0\circ\bm{\Psi}_{\eta_0}^{-1}\underline{\varphi}\dx 
%+ 
% \int_{\partial\Omega_{\eta_0}}\underline{\varphi}
% J_{\eta_0}f_1\circ\bm{\Psi}_{\eta_0}^{-1}\dy
%\end{aligned}
%\end{equation}
%after the transformation $\underline{\pi}=\overline{\pi}\circ\bm{\Psi}_{\eta_0}^{-1}$ and $\underline{\varphi}=\varphi\circ\bm{\Psi}_{\eta_0}^{-1}$.
 By classical elliptic estimates, it follows that
\begin{equation}
\begin{aligned}
\Vert \nabx^2\underline{\pi}\Vert_{L^2(\Omega_{\eta_0})}
 &
\lesssim
\Vert
 J_{\eta_0}f_0\circ\bm{\Psi}_{\eta_0}^{-1}
\Vert_{L^2(\Omega_{\eta_0})}
+
 \Vert J_{\eta_0}f_1\circ\bm{\Psi}_{\eta_0}^{-1}
\Vert_{W^{1/2,2}(\partial\Omega_{\eta_0})}
\end{aligned}
\end{equation}
which implies that
\begin{equation}
\begin{aligned}
\Vert \nabx^2\overline{\pi}\Vert_{L^2(\Omega )}^2
 &
\lesssim
\Vert
 f_0 \Vert_{L^2(\Omega )}^2
+
 \Vert f_1
\Vert_{W^{1/2,2}(\omega )}^2
\\
&
\lesssim
\Vert \overline{\pi}_0\Vert_{W^{1,2}(\Omega)}^2
+
\Vert \partial_t h(0) \Vert_{L^2(\Omega )}^2
+
\Vert \overline{\bu}_0 \Vert_{W^{3,2}(\Omega )}^2
\\&
+
\Vert \mathbf{h}(0) \Vert_{W^{1,2}(\Omega )}^2
+
\Vert \mathbf{H}(0)\Vert_{W^{2,2}(\Omega )}^2
+
\Vert
\mathbb{S}_\bq(\overline{\hbar}(0) )
\Vert_{W^{2,2}(\Omega )}^2
\\&
+
 \Vert\Dely\eta_\star\Vert^2_{W^{1,2}(\omega)}
  +
   \Vert \Dely^2\eta_0\Vert^2_{W^{1,2}(\omega)}  
  +
  \Vert g(0)\Vert^2_{W^{1,2}(\omega)},
  \label{12}
\end{aligned}
\end{equation}
where we have used the trace theorem. If we now combine \eqref{11b}, \eqref{11b'} (with $\delta$ sufficiently small) and \eqref{12}, and transform back to the reference domain (using the regularity of $\eta_0$),
we get that
\begin{align}
\label{estPressZero}
\Vert \overline{\pi}_0\Vert_{W^{2,2}(\Omega)}^2\lesssim
\mathcal{D}_*(g, \eta_0, \eta_\star, \bu_0, h, \mathbf{h},\mathbf{H}).
\end{align}
Combining \eqref{estTildeEta1}, \eqref{initialCondEstimate1} and \eqref{estPressZero}, gives 
\begin{equation}
\begin{aligned}
\label{estForEvolvingInitialCondition}
&\int_\omega\big( \vert \tilde{\eta}_\star\vert^2
 +
 \vert \naby\tilde{\eta}_\star\vert^2
 +
 \vert \Dely\tilde{\eta}_0\vert^2
 +
  \vert \naby\Dely\tilde{\eta}_0\vert^2
  \big)\dy
  +
  \int_\Omega \big(\vert
  \tilde{\bu}_0\vert^2
  +
   \vert\nabx\tilde{\bu}_0\vert^2 \big)\dx
   \\&
%{\color{blue}
%  +
%  \int_\Omega \big(\vert
%  \overline{\pi}(0)\vert^2
%  +
%   \vert\nabx \overline{\pi}(0)\vert^2 
%     +
%   \vert\nabx^2 \overline{\pi}(0)\vert^2 \big)\dx
%}  
+
   \int_I\Vert \partial_t \tilde{h} \Vert_{W^{-1,2}(\Omega)}^2\dt
   +
 \int_I\int_\omega  \vert \tilde{g}\vert^2 \dy\dt
  \\&
  +
  \int_I\int_\Omega\big(
  \vert \tilde{h}\vert^2
  +
  \vert \nabx \tilde{h}\vert^2
  +
   \vert \tilde{\mathbf{h}}\vert^2 +
  \vert \tilde{\mathbf{H}}\vert^2
  +
  \vert \nabx\tilde{\mathbf{H}}\vert^2
  +
  \vert \mathbb{S}_\bq(\tilde{\hbar}) ) \vert^2
  +
  \vert \nabx\mathbb{S}_\bq(\tilde{\hbar}) ) \vert^2
  \big)
 \dx\dt
 \\&\lesssim \mathcal{D}_*(g, \eta_0, \eta_\star, \bu_0, h, \mathbf{h},\mathbf{H})
\end{aligned}
\end{equation}
where $\mathcal{D}_*(\cdot)$ is as defined in \eqref{bigDataNorm*}. We obtain
\begin{equation}
\begin{aligned}
\label{energyEstLinearTildeB}
&\sup_I\int_\omega
\big(\vert \partial_t^2\naby \eta\vert^2 
+
\vert \partial_t\naby\Dely \eta\vert^2
\big)
\dy
+
\sup_I\int_\Omega\vert\partial_t\nabx \bu\vert^2\dx
\\&+
\int_I\int_\omega
\big(\vert \partial_t^2\Dely \eta \vert^2 + \vert \partial_t^3 \eta\vert^2+ \vert \partial_t\Dely^2 \eta\vert^2
 \big)\dy\dt\\
& +
\int_I\int_\Omega\big( \vert \partial_t\nabx^2\bu\vert^2 +\vert \partial_t^2\bu \vert^2  +\vert \partial_t\pi\vert^2 + \vert \partial_t\nabx \pi\vert^2
 \big)\dx\dt
 \\&\lesssim
\mathcal{D}_*(g, \eta_0, \eta_\star, \bu_0, h, \mathbf{h},\mathbf{H}).
\end{aligned}
\end{equation}
%
%\subsection{Higher spatial regularity}
%Our next goal is to combine the higher time regularity estimate derived the subsection above with a maximal-in-space regularity to obtain an optimal spacetime estimate for the fluid unknowns. To obtain this higher spatial regularity, we apply the maximal regularity theorem to the momentum equation rather than to differentiate the equations in our fluid system with respect to the spatial variable like it was done for the time regularity. Our main result is the following.
\\
We now proceed to obtain   the maximal-in-space regularity estimate for the velocity and pressure pair, i.e. the $L^2$-in-time estimate for the terms $\Vert\overline{\bu}\Vert_{W^{4,2}(\Omega )}^2$ and  $\Vert\overline{\pi} \Vert_{W^{3,2}(\Omega )}^2$. To obtain this higher spatial regularity, we apply the maximal regularity theorem to the momentum equation rather than differentiate the equations in our fluid system with respect to the spatial variable like it was done for the time regularity. 
%Our main result is the following.
%\begin{proposition}
%\label{prop:highestEst}
%Suppose that the dataset
%$(g, \eta_0, \eta_\star, \bu_0, h, \mathbf{h},\mathbf{H})$
%satisfies the assumptions of Proposition \ref{prop:bigData}.
% Then
%\begin{equation}
%\begin{aligned}
%\label{highestEst}
%&\sup_I\int_\omega
%\big(\vert \partial_t^2\naby \eta\vert^2 
%+
%\vert \partial_t\naby\Dely \eta\vert^2
%\big)
%\dy
%+
%\sup_I\int_\Omega\vert \partial_t\nabx \overline{\bu}\vert^2\dx
%\\&+
%\int_I\big(\Vert\overline{\bu}
% \Vert_{W^{4,2}(\Omega )}^2
% +
% \Vert\overline{\pi} \Vert_{W^{3,2}(\Omega )}^2\big) \dt
%+
%\int_I\int_\omega
%\big(\vert \partial_t^2\Dely \eta \vert^2 + \vert \partial_t^3 \eta\vert^2
% \big)\dy\dt
% \\&+
%\int_I\int_\Omega\big( \vert \partial_t \nabx^2\overline{\bu}\vert^2 +\vert \partial_t^2\overline{\bu} \vert^2  +\vert \partial_t\overline{\pi}\vert^2 + \vert \partial_t\nabx \overline{\pi}\vert^2
% \big)\dx\dt
%\lesssim
% \mathcal{D}(g, \eta_0, \eta_\star, \bu_0, h, \mathbf{h},\mathbf{H}).
%\end{aligned}
%\end{equation}
%\end{proposition}
%{\color{blue}\todo{Note to self}
%This estimate above may allow us to absorb the time derivative of the various h-terms embedded in $ \mathcal{D}(\cdot)$  into the left-hand side once we move from the linear equation to the equation where the h terms depends on the solution. Note that the only term truly important to us is $\int\Vert\overline{\bu}
% \Vert_{W^{4,2}(\Omega )}^2\dt$
%}
%\begin{proof}
First of all, we transform
\eqref{contEqAloneBarLinear} and \eqref{momEqAloneBarLinear} by applying $\bm{\Psi}_{\eta_0}^{-1}$ to them. By setting $\underline{\bu}:=\overline{\bu} \circ \bm{\Psi}_{\eta_0}^{-1}$,  $\underline{\pi}:=\overline{\pi} \circ \bm{\Psi}_{\eta_0}^{-1}$ and $\underline{\hbar}:=\overline{\hbar} \circ \bm{\Psi}_{\eta_0}^{-1}$, we obtain
\begin{align*}
&\divx \underline{\bu}=J_{\eta_0}^{-1}h\circ\bm{\Psi}_{\eta_0}^{-1},
\\&
  \Delx\underline{\bu}
 - \nabx \underline{\pi} 
 = 
\partial_t\underline{\bu}
+\divx(
\mathbb{S}_\bq(\underline{\hbar} )
)  -
 J_{\eta_0}^{-1}\big(
\mathbf{h}-
\divx  \mathbf{H} \big)\circ \bm{\Psi}_{\eta_0}^{-1}
\end{align*}
in $I \times \Omega_{\eta_0}$ with  $\underline{\bu}  \circ \bm{\varphi}_{\eta_0}  =(\partial_t\eta)\bn$   on $I\times \omega$. By maximal regularity theory, it follows that
\begin{align*}
\int_I\big(\Vert\underline{\bu}
 \Vert_{W^{4,2}(\Omega_{\eta_0})}^2
 +
 \Vert\underline{\pi} \Vert_{W^{3,2}(\Omega_{\eta_0})}^2\big) \dt
& \lesssim
 \int_I\Vert \partial_t\eta\Vert_{W^{7/2,2}_\by}^2\dt
 +
 \int_I \big(\Vert\partial_t\underline{\bu}
 \Vert_{W^{2,2}(\Omega_{\eta_0})}^2
 +
 \Vert\divx\mathbb{S}_\bq(\underline{\hbar} )
 \Vert_{W^{2,2}(\Omega_{\eta_0})}^2\big)
   \dt
 \\
 &+
 \int_I 
 \big(\Vert J_{\eta_0}^{-1}
h\circ \bm{\Psi}_{\eta_0}^{-1}
\Vert_{W^{3,2}(\Omega_{\eta_0})}^2
+
\Vert
J_{\eta_0}^{-1}\mathbf{h}\circ \bm{\Psi}_{\eta_0}^{-1}
\Vert_{W^{2,2}(\Omega_{\eta_0})}^2)\dt
\\&+\int_I
\Vert J_{\eta_0}^{-1}(\divx \mathbf{H} )\circ \bm{\Psi}_{\eta_0}^{-1}\Vert_{W^{2,2}(\Omega_{\eta_0})}^2\dt.
 \end{align*}
If we now transform back to $\Omega$ and use that $\eta_0\in W^{5,2}(\omega)$ we obtain
\begin{equation}
\begin{aligned}
\label{maxMaximalReg}
\int_I\big(\Vert\overline{\bu}
 \Vert_{W^{4,2}(\Omega )}^2
 +
 \Vert\overline{\pi} \Vert_{W^{3,2}(\Omega )}^2\big) \dt
 &\lesssim
 \int_I\Vert \partial_t\eta\Vert_{W^{4,2}_\by}^2\dt
 +
 \int_I\big( \Vert\partial_t\overline{\bu}
 \Vert_{W^{2,2}(\Omega )}^2
  +
 \Vert\nabx\mathbb{S}_\bq(\overline{\hbar} )
 \Vert_{W^{2,2}(\Omega)}^2\big)
   \dt
 \\
 &+
 \int_I 
 \big(
 \Vert
h \Vert_{W^{3,2}(\Omega )}^2
+\Vert
\mathbf{h} \Vert_{W^{2,2}(\Omega )}^2
+
\Vert \nabx \mathbf{H}  \Vert_{W^{2,2}(\Omega )}^2\big)\dt.
 \end{aligned}
 \end{equation}
Now take
\begin{equation}
\begin{aligned}
\label{nowTake}
\partial_t\Dely^2\eta
=
-\partial_t^3\eta + \partial_t^2\Dely \eta
+
\partial_t g+\bn \big[\partial_t\mathbf{H} -\mathbf{A}_{\eta_0} \partial_t \nabx\overline{\bu} +\mathbf{B}_{\eta_0}(\partial_t\overline{\pi}
-
\mathbb{S}_\bq(\partial_t\overline{\hbar} ))\big]\circ\bm{\varphi} \bn 
\end{aligned}
\end{equation}
which is just \eqref{shellEqAloneBarLinearTime}.
By using \eqref{energyEstLinearTilde} (and \eqref{estForEvolvingInitialCondition}) as well as
\begin{equation}
\begin{aligned}
\int_I&\int_{\partial\Omega}\big(
\vert\partial_t\mathbf{H} \vert^2
+
\vert \partial_t \nabx\overline{\bu}\vert^2
 +
 \vert \partial_t\overline{\pi}\vert^2
 +
\vert \mathbb{S}_\bq(\partial_t\overline{\hbar} )\vert^2 
 \big)
 \dy\dt
 \\&
 \lesssim
 \int_I\big(
\Vert\partial_t\mathbf{H} \Vert_{W^{1,2}(\Omega)}^2
+
\Vert \partial_t \nabx\overline{\bu}\Vert_{W^{1,2}(\Omega)}^2
 +
 \Vert \partial_t\overline{\pi}\Vert_{W^{1,2}(\Omega)}^2
+
\Vert
\mathbb{S}_\bq(\partial_t\overline{\hbar} )
\Vert_{W^{1,2}(\Omega)}^2
 \big)
 \dy\dt
\end{aligned}
\end{equation}
which follows from the trace theorem, we obtain from \eqref{energyEstLinearTildeB} and \eqref{nowTake}
\begin{equation}
\begin{aligned}
\int_I\int_\omega\vert\partial_t\Dely^2\eta
\vert^2\dy\dt
&\lesssim
\int_I\int_\omega\big(\vert\partial_t^3\eta\vert^2 + \vert\partial_t^2\Dely \eta
\vert^2
+
\vert
\partial_t g
\vert^2\big)
 \dy\dt
 \\&
+
\int_I\int_{\partial\Omega}\big(
\vert\partial_t\mathbf{H} \vert^2
+
\vert \partial_t \nabx\overline{\bu}\vert^2
 +
 \vert \partial_t\overline{\pi}\vert^2
  +
\vert \mathbb{S}_\bq(\partial_t\overline{\hbar} )\vert^2\big)
 \dy\dt
\\&\lesssim
 \mathcal{D}_{**}(g, \eta_0, \eta_\star, \bu_0, h, \mathbf{h},\mathbf{H})
\end{aligned}
\end{equation}
where $\mathcal{D}_{**}(\cdot)$ is given by
\begin{equation}
\begin{aligned}
\label{bigDataNorm**}
\mathcal{D}_{**}&(g, \eta_0, \eta_\star, \bu_0, h, \mathbf{h},\mathbf{H}):=
\mathcal{D}_*(g, \eta_0, \eta_\star, \bu_0, h, \mathbf{h},\mathbf{H})
   \\&+
   \int_I \big(
 \Vert   h \Vert_{W^{3,2}(\Omega)}^2  
 +
   \Vert \mathbf{h}\Vert_{W^{2,2}(\Omega)}^2
 +
\Vert
\mathbb{S}_\bq(\overline{\hbar} )
\Vert_{W^{3,2}(\Omega)}^2 
+
  \Vert \mathbf{H}\Vert_{W^{3,2}(\Omega)}^2
  \big)
\dt
\end{aligned}
\end{equation}
with $\mathcal{D}_*(\cdot)$ given by \eqref{bigDataNorm*}.
Substituting this into \eqref{maxMaximalReg} and using \eqref{energyEstLinearTildeB} again to estimate the term involving $\Vert\partial_t\overline{\bu}
 \Vert_{W^{2,2}(\Omega )}^2$ yield
\begin{equation}
\begin{aligned}
\label{moreMaximal}
\int_I&\big(\Vert\overline{\bu}
 \Vert_{W^{4,2}(\Omega )}^2
 +
 \Vert\overline{\pi} \Vert_{W^{3,2}(\Omega )}^2\big) \dt
 \lesssim 
 \mathcal{D}_{**}(g, \eta_0, \eta_\star, \bu_0, h, \mathbf{h},\mathbf{H}).
 \end{aligned}
 \end{equation}
Using regularity theory for equation \eqref{shellEqAloneBarLinear} (recall that we consider periodic boundary conditions) and setting
$$\overline{\mathbb T}:=\mathbf{H} -\mathbf{A}_{\eta_0}  \nabx\overline{\bu} +\mathbf{B}_{\eta_0}(\overline{\pi}\,\mathbb I_{3\times 3} -  \mathbb{S}_\bq(\overline{\hbar}))$$
we have
\begin{align*}
\sup_I\|\partial_t\eta\|_{W^{3,2}(\omega)}^2&+\sup_I\|\eta\|_{W^{5,2}(\omega)}^2+\int_I\big(\|\partial_t\eta\|^2_{W^{4,2}}+\|\eta\|_{W^{6,2}(\omega)}^2\big)\dt\\
&\lesssim \int_I\|g\|^2_{W^{2,2}(\omega)}\dt+\int_I\|\overline{\mathbb T}\|^2_{W^{2,2}(\partial\Omega)}\dt\\
&\lesssim \int_I\|g\|^2_{W^{2,2}(\omega)}\dt+\int_I\|\overline{\mathbb T}\|^2_{W^{3,2}(\Omega)}\dt\\
&\lesssim \int_I\|g\|^2_{W^{2,2}(\omega)}\dt+\int_I\Big(\|\overline\bfu\|^2_{W^{4,2}(\Omega)}+\|\overline\pi\|^2_{W^{3,2}(\Omega)}+\|\mathbb S_{\mathbf q}(\overline\hbar)\|^2_{W^{3,2}(\Omega)}\Big)\dt\\
& \lesssim
 \mathcal{D}(g, \eta_0, \eta_\star, \bu_0, h, \mathbf{h},\mathbf{H})
\end{align*}
with $\mathcal{D}(\cdot)$ defined in \eqref{bigDataNorm}.
The proof is now complete.
%
% {\color{blue}\todo{Note to self}
% The last-but-one terms above containing the various h and the time derivatives cannot be absorbed into the left-hand side when the depend on the solution}
% 
%Combing this estimate with \eqref{energyEstLinearTilde} and \eqref{estForEvolvingInitialCondition} yields our desired result.
\end{proof}
Interpolating between \eqref{energyEstLinear} and Proposition \ref{prop:bigData} with interpolation parameter\footnote{One can certainly obtain a corresponding result for any interpolation parameter in $(0,1)$.} $1/2$, we obtain the following corollary.
\begin{corollary}
\label{cor:bigData}
Suppose that the dataset
$(g, \eta_0, \eta_\star, \overline{\bu}_0,\overline{\hbar}, h, \mathbf{h},\mathbf{H})$
satisfies \eqref{datasetAlone} and in addition
\begin{equation}
\begin{aligned}
\label{bigData'}
&g\in L^2(I;W^{1,2}(\omega)) \cap W^{1/2,2}(I;W^{1/2,2}(\omega)) ,\quad  g(0)\in W^{1/2,2}(\omega),
\\&\eta_0 \in W^{4,2}(\omega) \text{ with } \Vert \eta_0 \Vert_{L^\infty( \omega)} < L, \quad \eta_\star \in W^{2,2}(\omega),
\\&\overline{\bu}_0 \in W^{2,2}(\Omega),
\quad \overline{\bu}_0\circ \bm{\varphi}  =\eta_\star\bn, \quad \mathbf{B}_{\eta_0}:\nabx \overline{\bu}_0=h,
\\&
h\in L^2(I;W^{2,2}(\Omega)) \cap W^{1,2}(I;L^{2}(\Omega))  \cap W^{3/2,2}(I;W^{-1,2}(\Omega)) 
%\cap \{h(0,\cdot)=0\}
, 
\\&
\mathbf{h} \in L^2(I;W^{1,2}(\Omega))
 \cap W^{1/2,2}(I;L^2(\Omega)) 
 ,\quad 
 \mathbf{h}(0)\in  W^{1/2,2}(\Omega),
 \\&
 \mathbf{H}\in L^2(I;W^{2,2}(\Omega)) \cap W^{1/2,2}(I;W^{1,2}(\Omega)),
 \quad
 \mathbf{H}(0)\in  W^{1,2}(\Omega),
 \\&
\mathbb{S}_\bq(\overline{\hbar} )\in L^2(I;W^{2,2}(\Omega)) \cap W^{1/2,2}(I;W^{1,2}(\Omega)),
 \quad
 \mathbb{S}_\bq(\overline{\hbar} (0))\in  W^{1,2}(\Omega),
\end{aligned}
\end{equation}
with the compatibility condition
\eqref{compatibilityCondition}.
 Then a strong solution $( \eta,  \overline{\bu},  \overline{\pi})$ of \eqref{shellEqAloneBarLinear}--\eqref{contEqAloneBarLinear} satisfies
 \begin{equation}
\begin{aligned}
\label{timeEstimate'}
&\sup_I\int_\omega
\big(
\vert \partial_t\naby^3 \eta\vert^2+
\vert \naby^4 \eta\vert^2
\big)
\dy
\\&+
\int_I\big(\Vert\overline{\bu}
 \Vert_{W^{3,2}(\Omega )}^2
 +
 \Vert\overline{\pi} \Vert_{W^{2,2}(\Omega )}^2\big) \dt
+\|\eta\|^2_{W^{5/2,2}(I;L^2(\omega))}+\|\eta\|^2_{W^{3/2,2}(I;W^{2,2}(\omega))}
\\&+\int_I\int_\omega
\big(\vert \partial_t\naby^3 \eta \vert^2
 +\vert\naby^5\eta\vert^2\big)\dy\dt
+
\|\bfu\|_{W^{1/2,2}(I;W^{2,2}(\Omega))}^2\\&+\|\partial_t\overline\bfu\|_{W^{1/2,2}(I;L^2(\Omega))}^2+\|\overline\pi\|_{W^{1/2,2}(I;W^{1,2}(\Omega))}^2
\lesssim
 \mathcal{D}(g, \eta_0, \eta_\star, \bu_0, h, \mathbf{h},\mathbf{H}),
\end{aligned}
\end{equation}
where
\begin{equation}
\begin{aligned}
\label{bigDataNorm'}
\mathcal{D}&(g, \eta_0, \eta_\star, \overline\bu_0, h, \mathbf{h},\mathbf{H}):=
\Vert \eta_\star\Vert_{W^{2,2}(\omega)}^2
 +
 \Vert \eta_0\Vert_{W^{4,2}(\omega)}^2
  +
 \Vert
  \overline{\bu}_0\Vert_{W^{2,2}(\Omega)}^2
  +
  \Vert \mathbb{S}_\bq(\overline{\hbar}(0))
  \Vert_{W^{1,2}(\Omega)}^2
   \\&
    +
   \Vert \mathbf{h}(0)\Vert_{W^{1/2,2}(\Omega)}^2 
   +
  \Vert \mathbf{H}(0)\Vert_{W^{1,2}(\Omega)}^2
    +
 \Vert
  g(0)\Vert_{W^{1/2,2}(\omega)}^2
   \\&+
   \int_I \big(
 \Vert g\Vert_{W^{1,2}(\omega)}^2 
 +
 \Vert \partial_t h \Vert_{L^{2}(\Omega)}^2   +
 \Vert   h \Vert_{W^{2,2}(\Omega)}^2  
\big)\dt+\Vert g \Vert_{W^{1/2,2}(I;W^{1/2,2}(\omega))}^2\\
&
   +  
   \Vert h \Vert_{W^{3/2}(I;W^{-1,2}(\Omega))}^2
   +
 \Vert \mathbf{h}\Vert_{W^{1/2,2}(I;L^2(\Omega))}^2+ 
  \Vert \mathbf{H}\Vert_{W^{1/2,2}(I;W^{1,2}(\Omega))}^2
  +
  \Vert
\mathbb{S}_\bq(\overline{\hbar} )
\Vert_{W^{1/2,2}(I;W^{1,2}(\Omega))}^2
\\&+
  \int_I\big(
   \Vert \mathbf{h}\Vert_{W^{1,2}(\Omega)}^2
  +
\Vert
\mathbb{S}_\bq(\overline{\hbar} )
\Vert_{W^{2,2}(\Omega)}^2
+
  \Vert \mathbf{H}\Vert_{W^{2,2}(\Omega)}^2
  \big)
\dt
\end{aligned}
\end{equation}
\end{corollary} 
Taking into account estimate \eqref{energyEstLinearTildeB} we also obtain the following.
\begin{corollary}\label{cor:bigData'}
Suppose that the dataset 
$(g, \eta_0, \eta_\star, \overline{\bu}_0,\overline{\hbar}, h, \mathbf{h},\mathbf{H})$
satisfies \eqref{datasetAlone} and in addition
\begin{equation}
\begin{aligned}
\label{bigData''}
&g\in  W^{1,2}(I;W^{1,2}(\omega)) ,\quad  
%\textcolor{red}{g(0)\in W^{1,2}(\omega)},
\\&\eta_0 \in W^{5,2}(\omega) \text{ with } \Vert \eta_0 \Vert_{L^\infty( \omega)} < L, \quad \eta_\star \in W^{3,2}(\omega),
\\&\overline{\bu}_0 \in W^{3,2}(\Omega),
\quad \overline{\bu}_0\circ \bm{\varphi} =\eta_\star\bn, \quad \mathbf{B}_{\eta_0}:\nabx \overline{\bu}_0=h,
\\&
h\in W^{1,2}(I;W^{1,2}(\Omega))  \cap W^{2,2}(I;W^{-1,2}(\Omega)) ,\quad \partial_t h(0)\in W^{1,2}(\Omega),
%\cap \{h(0,\cdot)=0\}
, 
\\&
\mathbf{h} \in  W^{1,2}(I;L^2(\Omega)) 
 ,\quad 
 \mathbf{h}(0)\in  W^{1,2}(\Omega),
 \\&
 \mathbf{H}\in W^{1,2}(I;W^{1,2}(\Omega)),
 \quad
 \mathbf{H}(0)\in  W^{2,2}(\Omega),
 \\&
\mathbb{S}_\bq(\overline{\hbar} )\in  W^{1,2}(I;W^{1,2}(\Omega)),
 \quad
 \mathbb{S}_\bq(\overline{\hbar} (0))\in  W^{2,2}(\Omega),
\end{aligned}
\end{equation}
with the compatibility condition
\eqref{compatibilityCondition}. Then a strong solution $( \eta,  \overline{\bu},  \overline{\pi})$ of \eqref{shellEqAloneBarLinear}--\eqref{contEqAloneBarLinear} satisfies%\todo{remove the red terms?}
 \begin{align}
\label{energyEstLinearTilde'}
\begin{aligned}
\sup_I\int_\omega
\big(&\vert \partial_t^2\naby \eta\vert^2 
+
\vert \partial_t \naby^3 \eta\vert^2
\big)
\dy
+
\sup_I\int_\Omega\vert \partial_t\nabx \overline\bu\vert^2\dx
\\&+
\int_I\int_\omega
\big(\vert \partial_t^2\naby^2 \eta \vert^2 + \vert \partial_t^3 \eta\vert^2+ \vert \partial_t\naby^4  \eta\vert^2
 \big)\dy\dt
 \\&+
\int_I\int_\Omega\big( \vert \partial_t\nabx^2 \overline\bu\vert^2 +\vert \partial_t^2\overline\bu \vert^2  +\vert \partial_t\overline\pi\vert^2 + \vert \partial_t\nabx  \overline\pi\vert^2
 \big)\dx\dt\lesssim \mathcal D_*(g, \eta_0, \eta_\star, \overline\bu_0, h, \mathbf{h},\mathbf{H}),
\end{aligned}
\end{align}
where
\begin{align*}
\mathcal{D}_*&(g, \eta_0, \eta_\star, \overline\bu_0, h, \mathbf{h},\mathbf{H}):=
\Vert \eta_\star\Vert_{W^{3,2}(\omega)}^2
 +
 \Vert \eta_0\Vert_{W^{5,2}(\omega)}^2
  +
 \Vert
  \overline{\bu}_0\Vert_{W^{3,2}(\Omega)}^2
  +
  \Vert \mathbb{S}_\bq(\overline{\hbar}(0))
  \Vert_{W^{2,2}(\Omega)}^2
   \\&
    +
   \Vert \mathbf{h}(0)\Vert_{W^{1,2}(\Omega)}^2 
   +
  \Vert \mathbf{H}(0)\Vert_{W^{2,2}(\Omega)}^2
   +
   \int_I\Vert \partial_t^2 h \Vert_{W^{-1,2}(\Omega)}^2\dt 
   +
\Vert \partial_t h(0) \Vert_{L^2(\Omega )}^2
%    +
%\textcolor{red}{ \Vert  g(0)\Vert_{W^{1,2}(\omega)}^2}
   \\&+
   \int_I \big(\Vert \partial_t g \Vert_{W^{1,2}(\omega)}^2
 +
 \Vert \partial_t h \Vert_{W^{1,2}(\Omega)}^2  
 +
 \Vert \partial_t\mathbf{h}\Vert_{L^2(\Omega)}^2
 +
\Vert
\mathbb{S}_\bq(\partial_t\overline{\hbar} )
\Vert_{W^{1,2}(\Omega)}^2
 +
  \Vert \partial_t\mathbf{H}\Vert_{W^{1,2}(\Omega)}^2
  \big)
\dt.
\end{align*}
\end{corollary}
 
\subsection{Fixed-point argument}
\label{fixedPointFluidAlone} 
In this section we assume that the triplet $(\zeta, \overline{\mathbf{w}}, \overline{q})$ are given and we wish to solve
\begin{align}
\label{contEqAloneBarFixed}
\mathbf{B}_{\eta_0}:\nabx \overline{\bu}= h_\zeta(\overline{\mathbf{w}}),
\\
\partial_t^2\eta - \partial_t\Dely \eta +  \Dely^2\eta
=
g+\bn^\top \big[\mathbf{H}_\zeta(\overline{\mathbf{w}}, \overline{q})-\mathbf{A}_{\eta_0}  \nabx\overline{\bu} +\mathbf{B}_{\eta_0}(\overline{\pi}\,\mathbb I_{3\times 3}-
\mathbb{S}_\bq(\overline{\hbar}) )\big]\circ\bm{\varphi} \bn ,
\label{shellEqAloneBarFixed}
\\
J_{\eta_0}\partial_t \overline{\bu}  -\divx(\mathbf{A}_{\eta_0}  \nabx\overline{\bu}-\mathbf{B}_{\eta_0}\overline{\pi}) 
 = 
 \divx(\mathbf{B}_{\eta_0}\mathbb{S}_\bq(\overline{\hbar}) )
 +
\mathbf{h}_\zeta(\overline{\mathbf{w}})-
\divx  \mathbf{H}_\zeta(\overline{\mathbf{w}}, \overline{q})
\label{momEqAloneBarFixed}
\end{align}
with  $\overline{\bu}  \circ \bm{\varphi}  =(\partial_t\eta)\bn$ on $I_*\times \omega$. Here, $I_*:=(0,T_*)$ is to be determined later. Let's define  the space
\begin{align*}
X_{I_*}:=& W^{1,\infty}\big(I_*;W^{3,2}(\omega)  \big) \cap W^{2,2}\big(I_*;W^{1,2}(\omega)  \big)
\cap W^{3,2}\big(I_*;L^2(\omega)  \big)
\\
&\qquad\qquad\qquad \cap W^{1,2}\big(I_*;W^{3,2}(\omega)  \big)\cap L^{\infty}\big(I_*;W^{4,2}(\omega)  \big)\cap L^{2}\big(I_*;W^{5,2}(\omega)  \big)\\
&
\times 
W^{1,\infty} \big(I_*; W^{1,2}(\Omega ) \big)\cap  W^{2,2}\big(I_*;L^2(\Omega)  \big)\cap W^{1,2}\big(I_*;W^{2,2}(\Omega)  \big)
\cap L^2\big(I_*;W^{3,2}(\Omega)  \big)
\\&
\times
W^{1,2}\big(I_*;W^{1,2}(\Omega)  \big)
\cap
L^2\big(I_*;W^{2,2}(\Omega)  \big)
\end{align*}
equipped with the norm
\begin{align*}
\Vert (\zeta,\overline{\mathbf{w}}, \overline{q}) \Vert_{X_{I_*}}^2
&:=
\sup_{I_*}\int_\omega\big( \vert \partial_t\naby^3\zeta\vert^2 +|\naby^4\zeta|^2 \big)\dy
+
\int_{I_*}\int_\omega\big(   \vert\partial_t^3\zeta \vert^2 +  \vert\naby^5\zeta \vert^2 \big)\dy\dt
\\&
+\sup_{I_*}\int_\Omega\big( \vert\partial_t \overline{\mathbf{w}}\vert^2 +  \vert \partial_t\nabx \overline{\mathbf{w}} \vert^2  \big)\dx
\\&+
\int_{I_*}\int_\omega\big( \vert\partial_t^2  \overline{\mathbf{w}} \vert^2
+ \vert \partial_t\nabx \overline{\mathbf{w}}\vert^2 +  \vert \partial_t \nabx^2 \overline{\mathbf{w}} \vert^2 + \vert \nabx^3\overline{\mathbf{w}}\vert^2 \big)\dx\dt
\\&+
\int_{I_*}\int_\omega\big( 
\vert \partial_t\overline{q}\vert^2 +\vert \partial_t\nabx\overline{q}\vert^2 
+\vert \nabx^2\overline{q}\vert^2
\big)\dx\dt.
\end{align*}
Note that for $\zeta$, we only keep track of the highest-order terms. However, on account of the embeddings
\begin{align}\label{eq:emb}
\begin{aligned}
%W^{2,\infty}\big(I_*;W^{1,2}(\omega)  \big)  \cap L^{\infty}\big(I_*;W^{5,2}(\omega)  \big)  \hookrightarrow W^{1,\infty}\big(I_*;W^{3,2}(\omega)  \big) ,\\
L^{2}\big(I_*;W^{5,2}(\omega)  \big)  \cap W^{3,2}\big(I_*;L^{2}(\omega)  \big)  \hookrightarrow W^{2,2}\big(I_*;W^{1,2}(\omega)  \big), \\
L^{2}\big(I_*;W^{5,2}(\omega)  \big)  \cap W^{3,2}\big(I_*;L^{2}(\omega)  \big)  \hookrightarrow W^{1,2}\big(I_*;W^{3,2}(\omega)  \big) ,
\end{aligned}
\end{align}
the norm $\Vert (\zeta,\overline{\mathbf{w}}, \overline{q}) \Vert_{X_{I_*}}^2$ actually controls
\begin{align*}
&\sup_{I_*}\int_\omega\big( \vert \partial_t\naby^3\zeta\vert^2 +|\naby^4\zeta|^2 \big)\dy
\\&+
\int_{I_*}\int_\omega\big( \vert \partial_t\naby^3\zeta\vert^2 +  \vert \partial_t^2\naby\zeta\vert^2 +  \vert\partial_t^3\zeta \vert^2 +  \vert\naby^5\zeta \vert^2 \big)\dy\dt.
\end{align*}
Now let $B_R^{X_{I_*}}$ be defined as
%\todo{We must include $w_0$, otherwise problem with compability condition}
\begin{align*} 
B_R^{X_{I_*}}:= \big\{ (\zeta,\overline{\mathbf{w}}, \overline{q})\in X_{I_*}
\text{ with } \,\, \zeta(0)=\eta_0, \,\, \partial_t\zeta(0) = \eta_\star, \,\, \overline{\mathbf w}(0) = \overline{\bu}_0, \, \text{ such that }\, \Vert(\zeta,\overline{\mathbf{w}}, \overline{q})\Vert_{X_{I_*}}^2\leq R^2 \big\}
\end{align*}
for some $R>0$ large enough, where the data is chosen to satisfy \eqref{compatibilityCondition}. We want to show that the solution map $\mathcal{T}:X_{I_*}\rightarrow X_{I_*}$ defined by $\mathcal{T}(\zeta, \overline{\mathbf{w}}, \overline{q})=(\eta,\overline{\bu}, \overline{\pi})$ maps the ball  $B_R^{X_{I_*}} $ into itself and that it is a contraction. By so doing  we obtain the existence of a unique fixed point. See for example, \cite[Lemma 2.3]{kreml2010local}.
\\
We will show these two properties of $\mathcal{T}$ in two different spaces where one space is contained in the other. The fact that $\mathcal{T}$ maps the ball into itself will be shown in the space  $X_{I_*}$ defined above. For the contraction property, we consider the auxiliary space $\widehat{X}_{I_*}$ defined by
\begin{align*}
\widehat{X}_{I_*}:=&W^{1,\infty}\big(I_*;W^{1,2}(\omega)  \big)  \cap 
 L^{\infty}\big(I_*;W^{3,2}(\omega)  \big)\cap W^{1,2}\big(I_*;W^{2,2}(\omega)  \big)\cap W^{2,2}\big(I_*;L^{2}(\omega)  \big)\\
&
\times 
L^{\infty} \big(I_*; W^{1,2}(\Omega ) \big)\cap  W^{1,2}\big(I_*;L^2(\Omega)  \big)\cap L^{2}\big(I_*;W^{2,2}(\Omega)  \big)
\\&\times L^2\big(I_*;W^{1,2}(\Omega)  \big)
\end{align*}
and equipped with its corresponding canonical norm $\Vert\cdot\Vert_{\widehat{X}_{I_*}}$. By keeping \eqref{eq:emb} in mind, one observes that $ X_{I_*} \subset \widehat{X}_{I_*}$. Furthermore, with $ \widehat{X}_{I_*}$ in hand, we refer to \cite{BMSS} where we show that
for any $(\zeta_i, \overline{\mathbf{w}}_i, \overline{q}_i)\in B_R^{X_{I_*}}$, $i=1,2$, we can find $\rho<1$ such that
\begin{align*}
\Vert \mathcal{T}(\zeta_1, \overline{\mathbf{w}}_1, \overline{q}_1)
-
\mathcal{T}(\zeta_2, \overline{\mathbf{w}}_2, \overline{q}_2)\Vert_{\widehat{X}_{I_*}}\leq \rho \Vert (\zeta_1, \overline{\mathbf{w}}_1, \overline{q}_1)-(\zeta_2, \overline{\mathbf{w}}_2, \overline{q}_2)\Vert_{\widehat{X}_{I_*}}.
\end{align*}
Thus, $\mathcal{T}$ is a contraction.
\\
To show the mapping  $\mathcal{T}:B_R^{X_{I_*}} \rightarrow  B_R^{X_{I_*}}$, we need to show that for any $(\zeta, \overline{\mathbf{w}}, \overline{q}) \in B_R^{X_{I_*}}$, we have that 
\begin{align}
\label{ballToBall}
\Vert\mathcal{T}(\zeta, \overline{\mathbf{w}}, \overline{q})\Vert_{X_{I_*}}^2
=
\Vert(\eta, \overline{\bu}, \overline{\pi})\Vert_{X_{I_*}}^2\leq R^2.
\end{align}
Indeed, from \eqref{timeEstimate'} and \eqref{energyEstLinearTilde'}, we can deduce that the solution to \eqref{contEqAloneBarFixed}--\eqref{momEqAloneBarFixed} satisfies
\begin{equation}
\begin{aligned}
\label{ballToballEst}
\Vert(\eta, \overline{\bu}, \overline{\pi})\Vert_{X_{I_*}}^2
& \lesssim f_0+
   \int_I \big(\Vert \partial_t^2 h_\zeta(\overline{\mathbf{w}}) \Vert_{W^{-1,2}(\Omega)}^2
 +
 \Vert \partial_t h_\zeta(\overline{\mathbf{w}}) \Vert_{W^{1,2}(\Omega)}^2   
   \\&
  +
 \Vert \partial_t\mathbf{h}_\zeta(\overline{\mathbf{w}})\Vert_{L^2(\Omega)}^2 
+
   \Vert \mathbf{h}_\zeta(\overline{\mathbf{w}})\Vert_{W^{1,2}(\Omega)}^2 +
  \Vert \partial_t\mathbf{H}_\zeta(\overline{\mathbf{w}}, \overline{q})\Vert_{W^{1,2}(\Omega)}^2
\\&+
  \Vert \mathbf{H}_\zeta(\overline{\mathbf{w}}, \overline{q})\Vert_{W^{2,2}(\Omega)}^2
  +
   \Vert   h_\zeta(\overline{\mathbf{w}}) \Vert_{W^{2,2}(\Omega)}^2  \big)\dt 
\\&
=:
f_0+K_1+\ldots+K_7
\end{aligned}
\end{equation}
where for $R>0$ large enough, the dataset estimate
\begin{equation}
\begin{aligned}
f_0:= &
 \Vert \eta_\star\Vert_{W^{3,2}(\omega)}^2
 +
 \Vert \eta_0\Vert_{W^{5,2}(\omega)}^2
  +
 \Vert
  \overline{\bu}_0\Vert_{W^{3,2}(\Omega)}^2+
  \Vert
  \mathbb{S}_\bq(\overline{\hbar}(0))
  \Vert_{W^{2,2}(\Omega)}^2  
\\&+
\int_I
\big(
  \Vert
  \mathbb{S}_\bq(\partial_t\overline{\hbar} )
  \Vert_{W^{1,2}(\Omega)}^2
 +
  \Vert
  \mathbb{S}_\bq(\overline{\hbar})
  \Vert_{W^{2,2}(\Omega)}^2
  \big)\dt  
    +
 \Vert
  g(0)\Vert_{W^{1,2}(\omega)}^2
   \\&+
   \int_I \big(\Vert \partial_t g \Vert_{W^{1,2}(\omega)}^2
   + 
 \Vert g\Vert_{W^{1,2}(\omega)}^2
 %\todo{******we may not need this. Check******}
  \big)\dt
   +
   \Vert \mathbf{h}_{\zeta(0)}(\overline{\mathbf{w}}(0))\Vert_{W^{1,2}(\Omega)}^2 
    \\&+
  \Vert \mathbf{H}_{\zeta(0)}(\overline{\mathbf{w}}(0), \overline{q}(0))\Vert_{W^{2,2}(\Omega)}^2
\end{aligned}
\end{equation}
is such that 
\begin{align}
\label{rBoundsFnot}
\tilde{c}f_0<R^2/2.
\end{align}
Here, $\tilde{c}$ is the constant in the inequality \eqref{ballToballEst}.
%Because our dataset is regular \eqref{datasetAlone}, for any $(\zeta, \overline{\mathbf{w}}, \overline{q}) \in B_R^{X_{I_*}}$, we can choose $R>0$ large enough so that $c f_0<R$. Here, $c$ is the constant in \eqref{ballToballEst}. 
Let us now estimate $K_1$ in \eqref{ballToballEst}.
First of all,  we write
\begin{align*}
\partial_t^2 h_{\zeta}(\overline{\mathbf{w}})
&
=
 (\mathbf{B}_{\eta_0}-\mathbf{B}_{\zeta}):\partial_t^2\nabx \overline{\mathbf{w}}
+2\partial_t(\mathbf{B}_{\eta_0}-\mathbf{B}_{\zeta}):\partial_t\nabx \overline{\mathbf{w}}
+
\partial_t^2(\mathbf{B}_{\eta_0}-\mathbf{B}_{\zeta}):\nabx \overline{\mathbf{w}}.
\end{align*} 
Now note that it follows from \eqref{210and212}--\eqref{211and213} and the continuous embedding
\begin{align}\label{eq:emb1}
L^2(I_*;W^{3,2}(\omega))\cap W^{1,2}(I_*;W^{2,2}(\omega))\hookrightarrow C^{0,1/4}(\overline I_*;W^{9/4,2}(\omega))\hookrightarrow
L^\infty(I_*;W^{1,\infty}(\omega)),
\end{align}
that
\begin{equation}
\begin{aligned}
\label{k1a}
\int_{I_*}\Vert (\mathbf{B}_{\eta_0}-\mathbf{B}_{\zeta }):\partial_t^2\nabx \overline{\mathbf{w}} \Vert_{W^{-1,2}_{\bx}}
^2\dt
&\lesssim
\int_{I_*}\big(\Vert \nabx(\mathbf{B}_{\eta_0}-\mathbf{B}_{\zeta})\Vert_{L^3_\bx}^2
+
\Vert  \mathbf{B}_{\eta_0}-\mathbf{B}_{\zeta }\Vert_{L^\infty_\bx}^2
\big)
\Vert\partial_t^2  \overline{\mathbf{w}}  \Vert_{L^2_\bx}^2
\dt
%\\&+
%\int_{I_*}\Vert  \mathbf{B}_{\eta_0}-\mathbf{B}_{\zeta }\Vert_{L^\infty_\bx}^2\Vert\partial_t^2 \overline{\mathbf{w}} \Vert_{L^2_\bx}^2
%\dt
\\&\lesssim\sup_{I_*}\big(\Vert  \eta_0 - \zeta\Vert_{W^{2,3}_\by}^2
+
\Vert  \eta_0 - \zeta\Vert_{W^{1,\infty}_\by}^2
\big)\int_{I_*}\Vert  \partial_t^2  \overline{\mathbf{w}}  \Vert_{L^2_\bx}^2
\dt
%\\&+
%\sup_{I_*}\Vert  \eta_0 - \zeta\Vert_{W^{1,\infty}_\by}^2\int_{I_*}\Vert  \partial_t^2 \overline{\mathbf{w}}  \Vert_{L^2_\bx}^2
%\dt
\\&
\lesssim
T^{1/2}_*
\Vert  (\zeta , \overline{\mathbf{w}} , \overline{q}) \Vert_{X_{I_*}}^2.
\end{aligned}
\end{equation}
On the other hand, due to the continuous embedding
\begin{align*}
L^{2}(I_*;W^{4,2}(\Omega)) \cap
W^{2,2}(I_*;W^{1,2}(\Omega))
\hookrightarrow
W^{3/2,2}(I_*;W^{1,4}(\Omega))\hookrightarrow
W^{1,4}(I_*;W^{1,4}(\Omega))
\end{align*}
and \eqref{eq:emb}
it follows that
\begin{equation}
\begin{aligned}
\label{k1c}
\int_{I_*}\Vert2\partial_t (\mathbf{B}_{\eta_0}-\mathbf{B}_{\zeta}):\partial_t\nabx \overline{\mathbf{w}}_1 \Vert_{W^{-1,2}_\bx}^2
\dt
%=
%\int_{I_*}\Vert2\partial_t \mathbf{B}_{\zeta}:\partial_t\nabx \overline{\mathbf{w}}  \Vert_{W^{-1,2}_\bx}^2
%\dt
%\\&
&\lesssim
\int_{I_*}\Vert  \partial_t\mathbf{B}_{\zeta}\Vert^2_{L^4_\bx}\Vert\partial_t\nabx \overline{\mathbf{w}}  \Vert_{L^2_\bx}^2
\dt
\\&\lesssim T^{1/2}_*
\bigg(\int_{I_*}(1+\Vert  \partial_t\zeta \Vert_{W^{1,4}_\by})^4\dt\bigg)^\frac{1}{2}
\sup_{I_*}\Vert  \partial_t\nabx \overline{\mathbf{w}}  \Vert_{L^2_\bx}^2
\\&
\lesssim
T^{1/2}_*
\Vert  (\zeta , \overline{\mathbf{w}} , \overline{q}) \Vert_{X_{I_*}}^2.
\end{aligned}
\end{equation}
Similar to \eqref{k1c}, we can use the embeddings
\begin{align*}
W^{3,2}(I_*;L^2(\omega))\cap W^{2,2}(I_*;W^{2,2}(\omega))\hookrightarrow W^{5/2,2}(I_\ast;W^{1,2}(\omega)) \hookrightarrow
W^{2,2}(I_*;W^{1,2}(\omega)),\\
W^{1,2}(I_*;W^{2,2}(\Omega))\hookrightarrow L^\infty(I_*;W^{1,4}(\Omega)),
\end{align*}
to obtain
\begin{equation}
\begin{aligned}
\nonumber
\int_{I_*}\Vert\partial_t^2 (\mathbf{B}_{\eta_0}-\mathbf{B}_{\zeta}):\nabx \overline{\mathbf{w}} \Vert_{W^{-1,2}_\bx}^2
\dt
%&\lesssim \int_{I_*}\|\partial_t^2\zeta\|^2_{W^{1,2}_\by}\|\overline{\mathbf w} \|_{W_{\bx}^{1,4}}^2\dt
%\\ 
&\lesssim\sup_{I_*} \|\overline{\mathbf w} \|_{W_{\bx}^{1,4}}^2\int_{I_*}\|\partial_t^2\zeta \|^2_{W^{1,2}_\by}\dt 
\\
&
\lesssim
T_*
\Vert  (\zeta , \overline{\mathbf{w}} , \overline{q}) \Vert_{X_{I_*}}^2.\label{k1e}
\end{aligned}
\end{equation}
It follows from \eqref{k1a}--\eqref{k1e} that
\begin{equation}
\begin{aligned}
\label{k1final}
K_1
\lesssim T^{1/2}_*
\Vert  (\zeta , \overline{\mathbf{w}} , \overline{q}) \Vert_{X_{I_*}}^2.
\end{aligned}
\end{equation}
To estimate $K_2$, note that
\begin{align*}
\vert
\partial_t h_{\zeta}(\overline{\mathbf{w}}) 
\vert
\vert
&\lesssim
\vert (\mathbf{B}_{\eta_0}-\mathbf{B}_{\zeta}) \partial_t\nabx\overline{\mathbf{w}} 
\vert
%+
%\vert (\mathbf{B}_{\eta_0}-\mathbf{B}_{\zeta}) \partial_t\nabx^2 \overline{\mathbf{w}} 
%\vert
%\\&
+
\vert \partial_t(\mathbf{B}_{\eta_0}-\mathbf{B}_{\zeta}) \nabx \overline{\mathbf{w}}
\vert
%+
%\vert \partial_t (\mathbf{B}_{\eta_0}-\mathbf{B}_{\zeta }) \nabx^2  \overline{\mathbf{w}}
%\vert
%\\&
%+
%\vert \nabx(\mathbf{B}_{\eta_0}-\mathbf{B}_{\zeta})\partial_t\nabx\overline{\mathbf{w}}
%\vert
%+
%\vert \partial_t\nabx(\mathbf{B}_{\eta_0}-\mathbf{B}_{\zeta})\nabx \overline{\mathbf{w}}
%\vert
.
\end{align*}
Using again the embedding \eqref{eq:emb1}
it follows from \eqref{210and212}--\eqref{211and213} that
\begin{equation}
\begin{aligned}
\label{K2a}
\int_{I_*}\Vert   (\mathbf{B}_{\eta_0}-\mathbf{B}_{\zeta}) \partial_t\nabx \overline{\mathbf{w}} \Vert_{L^2_\bx}^2\dt
&+
\int_{I_*}\Vert   (\mathbf{B}_{\eta_0}-\mathbf{B}_{\zeta}) \partial_t\nabx^2 \overline{\mathbf{w}} \Vert_{L^2_\bx}^2\dt
\\&\lesssim
\sup_{I_*}\Vert    \eta_0 - \zeta \Vert_{W^{1,\infty}_\by}^2
\int_{I_*} \Vert\partial_t   \overline{\mathbf{w}} \Vert_{W^{2,2}_\bx}^2\dt
\\&\lesssim
T^{1/2}_*
\Vert (\zeta, \overline{\mathbf{w}}, \overline{q})\Vert_{X_{I_*}}^2
\end{aligned}
\end{equation}
and similarly,
\begin{equation}
\begin{aligned}
\label{K2a1}
\int_{I_*}\Vert   \partial_t(\mathbf{B}_{\eta_0}-\mathbf{B}_{\zeta}) \nabx \overline{\mathbf{w}} \Vert_{L^2_\bx}^2\dt
&+
\int_{I_*}\Vert   \partial_t(\mathbf{B}_{\eta_0}-\mathbf{B}_{\zeta}) \nabx^2 \overline{\mathbf{w}} \Vert_{L^2_\bx}^2\dt
\\&\lesssim\sup_{I_*} \Vert  \overline{\mathbf{w}} \Vert_{W^{2,2}_\bx}^2
\int_{I_*}\Vert    \partial_t \zeta \Vert_{W^{1,\infty}_\by}^2\dt
\\&\lesssim
T_*
\Vert (\zeta, \overline{\mathbf{w}}, \overline{q})\Vert_{X_{I_*}}^2
\end{aligned}
\end{equation}
holds due to the embedding
\begin{align*}
W^{1,\infty}(I_\ast;W^{3,2}(\omega))\hookrightarrow  W^{1,2}(I_\ast, W^{1,\infty}(\omega)).
\end{align*}
%Now by  \eqref{210and212}--\eqref{211and213} and the continuous embeddings
%\begin{align*}
%&W^{1,\infty}(I_*;W^{3,2}(\omega))\cap \hookrightarrow
%W^{1,\infty}(I_*;W^{2,4}(\omega)), \quad
%\\
%&W^{2,2}(I_*;L^2(\Omega))\cap W^{1,2}(I_*;W^{2,2}(\Omega))\hookrightarrow
%W^{1,4}(I_*;W^{1,4}(\Omega)),
%%\\&
%%W^{1,2}(I_*;W^{4,2}(\omega))\cap W^{2,2}(I_*;W^{2,2}(\omega))\hookrightarrow
%%W^{1,\infty}(I_*;W^{2,4}(\omega)),
%\end{align*}
%and \eqref{eq:emb},
%it follows that
%\begin{equation}
%\begin{aligned}
%\label{K2b}
%\int_{I_*}\Big(\Vert  \nabx (\mathbf{B}_{\eta_0}-\mathbf{B}_{\zeta}) \partial_t\nabx \overline{\mathbf{w}} \Vert_{L^2_\bx}^2 
%&+\Vert \partial_t \nabx (\mathbf{B}_{\eta_0}-\mathbf{B}_{\zeta}) \nabx  \overline{\mathbf{w}}\Vert_{L^2_\bx}^2\Big)
%\dt
%\\&\lesssim
%\sup_{I_*}\Vert  \partial_t( \eta_0 - \zeta)\Vert_{W^{2,4}_\by}^2
%\int_{I_*} \Vert \partial_t\nabx \overline{\mathbf{w}} \Vert_{L^4_\bx}^2\dt
%\\&\lesssim T^{1/2}_*
%\sup_{I_*}\Vert    \partial_t \zeta \Vert_{W^{3,2}_\by}^2
%\bigg(\int_{I_*} \Vert  \partial_t\nabx \overline{\mathbf{w}} \Vert_{L^4_\bx}^4\dt\bigg)^\frac{1}{2}
%\\
%%&+
%%T^{1/2}_*
%%\Vert (\zeta_1, \overline{\mathbf{w}}_1, \overline{q}_1)-(\zeta_2, \overline{\mathbf{w}}_2, \overline{q}_2)\Vert_{X_{I_*}}^2
%%\\
%&\lesssim
%T^{1/2}_*
%\Vert (\zeta , \overline{\mathbf{w}} , \overline{q})\Vert_{X_{I_*}}^2.
%\end{aligned}
%\end{equation}
It follows %from \eqref{K2a}--\eqref{K2b}
 that
\begin{align}
\label{k2final}
K_2
\lesssim
T^{1/2}_*
\Vert (\zeta , \overline{\mathbf{w}} , \overline{q})\Vert_{X_{I_*}}^2.
\end{align}
Next, note that
\begin{align*}
\vert
\partial_t[\mathbf{h}_{\zeta}(\overline{\mathbf{w}})
&]
\vert
\lesssim
\vert (J_{\eta_0}-J_{\zeta})\partial_t^2 \overline{\mathbf{w}} 
\vert 
+
\big\vert
J_{\zeta} \partial_t\nabx  \overline{\mathbf{w}} \partial_t \bm{\Psi}_{\zeta}^{-1}\circ \bm{\Psi}_{\zeta} 
\big\vert
+
\big\vert
J_{\zeta}\big( \nabx \bm{\Psi}_{\zeta}^{-1}\circ \bm{\Psi}_{\zeta}\big) \nabx \overline{\mathbf{w}} \partial_t\overline{\mathbf{w}}
\big\vert
\\&
+
\vert
J_{\zeta} \partial_t \bff\circ \bm{\Psi}_{\zeta}  
\vert
+
\vert
\partial_t(J_{\zeta } ) 
\bff\circ \bm{\Psi}_{\zeta }  
\vert
%\\&
%+
%\vert
%J_{\zeta_1} \partial_t(\divx\overline{\mathbb{S}}_\bq\circ \bm{\Psi}_{\zeta_1}^{-1}-
%\divx\overline{\mathbb{S}}_\bq\circ \bm{\Psi}_{\zeta_2}^{-1} )
%\vert
%+
%\vert
%\partial_t(J_{\zeta_1}-J_{\zeta_2}) 
%\divx\overline{\mathbb{S}}_\bq\circ \bm{\Psi}_{\zeta_2}^{-1} 
%\vert
+\mathrm{L.O.T}.
\end{align*}
where $\mathrm{L.O.T}$
are lower-order terms satisfying
\begin{align}
\label{lot}
\int_{I_*}\Vert \mathrm{L.O.T}\Vert_{L^2(\Omega)}^2\dt
%:=
%\int_I\Vert\nabx[ h_{\zeta_1}(\overline{\mathbf{w}}_1)- h_{\zeta_2}(\overline{\mathbf{w}}_2)]\Vert_{L^2(\Omega)}^2\dt
\lesssim
T^{1/2}_*
\Vert (\zeta,\overline{\mathbf{w}}, \overline{q}) \Vert_{X_{I_*}}^2.
\end{align}
Due to the continuous embeddings \eqref{eq:emb1}, 
it follows from the definition $J_\eta=\det(\nabx \bm{\Psi}_\eta)$ and \eqref{210and212}--\eqref{211and213}  that
\begin{equation}
\begin{aligned}
\label{K3a}
\int_{I_*}\Vert   (J_{\eta_0}-J_{\zeta})\partial_t^2  \overline{\mathbf{w}} \Vert_{L^2_\bx}^2\dt
&\lesssim
\sup_{I_*}\Vert  \eta_0 - \zeta \Vert_{W^{1,\infty}_\by}^2
\int_{I_*}\Vert \partial_t^2 \overline{\mathbf{w}}   \Vert_{L^2_\bx}^2\dt
 \lesssim
T^{1/2}_*
\Vert (\zeta ,\overline{\mathbf{w}} , \overline{q} ) \Vert_{X_{I_*}}^2.
\end{aligned}
\end{equation}
By using the embeddings
\begin{align*}
&L^\infty(I_*;W^{3,2}(\omega))\hookrightarrow L^\infty(I_*;W^{1,\infty}(\omega)),
\\&
W^{2,2}(I_*;L^2(\omega))
\cap
L^\infty(I_*;W^{3,2}(\omega))\hookrightarrow 
W^{1,4}(I_*;L^{\infty}(\omega)),
\end{align*}
we obtain
\begin{equation}
\begin{aligned}
\label{K3b}
\int_{I_*} \Vert   J_{\zeta }&
\partial_t \nabx  \overline{\mathbf{w}}  \partial_t \bm{\Psi}_{\zeta }^{-1}\circ \bm{\Psi}_{\zeta } \Vert_{L^2_\bx}^2\dt
\lesssim
\int_{I_*}\big(1+\Vert   \zeta 
\Vert_{W^{1,\infty}_\by}^2
\big)
\Vert \partial_t\nabx  \overline{\mathbf{w}} \Vert_{L^2_\bx}^2
\big(1+
\Vert \partial_t\zeta  \Vert_{L^\infty_\by}^2\big) \dt
\\&
\lesssim
T^{1/2}_*
\sup_{I_*}\big(1+\Vert   \zeta 
\Vert_{W^{1,\infty}_\by}^2
\big)
\sup_{I_*}
\Vert \partial_t\nabx  \overline{\mathbf{w}} \Vert_{L^2_\bx}^2
\bigg(\int_{I_*}
\big(1+
\Vert \partial_t\zeta  \Vert_{L^\infty_\by}^4\big)
\dt\bigg)^\frac{1}{2}
\\&
\lesssim
T^{1/2}_*
\Vert (\zeta ,\overline{\mathbf{w}} , \overline{q} ) \Vert_{X_{I_*}}^2.
\end{aligned}
\end{equation}
Also, by using the embeddings
\begin{align*}
&L^\infty(I_*;W^{3,2}(\omega))\hookrightarrow L^\infty(I_*;W^{1,\infty}(\omega)),
\\&
W^{1,2}(I_*;W^{2,2}(\Omega))\hookrightarrow L^\infty(I_*;W^{1,4}(\Omega)),
\\&
W^{2,2}(I_*;L^2(\Omega))
\cap
W^{1,2}(I_*;W^{2,2}(\Omega))\hookrightarrow 
W^{1,4}(I_*;L^{ 4}(\Omega)),
\end{align*}
we obtain
\begin{equation}
\begin{aligned}
\label{K3d}
\int_{I_*}&\Vert   J_{\zeta }\big( \nabx \bm{\Psi}_{\zeta }^{-1}\circ \bm{\Psi}_{\zeta }\big)\nabx  \overline{\mathbf{w}} \partial_t \overline{\mathbf{w}}  \Vert_{L^2_\bx}^2\dt
 \lesssim
\int_{I_*}\big(1+\Vert   \zeta   \Vert_{W^{1,\infty}_\by}^4\big)
\Vert
\nabx \overline{\mathbf{w}} \Vert_{L^4_\bx}^2
\Vert \partial_t\overline{\mathbf{w}} \Vert_{L^4_\bx}^2\dt
\\&\lesssim
T^{1/2}_*
\Big(\sup_{I_*}\Vert   \zeta   \Vert_{W^{1,\infty}_\by}^4
+1\Big)
\sup_{I_*}
\Vert \nabx \overline{\mathbf{w}}   \Vert_{L^4_\bx}^2
\bigg(\int_{I_*}\Vert
\partial_t \overline{\mathbf{w}} \Vert_{L^4_\bx}^4\dt
\bigg)^\frac{1}{2} 
\\&
\lesssim
T^{1/2}_*
\Vert (\zeta ,\overline{\mathbf{w}} , \overline{q} ) \Vert_{X_{I_*}}^2.
\end{aligned}
\end{equation}
Next, we have that
\begin{equation}
\begin{aligned}
\label{K3f}
\int_{I_*}\Vert 
J_{\zeta } \partial_t(\bff\circ \bm{\Psi}_{\zeta }  )
 \Vert_{L^2_\bx}^2\dt
&
\lesssim
T_*\sup_{I_*}
\Vert
J_{\zeta }
\Vert_{L^\infty_\bx}^2\bigg(\sup_{I_*} \Vert \partial_t\bff \Vert_{L^2_\bx}^2+\sup_{I_*} \Vert \nabla\bff \Vert_{L^2_\bx}^2
\sup_{I_*}
\Vert
\partial_t \zeta 
\Vert_{L^\infty_\by}^2\bigg)\\&
\lesssim
T_*
\Vert (\zeta ,\overline{\mathbf{w}} , \overline{q} ) \Vert_{X_{I_*}}^2.
\end{aligned}
\end{equation}
Similarly, we have
\begin{equation}
\begin{aligned}
\label{K3g}
\int_{I_*}\Vert 
\partial_t(J_{\zeta }) 
\bff\circ \bm{\Psi}_{\zeta } 
 \Vert_{L^2_\bx}^2\dt
&
\lesssim
T_*
\sup_{I_*} \Vert \bff \Vert_{L^2_\bx}^2
\sup_{I_*}
\Vert
\bm{\Psi}_{\zeta }^{-1} 
\Vert_{L^\infty_\bx}^2
\sup_{I_*}
\Vert
\partial_t \zeta 
\Vert_{W^{1,\infty}_\by}^2
\lesssim
T_*
\Vert (\zeta ,\overline{\mathbf{w}} , \overline{q} ) \Vert_{X_{I_*}}^2.
\end{aligned}
\end{equation}
It follows from \eqref{K3a}--\eqref{K3g} that
\begin{equation}
\begin{aligned}
\label{k3final}
K_3
\lesssim T^{1/2}_*
\Vert (\zeta ,\overline{\mathbf{w}} , \overline{q} ) \Vert_{X_{I_*}}^2.
\end{aligned}
\end{equation}
Our next goal is to estimate $K_4$. First of all, note that
\begin{equation}
\begin{aligned}
\nonumber
\vert \partial_t \mathbf{H}_{\zeta }(\overline{\mathbf{w}} , \overline{q} )]\vert
&\lesssim
\vert (\mathbf{A}_{\eta_0} -\mathbf{A}_{\zeta })\partial_t\nabx \overline{\mathbf{w}} \vert 
 +
\vert \partial_t(\mathbf{A}_{\eta_0} -\mathbf{A}_{\zeta })\nabx \overline{\mathbf{w}} \vert 
+
\vert (\mathbf{B}_{\eta_0}-\mathbf{B}_{\zeta }) \partial_t (\overline{q}_1-\mathbb S_{\mathbf q}(\overline{\hbar})) \vert 
%+
%\vert \partial_t(\mathbf{B}_{\eta_0}-\mathbf{B}_{\zeta }) \mathbb S_{\mathbf q}(\overline{\hbar})\vert
%\\&
\\&+
\vert \partial_t(\mathbf{B}_{\eta_0}-\mathbf{B}_{\zeta })  (\overline{q}_1\,\mathbb I_{3\times 3}-\mathbb S_{\mathbf q}(\overline{\hbar})) \vert 
%+
%\vert (\mathbf{B}_{\eta_0}-\mathbf{B}_{\zeta }) \partial_t\mathbb S_{\mathbf q}(\overline{\hbar})\vert
\end{aligned}
\end{equation}
holds uniformly. Due to the continuous embeddings
\begin{align*}
&\qquad\qquad\qquad\qquad W^{1,\infty}(I_*;W^{3,2}(\omega))\hookrightarrow W^{1,\infty}(I_*;W^{1,\infty}(\omega)),
\\&L^{2}(I_*;W^{4,2}(\Omega)) \cap
W^{2,2}(I_*;W^{1,2}(\Omega))
\hookrightarrow
W^{3/2,2}(I_*;W^{1,4}(\Omega))\hookrightarrow
W^{1,4}(I_*;W^{1,4}(\Omega)),\\
&W^{1,2}(I_\ast;W^{2,2}(\Omega))\hookrightarrow C^{0,1/2}(\overline I_\ast;W^{2,2}(\Omega))\hookrightarrow L^{\infty}(I_\ast;W^{2,2}(\Omega)),
%\\&
%W^{1,2}(I_*;W^{3,2}(\omega))\cap W^{2,2}(I_*;W^{2,2}(\omega))\hookrightarrow
%W^{1,\infty}(I_*;W^{1,\infty}(\omega)),
\end{align*}
it follows from \eqref{210and212}--\eqref{211and213} that
\begin{equation}
\begin{aligned}
\label{K4a}
\int_{I_*}&\Vert (\mathbf{A}_{\eta_0} -\mathbf{A}_{\zeta }) \partial_t\nabx \overline{\mathbf{w}} \Vert_{W^{1,2}_\bx}^2\dt
+
\int_{I_*}\Vert \partial_t(\mathbf{A}_{\eta_0} -\mathbf{A}_{\zeta })\nabx \overline{\mathbf{w}} \Vert_{W^{1,2}_\bx}^2\dt
\\&\lesssim
\int_{I_*}\Vert \nabx(\mathbf{A}_{\eta_0} -\mathbf{A}_{\zeta })  \Vert_{L^4_\bx}^2\Vert\partial_t\nabx \overline{\mathbf{w}} \Vert_{L^4_\bx}^2\dt
 +
\int_{I_*}\Vert \partial_t\nabx(\mathbf{A}_{\eta_0} -\mathbf{A}_{\zeta })  \Vert_{L^4_\bx}^2\Vert\nabx \overline{\mathbf{w}} \Vert_{L^4_\bx}^2\dt
\\&+
\int_{I_*}\Vert  \mathbf{A}_{\eta_0} -\mathbf{A}_{\zeta }\Vert_{L^\infty_\bx}^2
\Vert\partial_t\nabx^2 \overline{\mathbf{w}} \Vert_{L^2_\bx}^2\dt 
+
\int_{I_*}\Vert  \partial_t(\mathbf{A}_{\eta_0} -\mathbf{A}_{\zeta })\Vert_{L^\infty_\bx}^2
\Vert\nabx^2 \overline{\mathbf{w}} \Vert_{L^2_\bx}^2\dt
\\&\lesssim
T^{1/2}_*
\sup_{I_*}\Vert \eta_0 - \zeta \Vert_{W^{2,4}_\by}^2\bigg(\int_{I_*}\Vert\partial_t\nabx \overline{\mathbf{w}} \Vert_{L^4_\bx}^4\dt\bigg)^\frac{1}{2}
\\
&+T^{1/2}_*
\sup_{I_*}\Vert \partial_t( \eta_0 - \zeta ) \Vert_{W^{2,4}_\by}^2\bigg(\int_{I_*}\Vert\nabx \overline{\mathbf{w}} \Vert_{L^4_\bx}^4\dt\bigg)^\frac{1}{2}
\\
&+
\sup_{I_*}\Vert  \eta_0 - \zeta  \Vert_{W^{1,\infty}_\by}^2
\int_{I_*}\Vert \partial_t\nabx^2 \overline{\mathbf{w}} \Vert_{L^2_\bx}^2\dt
\\
&+
\sup_{I_*}\Vert  \partial_t(\eta_0 - \zeta ) \Vert_{W^{1,\infty}_\by}^2
\int_{I_*}\Vert \nabx^2 \overline{\mathbf{w}} \Vert_{L^2_\bx}^2\dt
\\
&\lesssim
T^{1/2}_*
\Vert (\zeta , \overline{\mathbf{w}} , \overline{q} )\Vert_{X_{I_*}}^2.
\end{aligned}
\end{equation}
%Next, we use again the embedding
%\begin{align*}
%&W^{1,\infty}(I_*;W^{3,2}(\omega))\hookrightarrow W^{1,\infty}(I_*;W^{1,\infty}(\omega)),
%\end{align*}
Similarly, we obtain
\begin{equation}
\begin{aligned}
\label{K4b}
\int_{I_*}\Vert \partial_t(\mathbf{B}_{\eta_0}-&\mathbf{B}_{\zeta }) (\overline{q}\,\mathbb I_{3\times 3}-\mathbb S_{\mathbf q}(\overline{\hbar}))\Vert_{W^{1,2}_\bx}^2\dt
+
\int_{I_*}\Vert (\mathbf{B}_{\eta_0}-\mathbf{B}_{\zeta }) \partial_t (\overline{q}\,\mathbb I_{3\times 3}-\mathbb S_{\mathbf q}(\overline{\hbar})) \Vert_{W^{1,2}_\bx}^2\dt
\\
&\lesssim
\int_{I_*}\Vert \partial_t\nabx(\mathbf{B}_{\eta_0}-\mathbf{B}_{\zeta }) \Vert_{L^4_\bx}^2
\Vert
\overline{q}\,\mathbb I_{3\times 3}-\mathbb S_{\mathbf q}(\overline{\hbar}) \Vert_{L^4_\bx}^2\dt
\\& +\int_{I_*}\Vert \nabx(\mathbf{B}_{\eta_0}-\mathbf{B}_{\zeta }) \Vert_{L^4_\bx}^2
\Vert
\partial_t (\overline{q}\,\mathbb I_{3\times 3}-\mathbb S_{\mathbf q}(\overline{\hbar})) \Vert_{L^4_\bx}^2\dt
\\
&+
\int_{I_*}\Vert \mathbf{B}_{\eta_0}-\mathbf{B}_{\zeta } \Vert_{L^\infty_\bx}^2
\Vert
\partial_t\nabx (\overline{q}\,\mathbb I_{3\times 3}-\mathbb S_{\mathbf q}(\overline{\hbar})) \Vert_{L^2_\bx}^2\dt\\
& +
\int_{I_*}\Vert\partial_t( \mathbf{B}_{\eta_0}-\mathbf{B}_{\zeta }) \Vert_{L^\infty_\bx}^2
\Vert
\nabx (\overline{q}\,\mathbb I_{3\times 3}-\mathbb S_{\mathbf q}(\overline{\hbar})) \Vert_{L^2_\bx}^2\dt
\\
&\lesssim
\sup_{I_*}\Vert  \partial_t(\eta_0 -  \zeta  ) \Vert_{W^{2,4}_\by}^2
\int_{I_*}
\Vert
\overline{q}\,\mathbb I_{3\times 3}-\mathbb S_{\mathbf q}(\overline{\hbar}) \Vert_{L^4_\bx}^2\dt
 \\
&+\sup_{I_*}\Vert  \eta_0 -  \zeta   \Vert_{W^{2,4}_\by}^2
\int_{I_*}
\Vert
\partial_t (\overline{q}\,\mathbb I_{3\times 3}-\mathbb S_{\mathbf q}(\overline{\hbar})) \Vert_{L^4_\bx}^2\dt
 \\
&+
\sup_{I_*}\Vert  \eta_0 - \zeta  \Vert_{W^{1,\infty}_\by}^2
\int_{I_*}
\Vert
\partial_t (\overline{q}\,\mathbb I_{3\times 3}-\mathbb S_{\mathbf q}(\overline{\hbar}))\Vert_{W^{1,2}_\bx}^2\dt
 \\
&+
\sup_{I_*}\Vert  \partial_t(\eta_0 - \zeta ) \Vert_{W^{1,\infty}_\by}^2
\int_{I_*}
\Vert
\overline{q}\,\mathbb I_{3\times 3}-\mathbb S_{\mathbf q}(\overline{\hbar})\Vert_{W^{1,2}_\bx}^2\dt
 \\
&\lesssim
T^{1/2}_*\big(
\Vert (\zeta , \overline{\mathbf{w}} , \overline{q} )\Vert_{X_{I_*}}^2+1\big)
\end{aligned}
\end{equation}
using \eqref{rBoundsFnot} to controll $\mathbb S_{\mathbf q}(\overline{\hbar})$ in the last step.
%Similarly,
%\begin{equation}
%\begin{aligned}
%\label{K4c}
%\int_{I_*}\Vert \partial_t(\mathbf{B}_{\eta_0}-\mathbf{B}_{\zeta })  \mathbb{S}_\bq(\overline{\hbar})\Vert_{W^{1,2}_\bx}^2\dt
%&+\int_{I_*}\Vert (\mathbf{B}_{\eta_0}-\mathbf{B}_{\zeta })  \partial_t\mathbb{S}_\bq(\overline{\hbar})\Vert_{W^{1,2}_\bx}^2\dt
% \lesssim
%T^{1/2}_*
%\Vert (\zeta , \overline{\mathbf{w}} , \overline{q} )\Vert_{X_{I_*}}^2.
%\end{aligned}
%\end{equation}
By using \eqref{K4a}--\eqref{K4b}, it follows that
\begin{align}
\label{k4final}
K_4
\lesssim T^{1/2}_*
\big(\Vert (\zeta ,\overline{\mathbf{w}} , \overline{q} ) \Vert_{X_{I_*}}^2+1\big).
\end{align}
Let us now obtain an estimate for $K_5$.
Firstly, note that
\begin{align*}
\vert
\mathbf{h}_{\zeta}(\overline{\mathbf{w}})
\vert+\vert
\nabx\mathbf{h}_{\zeta}(\overline{\mathbf{w}})
\vert
&\lesssim
\vert (J_{\eta_0}-J_{\zeta})\partial_t\nabx \overline{\mathbf{w}} 
\vert
+
\vert
\nabx(J_{\eta_0}-J_{\zeta_1})\partial_t \overline{\mathbf{w}}
\vert
\\&
+
\big\vert
J_{\zeta} \nabx \overline{\mathbf{w}} \big(\partial_t \bm{\Psi}_{\zeta}^{-1}\circ \bm{\Psi}_{\zeta} 
+\overline{\mathbf{w}}\big( \nabx \bm{\Psi}_{\zeta}^{-1}\circ \bm{\Psi}_{\zeta}\big)\big)
\big\vert
\\&
+
\big\vert
\nabx[J_{\zeta}\big( \nabx \bm{\Psi}_{\zeta}^{-1}\circ \bm{\Psi}_{\zeta}\big)]  \overline{\mathbf{w}} \nabx \overline{\mathbf{w}} 
\big\vert 
\\&
+
\big\vert
J_{\zeta}  \nabx \overline{\mathbf{w}}  \nabx^2\big(\partial_t\bm{\Psi}_{\zeta}^{-1}\circ \bm{\Psi}_{\zeta} 
\big)
\big\vert   
\\&
+
\vert
J_{\zeta} \nabx(\bff\circ \bm{\Psi}_{\zeta} )
\vert
+
\vert
(\nabx J_{\zeta}) \bff\circ \bm{\Psi}_{\zeta}
\vert 
+\mathrm{L.O.T}.
\end{align*}
where $\mathrm{L.O.T}$
are lower-order terms satisfying
\begin{align}
\label{lot}
\int_{I_*}\Vert \mathrm{L.O.T}\Vert_{L^2(\Omega)}^2\dt
\lesssim T^{1/2}_*
\big(\Vert (\zeta ,\overline{\mathbf{w}} , \overline{q} ) \Vert_{X_{I_*}}^2+1\big).
\end{align}
Furthermore,   the other terms can be treated as was done for $K_3$ leading to
\begin{equation}
\begin{aligned}
\label{k5final}
K_5
\lesssim T^{1/2}_*
\big(\Vert (\zeta ,\overline{\mathbf{w}} , \overline{q} ) \Vert_{X_{I_*}}^2+1\big).
\end{aligned}
\end{equation}
The estimate for $K_6$ is similar to that of $K_4$ by noticing that
\begin{align*}
\sum_{k=0}^2
\vert
\nabx^k \mathbf{H}_{\zeta}(\overline{\mathbf{w}},\overline{q})
\vert
&\lesssim
\vert (\mathbf{B}_{\eta_0}-\mathbf{B}_{\zeta}) \nabx^2(\overline{q}  - S_{\mathbf q}(\overline{\hbar}))\vert
+
\vert \nabx^2(\mathbf{B}_{\eta_0}-\mathbf{B}_{\zeta_1}) (\overline{q}  - S_{\mathbf q}(\overline{\hbar}))\vert 
\\&
+
\vert (\mathbf{A}_{\eta_0} -\mathbf{A}_{\zeta})\nabx^3\overline{\mathbf{w}} \vert
+
\vert \nabx^2(\mathbf{A}_{\eta_0} -\mathbf{A}_{\zeta})\nabx \overline{\mathbf{w}} \vert
+\mathrm{L.O.T}
\end{align*}
where $\mathrm{L.O.T}$
are lower order terms also satisfying \eqref{lot}.
%\begin{align}
%\label{lot}
%\int_{I_*}\Vert \mathrm{L.O.T}\Vert_{L^2(\Omega)}^2\dt
%%:=
%%\int_I\Vert\nabx[ h_{\zeta_1}(\overline{\mathbf{w}}_1)- h_{\zeta_2}(\overline{\mathbf{w}}_2)]\Vert_{L^2(\Omega)}^2\dt
%\lesssim T^{1/2}_*
%\Vert (\zeta ,\overline{\mathbf{w}} , \overline{q} ) \Vert_{X_{I_*}}^2.
%\end{align}
Consequently, we obtain
\begin{equation}
\begin{aligned}
\label{k6final}
K_6
\lesssim T^{1/2}_*\big(
\Vert (\zeta ,\overline{\mathbf{w}} , \overline{q} ) \Vert_{X_{I_*}}^2+1\big).
\end{aligned}
\end{equation}
Similarly, we also obtain
\begin{equation}
\begin{aligned}
\label{k7final}
K_7
\lesssim T^{1/2}_*
\Vert (\zeta ,\overline{\mathbf{w}} , \overline{q} ) \Vert_{X_{I_*}}^2.
\end{aligned}
\end{equation}
By collecting the estimates \eqref{k1final}, \eqref{k2final}, \eqref{k3final}, \eqref{k4final}, \eqref{k5final}, \eqref{k6final} \eqref{k7final} and combining it with \eqref{ballToballEst}--\eqref{rBoundsFnot}, we have shown that for $R$-dependent constants $c(R),C(R)$
\begin{align*}
\Vert \mathcal{T}(\zeta, \overline{\mathbf{w}}, \overline{q} )\Vert_{X_{I_*}}^2
&\leq 
R^2/2
+
c(R) T^{1/2}_*
\big(\Vert (\zeta ,\overline{\mathbf{w}} , \overline{q} ) \Vert_{X_{I_*}}^2+1\big)
 \leq 
R^2/2
+
C(R) T^{1/2}_*.
\end{align*}
Choosing $T_*$ in $I_*=(0,T_*)$ so that $T^{1/2}_*\leq C(R)^{-1}R^2/2$ yields our desired result \eqref{ballToBall}.

\section{Solving the equation for the solute}
\label{sec:FP}
\noindent
In this section, for a known moving domain $\Ozeta$ and a known solenoidal velocity field $\mathbf{w}$, we aim to construct a strong solution of the Fokker--Planck equation
\begin{align}
\label{eq:FP}
M\big(\partial_t \widehat{f} + (\mathbf{w}\cdot \nabx) \widehat{f})
+
 \divq  \big(  (\nabx\mathbf{w}) \bq M\widehat{f} \big) 
=
\Delx(M\widehat{f})
+
 \divq  \big( M \nabq  \widehat{f}
\big)
\end{align} 
in $I\times\Omega_\zeta\times B$,
where the Maxwellian $M$ is given by
\begin{align*}
M(\mathbf{q}) = \frac{e^{-U \left(\frac{1}{2}\vert \mathbf{q} \vert^2 \right) }}{\int_Be^{-U \left(\frac{1}{2}\vert \mathbf{q} \vert^2 \right) }\,\mathrm{d}\mathbf{q}},\quad  U (s) = -\frac{b}{2} \log \bigg(1-  \frac{2s}{b} \bigg), \quad s\in [0,b/2)
\end{align*}
with $b>2$. Equation \eqref{eq:FP} is complemented
with the conditions
\begin{align}
&\widehat{f}(0, \cdot, \cdot) =\widehat{f}_0 \geq 0
& \quad \text{in }\Omega_{\zeta_0} \times B,
\label{fokkerPlankIintialxsec4}
\\
&
\nabx\widehat{f}\cdot \bn_\zeta =0
&\quad \text{on }I \times \partial\Omega_\zeta \times B,
\label{fokkerPlankBoundarySpacexsc4}
\\
&M\big(\nabq\widehat{f}   -  (\nabx \mathbf w) \bq \widehat{f}
 \big) \cdot \frac{\bq}{\vert\bq\vert} =0
&\quad \text{on }I \times \Omega_\zeta \times \partial \overline{B}.
\label{fokkerPlankBoundaryxsc4}
\end{align}
Let us start with a precise definition of what we mean by a strong solution.  
\begin{definition}
%[Strong solution] 
\label{def:strsolmartFP}
Assume that the triplet $(\widehat{f}_0,\zeta, \mathbf{w})$ satisfies
\begin{align}
\label{fokkerPlanckDataAlone}
\widehat{f}_0\in  W^{1,2}\big( \Omega_{\zeta(0)}; L^2_M(B) &\big)   ,
\qquad
\mathbf{w}\in L^2(I; W^{3,2}_{\divx}(\Ozeta)) \cap  W^{1,\infty}(I;W^{1,2} ( \Ozeta))\cap  W^{1,2}(I;W^{2,2} ( \Ozeta)),
\\
&\qquad\qquad\zeta\in W^{1,\infty}(I;W^{2,2}(\omega))
,
\label{fokkerPlanckDataAloneA}
\\& \mathbf w  \circ \bm{\varphi}_{\zeta} =(\partial_t\zeta)\bn
\quad \text{on }I \times \omega,  \quad\|\zeta\|_{L^\infty(I\times\omega)}<L.
\label{fokkerPlanckDataAloneB}
\end{align}
We call
$\widehat{f}$
a \textit{strong solution} of   \eqref{eq:FP} with data $(\widehat{f}_0,\zeta, \mathbf{w})$ if 
\begin{itemize}
\item[(a)] $\widehat{f}$ satisfies
\begin{align*}
\widehat{f}&\in   W^{1,\infty}\big(I;W^{1,2}(\Oeta;L^2_M(B))  \big)
\cap 
W^{1,2}\big(I;W^{2,2}(\Oeta;L^2_M(B))  \big)
\\&
\qquad\qquad \qquad\cap W^{1,2}\big(I;W^{1,2}(\Oeta;H^1_M(B))  \big);
\end{align*}
\item[(b)] for all  $  \varphi  \in C^\infty (\overline{I}\times \R^3 \times \overline{B} )$, we have
\begin{equation}
\begin{aligned}
\label{weakFokkerPlanckEq}
\int_I  \frac{\mathrm{d}}{\dt}
\int_{\Ozeta \times B}M \widehat{f} \, \varphi \dq \dx \dt 
&=\int_I\int_{\Ozeta \times B}\big(M \widehat{f} \,\partial_t \varphi 
+
M\mathbf{w} \widehat{f} \cdot \nabx \varphi
-
\nabx \widehat{f} \cdot \nabx \varphi
\big) \dq \dx \dt
\\&
+ \int_I\int_{ \Ozeta \times B}
 \big( M (\nabx\mathbf{w})  \bq\widehat{f}-
 M \nabq  \widehat{f} \big) \cdot \nabq\varphi \dq \dx \dt.
\end{aligned}
\end{equation}
\end{itemize}
\end{definition}
\noindent We now formulate our result on the existence of a  unique strong solution of \eqref{eq:FP}.
\begin{theorem}\label{thm:mainFP}
Let $(\widehat{f}_0,\zeta, \mathbf{w})$ satisfy  \eqref{fokkerPlanckDataAlone}--\eqref{fokkerPlanckDataAloneB} and  
suppose further that $\widetilde f_0\in L^{2}(\Omega_{\zeta(0)};L^2_M(B))$, where $\widetilde f_0$ is given by 
%\footnote{please change this back, otherwise we must explain how it depends on the data}
\begin{equation}
\begin{aligned}
\label{eq:FP1Approx1Initial}
M   \widetilde{f}_0 
&=
\Delx(M  \widehat{f}_0)
+
 \divq  \big( M \nabq    \widehat{f}_0
\big)
-
M  (\bu_0\cdot \nabx)   \widehat{f}_0
-
 \divq  \big(   (  \nabx\bu_0) \bq M\widehat{f}_0 \big) .
\end{aligned}
\end{equation}
Then there is a unique strong solution $\widehat{f}$ of \eqref{eq:FP}--\eqref{fokkerPlankBoundaryxsc4}, in the sense of Definition \ref{def:strsolmartFP},
 such that
\begin{equation}
\begin{aligned} 
&
 \sup_I \Vert \partial_t\widehat{f}(t)
\Vert_{L^{2}(\Ozeta;L^2_M(B))}^2+
 \int_I 
 \Vert \partial_t\widehat{f}
\Vert_{W^{1,2}(\Ozeta;L^2_M(B))}^2
 \dt
 +
 \int_I 
 \Vert \partial_t\widehat{f}
\Vert_{L^{2}(\Ozeta;H^1_M(B))}^2
 \dt
 \\
 &+
 \sup_I \Vert \widehat{f}(t)
\Vert_{W^{1,2}(\Ozeta;L^2_M(B))}^2+
 \int_I 
 \Vert \widehat{f}
\Vert_{W^{2,2}(\Ozeta;L^2_M(B))}^2
 \dt
 +
 \int_I 
 \Vert \widehat{f}
\Vert_{W^{1,2}(\Ozeta;H^1_M(B))}^2
 \dt
 \\ 
&\quad\lesssim  \exp\bigg(c\int_I
\Vert   \mathbf{w}\Vert_{W^{5/2+\kappa,2}(\Ozeta)}^2\dt
+
c\int_I \Vert\partial_t\mathbf{w}\Vert_{W^{3/2+\kappa,2 }(\Ozeta )}^2\dt
 \bigg)
\\ 
&\qquad\qquad
\times \exp\bigg( c
\sup_{I}
\Vert  \mathbf{w}\Vert_{ W^{7/4,2}(\Ozeta) }^2
+c
\sup_I\|\partial_t\zeta\|_{W^{1,2}(\omega)}^{3} \bigg)
\\
&\qquad\qquad
\times \Big( \|\widehat f_0\|^2_{W^{1,2}(\Omega_{\zeta(0)};L^2_M(B))}+
\Vert \widetilde{f}_0
\Vert_{L^{2}(\Omega_{\zeta(0)};L^2_M(B))}^2\Big)\label{eq:thm:mainFP}
\end{aligned}
\end{equation} 
holds for any $\kappa\in(0,1/2)$ with a constant depending on the $L^\infty(I;W^{1,\infty}(\omega))$-norm of $\zeta$ but otherwise being independent of the data.
\end{theorem}
\noindent We will obtain a solution of \eqref{eq:FP} by way of a limit to the following approximation
\begin{equation}
\begin{aligned}
\label{eq:FP1Approx}
M\big(\partial_t   \widehat{f}^n + (\mathbf{w}\cdot \nabx)   \widehat{f}^n\big)
+
 \divq  \big( \chi^n (  \nabx\mathbf{w}) \bq M\widehat{f}^n \big) 
&=
\Delx(M  \widehat{f}^n)
+
 \divq  \big( M \nabq    \widehat{f}^n
\big).
\end{aligned}
\end{equation}
Here, we solve the equation under the boundary conditions $\nabx \widehat f^n\cdot\bn_\zeta=0$ and $M\nabq \widehat{f} \vert_{\partial B}=0$  and consider the same initial condition $\widehat{f}^n(0)=\widehat{f}_0$. Also, $\chi^n=\chi^n(\bq)\in C^1_c(B)$ is a cut-off function that is identically equal to $1$ on a large part of the ball $B$ and converges as $n\rightarrow\infty$ to 1.
In the following lemmas, we will derive several estimates for
\eqref{eq:FP1Approx} which are uniform with respect to $n$. They transfer directly to \eqref{eq:FP} as the latter is linear. As far as \eqref{eq:FP1Approx} is concerned, we proceed formerly. A rigorous proof can be achieved by working with a Galerkin approximation as was done in \cite[Section 4]{breit2021incompressible}.
%We will combine the proof of Theorem into several lemmas. Let begin with uniqueness.
%\begin{lemma}\label{lem:mainFPunique}
%Consider two  strong solutions $\widehat{f}^i$, $i=1,2$ of   \eqref{eq:FP} with the same dataset $(\mathbf{w} ,\zeta,\widehat{f}_0)$ satisfying \eqref{fokkerPlanckDataAlone}. For $\widehat{f}^{12}=\widehat{f}^1-\widehat{f}^2$,
%\begin{equation}
%\begin{aligned}
%\sup_{t\in I} \Vert \widehat{f}^{12}
%\Vert_{L^2(\Ozeta;L^2_M(B))}^2
%&+
% \int_I 
% \Vert \widehat{f}^{12}
%\Vert_{W^{1,2}(\Ozeta;L^2_M(B))}^2
% \dt
% +
% \int_I 
% \Vert \widehat{f}^{12}
%\Vert_{L^2(\Ozeta;H^1_M(B))}^2
% \dt
%=0.
%\end{aligned}
%\end{equation}
%\begin{proof}
%For $\widehat{f}^{12}=\widehat{f}^1-\widehat{f}^2$, the difference $\widehat{f}^{12}$ solves
%\begin{equation}
%\begin{aligned}
%\label{eq:FP2}
%M\big(\partial_t  \widehat{f}^{12} &+ (\mathbf{w} \cdot \nabx)  \widehat{f}^{12} )
%+
% \divq  \big(  ( \nabx\mathbf{w} ) \bq M\widehat{f}^{12} \big) 
% =
% \Delx\big( M  \widehat{f}^{12}
%\big)
%+
% \divq  \big( M \nabq  \widehat{f}^{12}
%\big).
%\end{aligned}
%\end{equation} 
%Since \eqref{eq:FP2} retains the same structure as \eqref{eq:FP} but with zero initial condition (since both solution share the same data), we directly obtain \eqref{eq:FP2}  from \eqref{fokkerEnergyEst5}.
%\end{proof}
%\end{lemma}
The next result is the following.
\begin{lemma}\label{lem:mainFP1}
Let $(\widehat{f}_0,\zeta, \mathbf{w})$ satisfy  \eqref{fokkerPlanckDataAlone}--\eqref{fokkerPlanckDataAloneB}  and let $\widehat{f}^n$ be the corresponding solution to \eqref{eq:FP1Approx}. Then we have
\begin{equation}
\begin{aligned}
\label{fokkerEnergyEst5}
\sup_I \Vert \widehat{f}^n(t)
\Vert_{L^2(\Ozeta;L^2_M(B))}^2
&+
 \int_I 
 \Vert \widehat{f}^n
\Vert_{W^{1,2}(\Ozeta;L^2_M(B))}^2
 \dt
 +
 \int_I 
 \Vert \widehat{f}^n
\Vert_{L^2(\Ozeta;H^1_M(B))}^2
 \dt
\\&\leq  c\exp\bigg(c\int_I \Vert  \mathbf{w} \Vert_{W^{1,\infty}(\Ozeta)}^2 \dt \bigg) 
\Vert \widehat{f}_0
\Vert_{L^2(\Omega_{\zeta(0)};L^2_M(B))}^2
%\\&\lesssim 1
\end{aligned}
\end{equation} 
uniformly in $n\in\mathbb{N}$. 
\end{lemma}
\begin{proof}
If we   test \eqref{eq:FP1Approx} with $  \widehat{f}^n $ and integrate the resulting equation over the ball $B$, we obtain by using the boundary condition $M\nabq \widehat{f}\vert_{\partial B}=0$ and the property of the cut-off function $\chi^n$ that
\begin{equation}
\begin{aligned}
\label{fokkerEnergy}
\frac{1}{2}\partial_t\Vert    \widehat{f}^n \Vert_{L^2_M(B)}^2
+
\frac{1}{2}(\mathbf{w}\cdot \nabx)\Vert    \widehat{f}^n \Vert_{L^2_M(B)}^2
&+
\Vert   \nabx \widehat{f}^n \Vert_{L^2_M(B)}^2
+ 
\Vert   \widehat{f}^n\Vert_{H^1_M(B)}^2
\\&=
\int_B
 \chi^n(  \nabx\mathbf{w}) \bq M\widehat{f}^n
\nabq    \widehat{f}^n
\dq
\\&
\leq
\frac{1}{2}
\Vert
\widehat{f}^n
\Vert_{H^1_M(B)}^2
+
c
\vert  \nabx\mathbf{w} \vert^2\Vert\widehat{f}^n
\Vert_{L^2_M(B)}^2.
\end{aligned}
\end{equation}
If we now integrate \eqref{fokkerEnergy} over space-time, apply Reynolds transport theorem and Gr\"onwall's Lemma (keeping \eqref{fokkerPlanckDataAloneB} in mind), we obtain  
\eqref{fokkerEnergyEst5}.
\end{proof}
\begin{remark}\label{rem:max}
As it is common for parabolic equations, the proof of Lemma \ref{lem:mainFP1} can be repeated for powers $q\geq 2$ of $\widehat f^n$ obtaining (ignoring the dissipative terms) 
\begin{equation}
\begin{aligned}
\label{fokkerEnergyEst5'}
\sup_I \Vert \widehat{f}^n(t)
\Vert_{L^q(\Ozeta;L^2_M(B))}^q\leq  c\exp\bigg(c\int_I \Vert  \mathbf{w} \Vert_{W^{1,\infty}(\Ozeta)}^2 \dt \bigg) 
\Vert \widehat{f}_0
\Vert_{L^q(\Omega_{\zeta(0)};L^2_M(B))}^q
%\\&\lesssim 1
\end{aligned}
\end{equation} 
uniformly in $n\in\mathbb{N}$. Checking that the $q$-dependent constant does not explode, we obtain the maximum principle\footnote{Maximum principles for parabolic equations in moving domains were also proved in \cite{breit2021incompressible,BS1} and \cite{BS2}.}
\begin{equation}
\begin{aligned}
\label{fokkerEnergyEst5''}
\sup_I \Vert \widehat{f}^n(t)
\Vert_{L^\infty(\Omega_\eta;L^2_M(B))}\leq  c\exp\bigg(c\int_I \Vert  \mathbf{w} \Vert_{W^{1,\infty}(\Ozeta)}^2 \dt \bigg) 
\Vert \widehat{f}_0
\Vert_{L^\infty(\Omega_{\zeta(0)};L^2_M(B))};
%\\&\lesssim 1
\end{aligned}
\end{equation}
a minimum principle can be proved similarly, but it is not needed for our purposes. 
\end{remark}
Next, we show the following lemma.
\begin{lemma}\label{lem:mainFP2}
Let $(\widehat{f}_0,\zeta, \mathbf{w})$ satisfy  \eqref{fokkerPlanckDataAlone}--\eqref{fokkerPlanckDataAloneB}  and let $\widehat{f}^n$ be the corresponding solution to \eqref{eq:FP1Approx}.
Then we have 
\begin{equation}
\begin{aligned}
\label{fokkerEnergyNeg1}
&\sup_I \Vert \nabx\widehat{f}^n(t)
\Vert_{L^2(\Ozeta;L^2_M(B))}^2+
 \int_I 
 \Vert \Delx\widehat{f}^n
\Vert_{L^2(\Ozeta;L^2_M(B))}^2
 \dt
 +
 \int_I 
 \Vert \nabx\widehat{f}^n
\Vert_{L^2(\Ozeta;H^1_M(B))}^2
 \dt
 \\&
 \qquad\quad\leq  c\exp\bigg(c\int_I \Vert  \mathbf{w} \Vert_{W^{5/2+\kappa,2}(\Ozeta)}^2 \dt
 +
c \int_I\Vert\partial_t\zeta\Vert_{W^{ 2/3,2}(\omega)}^3\dt
\bigg) 
\Vert \widehat{f}_0
\Vert_{W^{1,2}(\Omega_{\zeta(0)};L^2_M(B))}^2
%\\&\leq  c\exp\bigg(c\int_I \big(\Vert  \mathbf{w} \Vert_{W^{4,2}(\Ozeta)}^2
%+
%\Vert  \partial_t\zeta \Vert_{W^{1+\epsilon/2,2}(\omega)}^{2/(1-2\epsilon)}
%\big) \dt \bigg)
%\\&
%\times
%\bigg( 
%\Vert \widehat{f}_0
%\Vert_{W^{1,2}(\Omega_{\zeta(0)};L^2_M(B))}^2
%+
%\int_I\Vert \mathbf{w}\Vert_{W^{4,2}(\Ozeta)}^2\Vert  \widehat{f}^n\Vert_{L^2(\Ozeta;L^2_M(B))}^2\dt
%\bigg)
%\\&
\end{aligned}
\end{equation}
for any $\kappa\in(0,1/2)$ uniformly in $n\in\mathbb{N}$. 
\end{lemma}
\begin{proof}
Now, we test \eqref{eq:FP1Approx} with $\Delx  \widehat{f}^n$. First of all, note that by \eqref{fokkerPlankBoundarySpacexsc4}, the Reynolds transport theorem and \eqref{fokkerPlanckDataAloneB}, 
\begin{equation}
\begin{aligned}
\label{fokkerEnergyEst6}
\int_I\int_{\Ozeta\times B}M \partial_t   \widehat{f}^n  \Delx  \widehat{f}^n\dq\dx\dt
&=
\frac{1}{2}
\int_I\int_{\partial\Ozeta }
\bn_\zeta\cdot((\partial_t\zeta)\bn)\circ\bm{\varphi}_\zeta^{-1}
\Vert \nabx  \widehat{f}^n\Vert_{L^2_M(B)}^2\dd\mathcal{H}^2\dt
\\&-
\frac{1}{2}\int_I\frac{\dd}{\dt}
\Vert \nabx  \widehat{f}^n\Vert_{L^2(\Ozeta;L^2_M(B))}^2\dt,
\end{aligned}
\end{equation}
where, by interpolation, the trace theorem and Young's inequality,
\begin{equation}
\begin{aligned}
\label{fokkerEnergyEst7}
\bigg\vert\int_I&\int_{\partial\Ozeta }\bn_\zeta\cdot((\partial_t\zeta)\bn)\circ\bm{\varphi}_\zeta^{-1}\Vert \nabx  \widehat{f}^n\Vert_{L^2_M(B)}^2\dd\mathcal{H}^2\dt\bigg\vert
\\&\lesssim
\int_I\Vert \bn_\zeta\cdot((\partial_t\zeta)\bn)\circ\bm{\varphi}_\zeta^{-1}\Vert_{L^6(\partial\Ozeta)}\Vert \nabx  \widehat{f}^n\Vert_{L^{12/5}(\partial\Ozeta;L^2_M(B))}^2 \dt
\\&
\lesssim
\int_I\Vert\partial_t\zeta\Vert_{W^{2/3,2}(\omega)}\Vert \nabx  \widehat{f}^n\Vert_{W^{2/3,2}(\Ozeta;L^2_M( B))}^2 \dt
\\&
\lesssim
\int_I\Vert\partial_t\zeta\Vert_{W^{ 2/3,2}(\omega)}\Vert \nabx  \widehat{f}^n\Vert_{L^2(\Ozeta;L^2_M(B))}^{2/3}
\Vert \nabx  \widehat{f}^n\Vert_{W^{1,2}(\Ozeta;L^2_M( B))}^{4/3} \dt
\\&
\leq
\delta
\int_I 
\Vert \nabx  \widehat{f}^n\Vert_{W^{1,2}(\Ozeta;L^2_M(B))}^2 \dt
+
c(\delta)
\int_I\Vert\partial_t\zeta\Vert_{W^{2/3,2}(\omega)}^{3}\Vert \nabx  \widehat{f}^n\Vert_{L^2(\Ozeta;L^2_M(B))}^2 \dt.
\end{aligned}
\end{equation}
Next,
\begin{equation}
\begin{aligned}
\label{fokkerEnergyEst8}
\bigg\vert\int_I\int_{\Ozeta\times B}M  (\mathbf{w}\cdot \nabx)   \widehat{f}^n &\Delx  \widehat{f}^n\dq\dx\dt
\bigg\vert
\leq
\delta
\int_I 
\Vert \Delx  \widehat{f}^n\Vert_{L^2(\Ozeta;L^2_M(B))}^2 \dt
\\&+
c(\delta)
\int_I\Vert\mathbf{w}\Vert_{L^\infty(\Ozeta)}^2\Vert \nabx  \widehat{f}^n\Vert_{L^2(\Ozeta;L^2_M(B))}^2 \dt.
\end{aligned}
\end{equation}
For the dissipative term, we obtain
\begin{align}
\label{fokkerEnergyEst9}
\int_I\int_{\Ozeta\times B}  \Delx(M   \widehat{f}^n) \Delx  \widehat{f}^n\dq\dx\dt
=
\int_I\Vert\Delx  \widehat{f}^n\Vert_{L^2(\Ozeta;L^2_M(B))}\dt.
\end{align}
Next we use \eqref{fokkerPlankBoundarySpacexsc4}--\eqref{fokkerPlankBoundaryxsc4} and Sobolev embeddings and we obtain
\begin{align}
\int_I&\int_{\Ozeta\times B}\divq  \big( \chi^n (  \nabx\mathbf{w}) \bq M\widehat{f}^n - M \nabq    \widehat{f}^n
\big)\Delx  \widehat{f}^n\dq\dx\dt\nonumber
\\&
=
\int_I\int_{\Ozeta\times B}\nabx  \big( \chi^n (  \nabx\mathbf{w}) \bq M\widehat{f}^n - M \nabq    \widehat{f}^n
\big):\nabx\nabq  \widehat{f}^n\dq\dx\dt\nonumber
\\&
\leq-\int_I\Vert\nabx \widehat{f}^n\Vert_{L^2(\Ozeta;H^1_M(B))}^2\dt
+c
\int_I\Vert\nabx^2\mathbf{w}\Vert_{L^3(\Ozeta)}\Vert  \widehat{f}^n\Vert_{L^6(\Ozeta;L^2_M(B))}
\Vert \nabx \widehat{f}^n\Vert_{L^2(\Ozeta;H^1_M(B))}\dt\nonumber
\\&
+c
\int_I\Vert\nabx \mathbf{w}\Vert_{L^\infty(\Ozeta)}\Vert \nabx \widehat{f}^n\Vert_{L^2(\Ozeta;L^2_M(B))}
\Vert \nabx \widehat{f}^n\Vert_{L^2(\Ozeta;H^1_M(B))}\dt\nonumber
\\&
\leq
-\frac{1}{2}\int_I\Vert\nabx \widehat{f}^n\Vert_{L^2(\Ozeta;H^1_M(B))}^2\dt
+c
\int_I\Vert\mathbf{w}\Vert_{W^{5/2+\kappa,2}(\Ozeta)}^2\Vert\widehat{f}^n\Vert_{W^{1,2}(\Ozeta;L^2_M(B))}^2\dt,\label{fokkerEnergyEst10}
\end{align}
where $\kappa\in(0,1/2)$.
By combining \eqref{fokkerEnergyEst7}--\eqref{fokkerEnergyEst10} and applying Sobolev embeddings to the $\mathbf{w}$-terms, we obtain
\eqref{fokkerEnergyNeg1}
uniformly in $n\in \mathbb{N}$.
\end{proof}
Our next lemma is the following.
\begin{lemma}\label{lem:mainFP3}
Let $(\widehat{f}_0,\zeta, \mathbf{w})$ satisfy  \eqref{fokkerPlanckDataAlone}--\eqref{fokkerPlanckDataAloneB}  and let $\widehat{f}^n$ be the corresponding solution to \eqref{eq:FP1Approx}.
Then we have
\begin{equation}
\begin{aligned}
\label{fokkerEnergyNeg2}
&\int_I\Vert \partial_t\widehat{f}^n
\Vert_{L^2(\Ozeta;L^2_M(B))}^2\dt+
\sup_I \Vert \nabx\widehat{f}^n(t)
\Vert_{L^2(\Ozeta;L^2_M(B))}^2
+
\sup_I \Vert  \widehat{f}^n(t)
\Vert_{L^2(\Ozeta;H^1_M(B))}^2
\\
&\qquad\quad\leq\,c 
\exp\bigg(c\int_I
\Vert   \mathbf{w}\Vert_{W^{5/2+\kappa,2}(\Ozeta)}^2\dt
+c
\int_I
\Vert \partial_t \nabx\mathbf{w}\Vert_{L^{2}(\Ozeta)}^2\dt
\bigg)
\\
&\qquad\qquad\,\times
\exp\bigg(c
\sup_{I}
\Vert  \mathbf{w}\Vert_{ W^{7/4,2}(\Ozeta)  }^2
 +c
\sup_I\|\partial_t\zeta\|_{W^{2/3,2}(\omega)}^{3}\bigg)
\|\widehat f_0\|^2_{W^{1,2}(\Omega_{\zeta(0)};L^2_M(B))}
\end{aligned}
\end{equation} 
for any $\kappa\in(0,1/2)$ uniformly in $n\in\mathbb{N}$. 
\end{lemma}
\begin{proof}
Test \eqref{eq:FP1Approx} with $\partial_t \widehat{f}^n$. This yields
\begin{equation}
\begin{aligned}
\Vert   \partial_t \widehat{f}^n \Vert_{L^2_M(B)}^2
&=
-
\frac{1}{2}
\partial_t
\Vert   \nabx \widehat{f}^n \Vert_{L^2_M(B)}^2
-
\frac{1}{2}
\partial_t
\Vert   \widehat{f}^n\Vert_{H^1_M(B)}^2
+
\int_B\divx(M\nabx \widehat{f}^n\partial_t\widehat{f}^n)\dq
\\&
-
\int_B
M\mathbf{w}\cdot \nabx   \widehat{f}^n\partial_t\widehat{f}^n\dq
-
\int_B
\divq( \chi^n(  \nabx\mathbf{w}) \bq M\widehat{f}^n)
\partial_t \widehat{f}^n
\dq
\\&
=:I_1+\ldots+I_5.
\end{aligned}
\end{equation}
By  Reynold's transport theorem,
\begin{equation}
\begin{aligned}
\int_I\int_{\Ozeta}(I_1&+I_2)\dx\dt
=
-
\frac{1}{2}
\int_I
\frac{\dd}{\dt}
\Big(
\Vert   \nabx \widehat{f}^n \Vert_{L^2(\Ozeta;L^2_M(B))}^2
+
\Vert   \widehat{f}^n\Vert_{L^2(\Ozeta;H^1_M(B))}^2
\Big)\dt
\\&+
\frac{1}{2}
\int_I
\int_{\partial \Ozeta}
\bn_\zeta\cdot((\partial_t\zeta)\bn)\circ\bm{\varphi}_\zeta^{-1}
\Big(
\Vert   \nabx \widehat{f}^n \Vert_{L^2_M(B)}^2
+
\Vert   \widehat{f}^n\Vert_{H^1_M(B)}^2
\Big)\dd\mathcal{H}^2\dt,
\end{aligned}
\end{equation}
where, by the trace theorem, \eqref{fokkerPlanckDataAloneA} and Lemma \ref{lem:mainFP2} for $\kappa\in(0,1/2)$ 
\begin{align*}
\frac{1}{2}&
\int_I
\int_{\partial \Ozeta}
\bn_\zeta\cdot((\partial_t\zeta)\bn)\circ\bm{\varphi}_\zeta^{-1}\Big(
\Vert   \nabx \widehat{f}^n \Vert_{L^2_M(B)}^2
+
\Vert   \widehat{f}^n\Vert_{H^1_M(B)}^2
\Big)\dd\mathcal{H}^2\dt
\\&\lesssim
\int_I
\Vert\bn_\zeta\cdot((\partial_t\zeta)\bn)\circ\bm{\varphi}_\zeta^{-1}\Vert_{L^4(\partial \Ozeta)} 
\Big(
\Vert   \nabx \widehat{f}^n \Vert_{L^{8/3}(\partial \Ozeta;L^2_M(B))}^2
+
\Vert   \widehat{f}^n\Vert_{L^{8/3}(\partial \Ozeta;H^1_M(B))}^2
\Big)\dt 
\\&\lesssim
\sup_I\Vert\partial_t\zeta\Vert_{W^{1/2,2}(\omega)}
\int_I\Big(
\Vert   \nabx \widehat{f}^n \Vert_{W^{1,2}(  \Ozeta;L^2_M(B))}^2
+
\Vert   \widehat{f}^n\Vert_{W^{1,2}(  \Ozeta;H^1_M(B))}^2
\Big)\dt
\\
&\lesssim
\sup_I\Vert\partial_t\zeta\Vert_{W^{1/2,2}(\omega)}
\exp\bigg(c\int_I \Vert  \mathbf{w} \Vert_{W^{5/2+\kappa,2}(\Ozeta)}^2 \dt+c\int_I\Vert\partial_t\zeta\Vert_{W^{2/3,2}(\omega)}^{3}\dt \bigg) 
\Vert \widehat{f}_0
\Vert_{W^{1,2}(\Omega_{\zeta(0)};L^2_M(B))}^2
\\
&\lesssim
\exp\bigg(c\int_I \Vert  \mathbf{w} \Vert_{W^{5/2+\kappa,2}(\Ozeta)}^2 \dt+c\sup_I\Vert\partial_t\zeta\Vert_{W^{2/3,2}(\omega)}^{3} \bigg) 
\Vert \widehat{f}_0
\Vert_{W^{1,2}(\Omega_{\zeta(0)};L^2_M(B))}^2.
\end{align*}
By Gauss theorem and \eqref{fokkerPlankBoundarySpacexsc4}, we obtain $
\int_{\Omega_\zeta}I_3\dx=0$.
Also, by Lemma \ref{lem:mainFP2}, for any $\delta>0$,
\begin{equation*}
\begin{aligned}
\int_I\int_{\Ozeta}I_4\dx\dt
&\leq 
c(\delta)\int_I\|\mathbf w\|^2_{L^\infty(\Omega_\zeta)}\dt\sup_I\Vert   \nabx \widehat{f}^n \Vert_{L^2(\Ozeta;L^2_M(B))}^2+\delta\int_I\Vert   \partial_t \widehat{f}^n \Vert_{L^2(\Ozeta;L^2_M(B))}^2
\\&
\leq 
c(\delta)\exp\bigg(c\int_I \Vert  \mathbf{w} \Vert_{W^{5/2+\kappa,2}(\Ozeta)}^2 \dt 
+
c\int_I\Vert\partial_t\zeta\Vert_{W^{2/3,2}(\omega)}^{3}\dt\bigg) 
\Vert \widehat{f}_0
\Vert_{W^{1,2}(\Omega_{\zeta(0)};L^2_M(B))}^2\\
&+\delta\int_I\Vert   \partial_t \widehat{f}^n \Vert_{L^2(\Ozeta;L^2_M(B))}^2.
\end{aligned}
\end{equation*}
Using integration by parts and applying Reynolds transport theorem,
\begin{equation}
\begin{aligned}
\int_I\int_{\Ozeta}I_5\dx\dt
&=
\int_I\int_{\Ozeta\times B}
 \chi^n(  \nabx\mathbf{w}) \bq M \widehat{f}^n
\partial_t \nabq\widehat{f}^n
\dq\dx\dt
\\
&=-
\int_I\int_{\Ozeta\times B}
 \chi^n(  \partial_t\nabx\mathbf{w}) \bq M \widehat{f}^n
 \nabq\widehat{f}^n
\dq\dx\dt
\\&-
\int_I\int_{\Ozeta\times B}
 \chi^n(  \nabx\mathbf{w}) \bq M \partial_t\widehat{f}^n
 \nabq\widehat{f}^n
\dq\dx\dt
\\&+
\int_I\frac{\dd}{\dt}
\int_{\Ozeta\times B}
 \chi^n(  \nabx\mathbf{w}) \bq M \widehat{f}^n
 \nabq\widehat{f}^n
\dq\dx\dt
\\&-
\int_I\int_{\partial\Ozeta \times B}\bn_\zeta\cdot((\partial_t\zeta)\bn)\circ\bm{\varphi}_\zeta^{-1}
 \chi^n(   \nabx\mathbf{w}) \bq M \widehat{f}^n
 \nabq\widehat{f}^n
\dq\dd\mathcal{H}^2\dt 
\\&=:
I_5^1+\ldots+I_5^4,
\end{aligned}
\end{equation}
where, by \eqref{fokkerPlanckDataAlone}--\eqref{fokkerPlanckDataAloneB} as well as Lemmas \ref{lem:mainFP1} and \ref{lem:mainFP2}, 
\begin{equation}
\begin{aligned}
I_5^1
&\lesssim
\int_I
\Vert \widehat{f}^n\Vert_{L^6(\Ozeta;H^1_M(B))}^2\dt
+
\int_I
\Vert \partial_t\nabx \mathbf{w}\Vert_{L^2(\Ozeta)}^2\Vert \widehat{f}^n\Vert_{L^3(\Ozeta;L^2_M(B))}^2\dt
\\
&\lesssim
\int_I
\Vert \widehat{f}^n\Vert_{W^{1,2}(\Ozeta;H^1_M(B))}^2\dt
+
\sup_I\Vert \widehat{f}^n(t)\Vert_{W^{1,2}(\Ozeta;L^2_M(B))}^2\int_I
\Vert \partial_t \nabx\mathbf{w}\Vert_{L^{2}(\Ozeta)}^2\dt
\\&\lesssim\exp\bigg(c\int_I \Vert  \mathbf{w} \Vert_{W^{5/2+\kappa,2}(\Ozeta)}^2 \dt
+
c\int_I
\Vert \partial_t \nabx\mathbf{w}\Vert_{L^{2}(\Ozeta)}^2\dt
 +c\int_I\Vert\partial_t\zeta\Vert_{W^{2/3,2}(\omega)}^{3}\dt \bigg) 
\\&
\qquad\times\Vert \widehat{f}_0
\Vert_{W^{1,2}(\Omega_{\zeta(0)};L^2_M(B))}^2
%\\&
%\times \bigg(1+\int_I
%\Vert \partial_t \mathbf{w}\Vert_{W^{2,2}(\Ozeta)}^2\dt
%\bigg)
%\\
%&\lesssim  
%\exp\bigg(c\int_I \Vert  \mathbf{w} \Vert_{W^{3,2}(\Ozeta)}^2 \dt +c\int_I\Vert\partial_t\zeta\Vert_{W^{2,2}(\omega)}^{3}\dt \bigg) 
%\Vert \widehat{f}_0
%\Vert_{W^{1,2}(\Omega_{\zeta(0)};L^2_M(B))}^2
.
\end{aligned}
\end{equation}
Also, 
\begin{align*}
I_5^2
&\leq
\delta
\int_I
\Vert 
\partial_t\widehat{f}^n\Vert_{L^2(\Ozeta;L^2_M(B))}^2\dt
+
c(\delta)
\int_I
\Vert   \mathbf{w}\Vert_{W^{1,\infty}(\Ozeta)}^2\Vert \widehat{f}^n\Vert_{L^2(\Ozeta;H^1_M(B))}^2\dt
%\\
%&\leq \delta
%\int_I
%\Vert 
%\partial_t\widehat{f}^n\Vert_{L^2(\Ozeta;L^2_M(B))}^2\dt
%+
%c(\delta)
%\int_I
%\Vert   \mathbf{w}\Vert_{W^{1,\infty}(\Ozeta)}^2\dt\exp\bigg(c\int_I
%\Vert   \mathbf{w}\Vert_{W^{1,\infty}(\Ozeta)}^2\dt\bigg)\|\widehat f_0\|^2_{L^2(\Omega_{\zeta(0)};H^1_M(B))}\\
%&\leq \delta
%\int_I
%\Vert 
%\partial_t\widehat{f}^n\Vert_{L^2(\Ozeta;L^2_M(B))}^2\dt
%+
%c(\delta)
%\exp\bigg(c\int_I
%\Vert   \mathbf{w}\Vert_{W^{1,\infty}(\Ozeta)}^2\dt\bigg)\|\widehat f_0\|^2_{L^2(\Omega_{\zeta(0)};H^1_M(B))}
\end{align*}
holds for any $\delta>0$, where the second term will be handled using Gr\"onwall's lemma. Also, by  Lemma \ref{lem:mainFP2},
\begin{align*}
I_5^3
&\leq c(\delta)
\sup_{I}\Big(
\Vert \nabx \mathbf{w}(t)\Vert_{L^4(\Ozeta)}^2\Vert \widehat{f}^n(t)\Vert_{L^4(\Ozeta;L^2_M(B))}^2\Big)+\delta\sup_{I}
\Vert \widehat{f}^n(t)\Vert_{L^2(\Ozeta;H^1_M(B))}^2
\\
&\leq c(\delta)
\sup_{I}\Big(
\Vert  \mathbf{w}(t)\Vert_{W^{7/4,2}(\Ozeta)}^2
\Vert \widehat{f}^n(t)\Vert_{W^{1,2}(\Ozeta;L^2_M(B))}^2\Big)+\delta\sup_{I}
\Vert \widehat{f}^n(t)\Vert_{L^2(\Ozeta;H^1_M(B))}^2
\\
&\leq  
c(\delta)
\exp\bigg(c\int_I \Vert  \mathbf{w} \Vert_{W^{5/2+,2}(\Ozeta)}^2 \dt
+
c
\sup_{I}
\Vert \mathbf{w}\Vert_{W^{7/4,2}(\Ozeta)}^2
 +
 c\int_I\Vert\partial_t\zeta\Vert_{W^{2/3,2}(\omega)}^{3}\dt \bigg) 
\Vert \widehat{f}_0
\Vert_{W^{1,2}(\Omega_{\zeta(0)};L^2_M(B))}^2\\
&+\delta\sup_{I}
\Vert \widehat{f}^n(t)\Vert_{L^2(\Ozeta;H^1_M(B))}^2.
\end{align*}
%where in the second estimate, we use the embedding $W^{2,2}(I;L^2(\Ozeta))\cap L^2(I;W^{4,2}(\Ozeta))\hookrightarrow L^\infty(I;W^{1,\infty}(\Ozeta))$ to bound the velocity term and using Lemma \ref{lem:mainFP1} to bound the Fokker--Planck term.
By the trace theorem, \eqref{fokkerPlanckDataAlone}--\eqref{fokkerPlanckDataAloneB}, Lemma \ref{lem:mainFP1} and Lemma \ref{lem:mainFP2}, 
\begin{align*}
I_5^4
&\lesssim
\int_I\Vert \bn_\zeta\cdot((\partial_t\zeta)\bn)\circ\bm{\varphi}_\zeta^{-1}\Vert_{L^4(\partial\Ozeta)}^2
\Vert  \nabx \mathbf{w}\Vert_{L^4(\partial\Ozeta)}^2
 \Vert   \widehat{f}^n\Vert_{L^4(\partial\Ozeta;L^2_M(B))}^2 \dt
+
\int_I
\Vert \widehat{f}^n\Vert_{L^4(\partial\Ozeta;H^1_M(B))}^2\dt
\\&
\lesssim
\int_I\Vert\partial_t\zeta\Vert_{W^{1/2,2}(\omega)}^2\Vert   \nabla\mathbf{w}\Vert_{W^{1/2,2}( \partial\Ozeta)}^2\Vert  \widehat{f}^n\Vert_{W^{1,2}(\Ozeta;L^2_M(B))}^2 \dt
+
\int_I
\Vert \widehat{f}^n\Vert_{W^{1,2}(\Ozeta;H^1_M(B))}^2\dt
\\
&\lesssim\int_I\Vert\partial_t\zeta\Vert_{W^{1/2,2}(\omega)}^2\Vert   \mathbf{w}\Vert_{W^{2,2}( \Ozeta)}^2 \dt\exp\bigg(c\int_I
\Vert   \mathbf{w}\Vert_{W^{5/2+\kappa,2}(\Ozeta)}^2\dt
+
c
\int_I\|\partial_t\zeta\|_{W^{2/3,2}(\omega)}^{3}\dt\bigg)
\\&
\qquad\times\|\widehat f_0\|^2_{W^{1,2}(\Omega_{\zeta(0)};L^2_M(B))}
\\
&\lesssim\exp\bigg(c\int_I
\Vert   \mathbf{w}\Vert_{W^{5/2+\kappa,2}(\Ozeta)}^2\dt+c
\sup_I\|\partial_t\zeta\|_{W^{2/3,2}(\omega)}^{3} \bigg)\|\widehat f_0\|^2_{W^{1,2}(\Omega_{\zeta(0)};L^2_M(B))}.
\end{align*}
Collecting all estimates, we obtain the desired estimate \eqref{fokkerEnergyNeg2}.  
%\begin{equation}
%\begin{aligned}
%\int_I&\Vert \partial_t\widehat{f}^n
%\Vert_{L^2(\Ozeta;L^2_M(B))}^2\dt
%+
%\sup
%_{t\in I}\Vert \nabx\widehat{f}^n(t)
%\Vert_{L^2(\Ozeta;L^2_M(B))}^2
%+
%\sup
%_{t\in I}\Vert  \widehat{f}^n(t)
%\Vert_{L^2(\Ozeta;H^1_M(B))}^2
%\\
%&\lesssim \exp\bigg(c\int_I
%\Vert   \mathbf{w}\Vert_{W^{4,2}(\Ozeta)}^2\dt+c\Vert   \mathbf{w}\Vert_{W^{2,2}(I;L^{2}(\Ozeta))}^2+c
%\sup_I\|\partial_t\zeta\|_{W^{2,2}(\omega)}^{3}\bigg)\|\widehat f_0\|^2_{W^{1,2}(\Omega_{\zeta(0)};L^2_M(B))},
%%&\lesssim 1+
%% \Vert  \mathbf{w}_0\Vert_{W^{3,2}(\Omega_{\zeta(0)})}\Vert \widehat{f}_0\Vert_{L^2(\Omega_{\zeta(0)};L^2_M(B))}\Vert \widehat{f}_0\Vert_{L^2(\Omega_{\zeta(0)};H^1_M(B))}
%%\\&+
%%\int_I 
%%\Vert    \mathbf{w}\Vert_{W^{3,2}(\Ozeta)}^2
%%\Vert \widehat{f}^n\Vert_{L^2(\Ozeta;H^1_M(B))}^2\dt.
%\end{aligned}
%\end{equation}
%which finishes the proof.
\end{proof}

\begin{lemma}\label{lem:mainFP4}
Let $(\widehat{f}_0,\zeta, \mathbf{w})$ satisfy  \eqref{fokkerPlanckDataAlone}--\eqref{fokkerPlanckDataAloneB}  and let $\widehat{f}^n$ be the corresponding solution to \eqref{eq:FP1Approx}.
Suppose further that   $\widetilde{f}_0$ satisfies \eqref{eq:FP1Approx1Initial}. 
Then we have 
\begin{equation}
\begin{aligned}
\label{fokkerEnergyNeg3}
\sup_I &\Vert \partial_t\widehat{f}^n(t)
\Vert_{L^2(\Ozeta;L^2_M(B))}^2+
 \int_I 
 \Vert \partial_t\widehat{f}^n
\Vert_{W^{1,2}(\Ozeta;L^2_M(B))}^2
 \dt
 +
 \int_I 
 \Vert \partial_t\widehat{f}^n
\Vert_{L^2(\Ozeta;H^1_M(B))}^2
 \dt
\\
&\lesssim\exp\bigg(c\int_I
\Vert   \mathbf{w}\Vert_{W^{5/2+\kappa,2}(\Ozeta)}^2\dt
+
c\int_I 
\Vert\partial_t\mathbf{w}\Vert_{W^{3/2+\kappa,2}(\Ozeta )}^2\dt
+c
\, \sup_I\|\partial_t\zeta\|_{W^{1,2}(\omega)}^{3} \bigg) 
\\&
\qquad
\times \Big( \|\widehat f_0\|^2_{W^{1,2}(\Omega_{\zeta(0)};L^2_M(B))}+
\Vert \widetilde{f}_0
\Vert_{L^{2}(\Omega_{\zeta(0)};L^2_M(B))}^2\Big)
\end{aligned}
\end{equation}
for all $\kappa\in(0,1/2)$ uniformly in $n\in\mathbb{N}$.  
\end{lemma}
\begin{proof}
Now  set $\widetilde{f}^n:=\partial_t\widehat{f}^n$ and consider the following equation
\begin{equation}
\begin{aligned}
\label{eq:FP1Approx1}
M\big(\partial_t \widetilde{f}^n + (\mathbf{w}\cdot \nabx)  \widetilde{f}^n)
&+
 \divq  \big( \chi^n (  \nabx\mathbf{w}) \bq M\widetilde{f}^n \big) 
-
\Delx(M \widetilde{f}^n)
-
 \divq  \big( M \nabq \widetilde{f}^n
\big)
\\&
=
-
M (\partial_t\mathbf{w}\cdot \nabx)   \widehat{f}^n
+
 \divq  \big( \chi^n (  \partial_t\nabx\mathbf{w}) \bq M\widehat{f}^n \big) 
\end{aligned}
\end{equation}
in $I\times\Omega_\zeta\times B$ subject to%\todo{Note \eqref{fokkerPlankBoundaryxsc4time} because of the cut-off!}
\begin{align}
&\widetilde{f}^n(0, \cdot, \cdot) =\widetilde{f}_0 \geq 0
& \quad \text{in }\Omega_{\zeta_0} \times B,
\label{fokkerPlankIintialxsec4time}
\\
&
\nabx\widetilde{f}^n\cdot \bn_\zeta =-\nabx\widehat{f}^n\cdot \partial_t\bn_\zeta
&\quad \text{on }I \times \partial\Omega_\zeta \times B,
\label{fokkerPlankBoundarySpacexsc4time}
\\
%&\big(\nabq\widetilde{f}^n -  (\nabx \mathbf w) \bq \widetilde{f}^n -  (\nabx \partial_t\mathbf w) \bq \widehat{f}^n
% \big) \cdot \bn =0
& M\nabq\widetilde{f}^n \cdot \frac{\bq}{\vert\bq\vert} =0
&\quad \text{on }I \times \Omega_\zeta \times \partial \overline{B}
\label{fokkerPlankBoundaryxsc4time}
\end{align}
and where $\widetilde{f}_0$ satisfies \eqref{eq:FP1Approx1Initial}. 
%\begin{equation}
%\begin{aligned}
%\label{eq:FP1Approx1Initial}
%M   \widetilde{f}_0 
%&=
%\Delx(M  \widehat{f}_0)
%+
% \divq  \big( M \nabq    \widehat{f}_0
%\big)
%-
%M  (\mathbf{w}_0\cdot \nabx)   \widehat{f}_0
%-
% \divq  \big(   (  \nabx\mathbf{w}_0) \bq M\widehat{f}_0 \big) .
%\end{aligned}
%\end{equation}
We now test \eqref{eq:FP1Approx1} with $\widetilde{f}^n$. Since the left-hand side of \eqref{eq:FP1Approx1} is of the same form as \eqref{eq:FP1Approx}, we obtain similarly to \eqref{fokkerEnergy}
\begin{equation}
\begin{aligned}
\label{fokkerEnergy1}
\int_I\int_{\Ozeta}
&\Big(
\frac{1}{2}\partial_t\Vert   \widetilde{f}^n \Vert_{L^2_M(B)}^2
+
\frac{1}{2}(\mathbf{w}\cdot \nabx)\Vert    \widetilde{f}^n \Vert_{L^2_M(B)}^2
+
\Vert   \nabx \widetilde{f}^n \Vert_{L^2_M(B)}^2
+ 
\frac{1}{2}
\Vert  \widetilde{f}^n\Vert_{H^1_M(B)}^2
\Big)\dx\dt
\\&
\lesssim
\int_I 
\Vert  \nabx\mathbf{w} \Vert_{L^\infty(\Ozeta)}^2\Vert\widetilde{f}^n
\Vert_{L^2(\Ozeta;L^2_M(B))}^2\dt
-
\int_I\int_{\partial\Ozeta\times B}
    M\partial_t\bn_\zeta\cdot\nabx\widehat{f}^n 
  \widetilde{f}^n\dq\dx\dt
\\&
+
\int_I\int_{\Ozeta\times B}
  \chi^n (  \partial_t\nabx\mathbf{w}) \bq M\widehat{f}^n \cdot \nabq
  \widetilde{f}^n\dq\dx\dt
  -
\int_I\int_{\Ozeta\times B} M
(\partial_t\mathbf{w}\cdot \nabx)   \widehat{f}^n
  \widetilde{f}^n\dq\dx\dt,
\end{aligned}
\end{equation}
where the second term on the right-hand side is due to \eqref{fokkerPlankBoundarySpacexsc4time}. For the boundary term, we use the trace theorem and Lemma \ref{lem:mainFP2} to obtain
\begin{align*}
\bigg\vert\int_I&\int_{\partial\Ozeta\times B}
    M\partial_t\bn_\zeta\cdot\nabx\widehat{f}^n 
  \widetilde{f}^n\dq\dx\dt
  \bigg\vert
\\
&\leq
\int_I 
\Vert \widetilde{f}^n\Vert_{L^{4}(\partial\Ozeta;L^2_M(B))}\|\partial_t\zeta\|_{W^{1,2}(\omega)}
\Vert \nabx \widehat{f}^n\Vert_{L^{4}(\partial\Ozeta;L^2_M(B))}\dt
\\
&\leq
\delta
\int_I 
\Vert \widetilde{f}^n\Vert_{W^{1,2}(\Ozeta;L^2_M(B))}^2\dt
+c(\delta)\sup_I\|\partial_t\zeta\|_{W^{1,2}(\omega)}^2
\int_I 
\Vert \nabx \widehat{f}^n\Vert_{W^{1,2}(\Ozeta;L^2_M(B))}^2\dt
\\
&\leq
\delta
\int_I 
\Vert \widetilde{f}^n\Vert_{W^{1,2}(\Ozeta;L^2_M(B))}^2\dt
\\&\quad+
c(\delta)\exp\bigg(c\int_I
\Vert   \mathbf{w}\Vert_{W^{5/2+\kappa,2}(\Ozeta)}^2\dt
+
c\,\sup_I\|\partial_t\zeta\|_{W^{1,2}}^{3} \bigg)\|\widehat f_0\|^2_{W^{1,2}(\Omega_{\zeta(0)};L^2_M(B))}.
\end{align*}
Next we use \eqref{fokkerPlanckDataAlone} and Lemma \ref{lem:mainFP2} to obtain
%\todo{the bound $x\leq \exp(x)$ is trivial, shouldn't be mentioned}
\begin{equation}
\begin{aligned}
\bigg\vert\int_I&\int_{ \Ozeta\times B}
  \chi^n (  \partial_t\nabx\mathbf{w}) \bq M\widehat{f}^n \cdot \nabq
  \widetilde{f}^n\dq\dx\dt
  \bigg\vert
\\&\leq
c(\delta)
\int_I\int_{ \Ozeta }
\vert\partial_t\nabx\mathbf{w}\vert^2
\Vert  \widehat{f}^n\Vert_{L^2_M(B)}^2\dx\dt+\delta
\int_I\Vert    \widetilde{f}^n\Vert_{L^2(\Ozeta;H^1_M(B))}^2\dt
\\&
\leq
c
\int_I 
\Vert\partial_t\nabx\mathbf{w}\Vert_{L^3(\Ozeta )}^2\dt
\sup_I
\Vert  \widehat{f}^n\Vert_{W^{1,2}(\Ozeta;L^2_M(B))}^2 +\delta
\int_I\Vert    \widetilde{f}^n\Vert_{L^2(\Ozeta;H^1_M(B))}^2\dt
\\&
\leq
c \exp\bigg(c\int_I
\Vert   \mathbf{w}\Vert_{W^{5/2+\kappa,2}(\Ozeta)}^2\dt
+
c\int_I 
\Vert\partial_t\nabx\mathbf{w}\Vert_{L^3(\Ozeta )}^2\dt
+
c\int_I\|\partial_t\zeta\|_{W^{1,2}(\omega)}^{3}\dt\bigg)
\\
&\quad \times
\|\widehat f_0\|^2_{W^{1,2}(\Omega_{\zeta(0)};L^2_M(B))}
+
\delta
\int_I\Vert    \widetilde{f}^n\Vert_{L^2(\Ozeta;H^1_M(B))}^2\dt.
\end{aligned}
\end{equation} 
Finally, we use Lemma \ref{lem:mainFP1} to also obtain
\begin{align*}
\bigg\vert\int_I\int_{ \Ozeta\times B}& M
(\partial_t\mathbf{w}\cdot \nabx)   \widehat{f}^n
  \widetilde{f}^n\dq\dx\dt
  \bigg\vert
\\&\lesssim
\int_I 
\Vert \nabx \widehat{f}^n\Vert_{L^2(\Ozeta;L^2_M(B))}^2\dt
+
\int_I\int_{ \Ozeta }
\vert\partial_t\mathbf{w}\vert^2
\Vert \widetilde{f}^n\Vert_{L^2_M(B)}^2\dx\dt
\\&\lesssim\exp\bigg(c\int_I
\Vert   \mathbf{w}\Vert_{W^{1,\infty}(\Ozeta)}^2\dt\bigg)\|\widehat f_0\|^2_{L^{2}(\Omega_{\zeta(0)};L^2_M(B))}
\\&\quad+
\int_I 
 \Vert\partial_t\mathbf{w}\Vert_{L^{\infty}(\Ozeta )}^2 
\Vert \widetilde{f}^n\Vert_{L^2(\Ozeta ;L^2_M(B))}^2 \dt.
\end{align*} 
Subsequently, similar to \eqref{fokkerEnergyEst5}, we use Reynold's transport theorem and the embedding $$W^{3/2+\kappa,2}(\Omega_\zeta)\hookrightarrow L^\infty(\Omega_\zeta)\cap W^{1,3}(\Omega_\zeta)$$ for any $\kappa>0$ and obtain
\begin{equation}
\begin{aligned}
\label{fokkerEnergyEst6}
\sup_I &\Vert\widetilde{f}^n
\Vert_{L^2(\Ozeta;L^2_M(B))}^2
+
 \int_I 
 \Vert \widetilde{f}^n
\Vert_{W^{1,2}(\Ozeta;L^2_M(B))}^2
 \dt
 +
 \int_I 
 \Vert \widetilde{f}^n
\Vert_{L^2(\Ozeta;H^1_M(B))}^2
 \dt
\\&\lesssim 
\exp\bigg(c\int_I
\Vert   \mathbf{w}\Vert_{W^{5/2+\kappa,2}(\Ozeta)}^2\dt
+
c 
\int_I 
\Vert\partial_t\mathbf{w}\Vert_{W^{3/2+\kappa,2}(\Ozeta )}^2\dt 
 +c\,
 \sup_I \|\partial_t\zeta\|_{W^{1,2}(\omega)}^{3}\bigg)
\\
& \quad
\times\|\widehat f_0\|^2_{W^{1,2}(\Omega_{\zeta(0)};L^2_M(B))}
+
\int_I \Big(\Vert\partial_t\mathbf{w}\Vert_{W^{3/2+\kappa,2}(\Ozeta )}^2 
+
\Vert  \mathbf{w} \Vert_{W^{1,\infty}(\Ozeta)}^2
\Big)
\Vert\widetilde{f}^n
\Vert_{L^2(\Ozeta;L^2_M(B))}^2 \dt.
\end{aligned}
\end{equation}
Applying Gr\"onwall's lemma yields the claim. 
\end{proof}

\section{The fully coupled system}
\label{sec:fullyCoupled}
\noindent 
In the section  we use  yet again, a fixed point argument to establish the existence of a unique local strong solution to the fully mutually coupled solute-solvent-structure system. As already shown in Section \ref{sec:solveSolventStructure}, such a fixed point argument requires showing the closedness of an anticipated solution in a ball and a contraction argument. These two properties will be shown in two different spaces where one space is a subspace of the other. More precisely, we consider
\begin{align*}
X&:=  L^{\infty}\big(I;W^{1,2}(\Oeta;L^2_M(B))  \big)
\cap
L^{2}\big(I;W^{2,2}(\Oeta;L^2_M(B))  \big)
\cap L^{2}\big(I;W^{1,2}(\Oeta;H^1_M(B))  \big)\\
&\cap W^{1,\infty}\big(I;L^{2}(\Oeta;L^2_M(B))  \big)
\cap
W^{1,2}\big(I;W^{1,2}(\Oeta;L^2_M(B))  \big)
\cap W^{1,2}\big(I;L^{2}(\Oeta;H^1_M(B))  \big),
\\
Y&:=  L^{\infty}\big(I;L^2(\Oeta;L^2_M(B))  \big)
\cap 
L^2\big(I;W^{1,2}(\Oeta;L^2_M(B))  \big)
\cap
L^2\big(I;L^2(\Oeta;H^1_M(B))  \big) 
\end{align*}
equipped with their canonical norms $\Vert \cdot\Vert_X$ and $\Vert \cdot\Vert_Y$, respectively.
Here, and in what follows, we have abused notation by reusing $I:=I_{**}$ where $I_{**}=(0,T_{**})$ is such that $T_{**}<T_*$ with $I_*=(0,T_*)$ being the local time on which the solution to the purely solvent-structure system was constructed in Section \ref{sec:solveSolventStructure}. Accordingly, we also abuse notation and set $T:=T_{**}$.
\\
For the purpose of a contraction argument, which is to be performed in the larger space $Y$, it is convenient to transform the moving domain to the fixed reference domain. For this reason, we also introduce the space
\begin{align*}
\overline{Y}:=  L^{\infty}\big(I;L^2(\Omega;L^2_M(B))  \big)
\cap 
L^2\big(I;W^{1,2}(\Omega;L^2_M(B))  \big)
\cap 
L^2\big(I;L^2(\Omega;H^1_M(B))  \big) 
\end{align*}
equipped with its canonical norm  $\Vert \cdot\Vert_{\overline{Y}}$
\\ 
Now, for ${\hbar}\in Y$, let $(\eta,\bu,\pi)$  be a unique solution of \eqref{contEqAlone}--\eqref{momEqAlone} with dataset $(\bff, g, \eta_0, \eta_\star, \bu_0, \mathbb{S}_\bq( \hbar))$ as shown in Section \ref{sec:solveSolventStructure}. On the other hand, for
%\todo{the green has been increased from $W^{2,2}$ to $W^{3,2}$ because of estimates such as (5.11). See blue explanation below (5.11) }
%\todo{we need the blue because of e.g. (5.2)}
\begin{align*}
(\eta,\bu)\in  
W^{1,\infty}(I&;W^{3,2}(\omega))
{
 \cap W^{2,2}\big(I ;W^{1,2}(\omega)  \big)
 }
\times
W^{1,\infty}(I;W^{1,2}_{\divx} ( \Oeta))
\\&
\cap L^2(I; W^{3,2}(\Oeta)) \cap    W^{1,2}(I;W^{2,2} ( \Oeta)),
\end{align*}
let $\widehat{f}$ be the solution of \eqref{eq:FP} with dataset $(\widehat{f}_0,\eta,\bu)$ as shown in Section \ref{sec:FP}. Now define the mapping $\mathtt{T}=\mathtt{T}_1\circ\mathtt{T}_2$ where
\begin{align*}
\mathtt{T}( \hbar)=\widehat{f}, \qquad \mathtt{T}_2( \hbar)=(\eta,\bu,\pi), \qquad \mathtt{T}_1(\eta,\bu,\pi)=\widehat{f}
\end{align*}
and let
\begin{align*}
B_R:=\big\{  \hbar\in X \,:\,\Vert \hbar\Vert_X^2 \leq R^2  \big\}.
\end{align*}
Let show that $\mathtt{T}:X\rightarrow X$ maps $B_R$ into $B_R$, i.e., for any $ \hbar \in B_R$, we have that 
\begin{align*}
\Vert \widehat{f}\Vert_X^2=\Vert T( \hbar) \Vert_X^2=\Vert T_1\circ T_2(\hbar) \Vert_X^2=\Vert T_1(\eta,\bu, \pi) \Vert_X^2 \leq R^2.
\end{align*}
Indeed if we let $\hbar \in B_R$ then by the a priori estimate  \eqref{eq:thm:mainFP},
\begin{equation}
\begin{aligned} 
&
 \sup_I \Vert \partial_t\widehat{f}(t)
\Vert_{L^{2}(\Oeta;L^2_M(B))}^2+
 \int_I 
 \Vert \partial_t\widehat{f}
\Vert_{W^{1,2}(\Oeta;L^2_M(B))}^2
 \dt
 +
 \int_I 
 \Vert \partial_t\widehat{f}
\Vert_{L^{2}(\Oeta;H^1_M(B))}^2
 \dt
 \\
 &+
 \sup_I \Vert \widehat{f}(t)
\Vert_{W^{1,2}(\Oeta;L^2_M(B))}^2+
 \int_I 
 \Vert \widehat{f}
\Vert_{W^{2,2}(\Oeta;L^2_M(B))}^2
 \dt
 +
 \int_I 
 \Vert \widehat{f}
\Vert_{W^{1,2}(\Oeta;H^1_M(B))}^2
 \dt
 \\ 
&\quad\lesssim  \exp\bigg(c\int_I
\Vert   \mathbf{u}\Vert_{W^{5/2+\kappa,2}(\Oeta)}^2\dt
+ 
c\int_I \Vert\partial_t\mathbf{u}\Vert_{W^{3/2+\kappa,2 }(\Oeta )}^2\dt
 \bigg)
\\ 
&\qquad\qquad
\times \exp\bigg( c
\sup_{I}
\Vert  \mathbf{u}\Vert_{ W^{7/4,2}(\Oeta) }^2
+c
\sup_I\|\partial_t\eta\|_{W^{1,2}(\omega)}^{3} \bigg)
\\
&\qquad\qquad
\times \Big( \|\widehat f_0\|^2_{W^{1,2}(\Omega_{\eta_0};L^2_M(B))}+
\Vert \widetilde{f}_0
\Vert_{L^{2}(\Omega_{\eta_0};L^2_M(B))}^2\Big)\label{aprioriFP}
\end{aligned}
\end{equation} 
%
%\begin{equation}\label{aprioriFP'}
%\begin{aligned}
%&
% \sup_{t\in I} \Vert \partial_t\widehat{f}(t)
%\Vert_{W^{1,2}(\Oeta;L^2_M(B))}^2+
% \int_I 
% \Vert \partial_t\widehat{f}
%\Vert_{W^{2,2}(\Oeta;L^2_M(B))}^2
% \dt
% +
% \int_I 
% \Vert \partial_t\widehat{f}
%\Vert_{W^{1,2}(\Oeta;H^1_M(B))}^2
% \dt
% \\
%&\quad\lesssim  \exp\bigg(c\int_I
%\Vert   \bu\Vert_{W^{3,2}(\Oeta)}^2\dt
%+
%c\sup_I 
%\Vert\partial_t\nabx\bu\Vert_{L^2(\Oeta )}^2
%+
%c\int_I \Vert\partial_t\bu\Vert_{W^{2,2}(\Oeta )}^2\dt
%+c
%\sup_I\|\partial_t\eta\|_{W^{2,2}(\omega)}^{3}\bigg)\\
%&\qquad\qquad\qquad\qquad\qquad\qquad\times \Big( \|\widehat f_0\|^2_{W^{1,2}(\Omega_{\eta(0)};L^2_M(B))}+
%\Vert \widetilde{f}_0
%\Vert_{W^{1,2}(\Omega_{\eta(0)};L^2_M(B))}^2\Big).
%\end{aligned}
%\end{equation}
We now aim to derive an  estimate for the terms in the exponential. By interpolation we obtain for some $\theta_1\in(0,1)$
\begin{align*}
\int_I
\Vert   \mathbf{u}\Vert_{W^{5/2+\kappa,2}(\Oeta)}^2\dt&\leq \int_I
\Vert   \mathbf{u}\Vert_{W^{2,2}(\Oeta)}^{2\theta_1}\Vert   \mathbf{u}\Vert_{W^{3,2}(\Oeta)}^{2(1-\theta_1)}\dt
\\
&\leq \bigg(\int_I
\Vert   \mathbf{u}\Vert_{W^{2,2}(\Oeta)}^{2}\dt\bigg)^{\theta_1}\bigg(\int_I\Vert   \mathbf{u}\Vert_{W^{3,2}(\Oeta)}^{2}\dt\bigg)^{1-\theta_1}
\\
&\leq T^{\theta_1}\sup_I
\Vert   \mathbf{u}\Vert_{W^{2,2}(\Oeta)}^{2\theta_1}\bigg(\int_I\Vert   \mathbf{u}\Vert_{W^{3,2}(\Oeta)}^{2}\dt\bigg)^{1-\theta_1}\\
&\lesssim T^{\theta_1}\bigg(\sup_I
\Vert   \mathbf{u}\Vert_{W^{2,2}(\Oeta)}^{2}+\int_I\Vert   \mathbf{u}\Vert_{W^{3,2}(\Oeta)}^{2}\dt\bigg)
\end{align*}
and, similarly, for some $\theta_2,\theta_3\in(0,1)$,
\begin{align*}
\int_I
\Vert   \partial_t\mathbf{u}\Vert_{W^{3/2+\kappa,2}(\Oeta)}^2\dt
&\leq 
\int_I \Vert   \partial_t\mathbf{u}\Vert_{W^{1,2}(\Oeta)}^{2\theta_2}\Vert   \partial_t\mathbf{u}\Vert_{W^{2,2}(\Oeta)}^{2(1-\theta_2)}\dt
\\
&\lesssim T^{\theta_2}\bigg(\sup_I
\Vert   \partial_t\mathbf{u}\Vert_{W^{1,2}(\Oeta)}^{2}+\int_I\Vert   \partial_t\mathbf{u}\Vert_{W^{2,2}(\Oeta)}^{2}\dt\bigg),
\\
\sup_I
\Vert   \mathbf{u}\Vert_{W^{7/4,2}(\Oeta)}^2
&\lesssim
\sup_I
\Vert   \mathbf{u}\Vert_{W^{2,2}(\Oeta)}^2
\\&\lesssim
\Vert   \mathbf{u}\Vert_{L^2(I;W^{2,2}(\Oeta))}^{2\theta_3}\Vert\mathbf{u}\Vert_{W^{1,2}(I;W^{2,2}(\Oeta))}^{2(1-\theta_3)}
\\
&\leq T^{\theta_3}\bigg(\Vert   \mathbf{u}\Vert_{L^\infty(I;W^{2,2}(\Oeta))}^{2}+\Vert\mathbf{u}\Vert^2_{W^{1,2}(I;W^{2,2}(\Oeta))}\bigg)\\
&\lesssim T^{\theta_3}\Vert\mathbf{u}\Vert^2_{W^{1,2}(I;W^{2,2}(\Oeta))}.
%\int_I
%\Vert   \partial_t\mathbf{u}\Vert_{W^{1,3+\kappa}(\Ozeta)}^2\dt
%&\leq 
%\int_I \Vert   \partial_t\mathbf{u}\Vert_{W^{1,2}(\Ozeta)}^{2\theta_2}\Vert   \partial_t\mathbf{u}\Vert_{W^{1,6}(\Ozeta)}^{2(1-\theta_2)}\dt
%\\
%&\lesssim T^{\theta_2}\bigg(\sup_I
%\Vert   \partial_t\mathbf{u}\Vert_{W^{1,2}(\Ozeta)}^{2}+\int_I\Vert   \partial_t\mathbf{u}\Vert_{W^{2,2}(\Ozeta)}^{2}\dt\bigg),
%\\
%\sup_I
%\Vert   \mathbf{u}\Vert_{W^{7/4,2}(\Ozeta)}^2&\lesssim\Vert   \mathbf{u}\Vert_{L^2(I;W^{2,2}(\Ozeta))}^{2\theta_3}\Vert\mathbf{u}\Vert_{W^{1,2}(I;W^{2,2}(\Ozeta))}^{2(1-\theta_3)}\\
%&\leq T^{\theta_3}\bigg(\Vert   \mathbf{u}\Vert_{L^\infty(I;W^{2,2}(\Ozeta))}^{2}+\Vert\mathbf{u}\Vert^2_{W^{1,2}(I;W^{2,2}(\Ozeta))}\bigg)\\
%&\lesssim T^{\theta_3}\Vert\mathbf{u}\Vert^2_{W^{1,2}(I;W^{2,2}(\Ozeta))}.
\end{align*}
Finally, we have for some $\theta_4\in(0,1)$ 
\begin{align*}
\sup_I
\Vert    \partial_t\eta\Vert_{W^{1,2}(\omega)}^2&\leq \Vert    \partial_t\eta\Vert_{L^2(I;W^{1,2}(\omega))}^{2\theta_4}\Vert    \partial_t\eta\Vert_{W^{1,2}(I;W^{1,2}(\omega))}^{2(1-\theta_4)}
\\
&\leq T^{\theta_4}\Vert    \partial_t\eta\Vert_{L^\infty(I;W^{1,2}(\omega))}^{2\theta_4}\Vert    \partial_t\eta\Vert_{W^{1,2}(I;W^{1,2}(\omega))}^{2(1-\theta_4)}\\
&\leq T^{\theta_4}\bigg(\Vert    \partial_t\eta\Vert_{L^\infty(I;W^{1,2}(\omega))}^{2}+\Vert    \partial_t\eta\Vert_{W^{1,2}(I;W^{1,2}(\omega))}^{2}\bigg)
\\
&\lesssim T^{\theta_4}\Vert   \eta\Vert_{W^{2,2}(I;W^{1,2}(\omega))}^{2}.
\end{align*}
By  Theorem \ref{thm:fluidStructureWithoutFK}, we can control
%Due to \eqref{energyEstLinearTilde} and \eqref{estForEvolvingInitialCondition} (recall that $\tilde{\bu}=\partial_t\overline{\bu}$ and so on), we can infer that
\begin{equation}
\begin{aligned}
\label{twoVeloOneShell}
\Vert\mathbf{u}\Vert^2_{L^{2}(I;W^{3,2}(\Oeta))}
&+\Vert\mathbf{u}\Vert^2_{W^{1,\infty}(I;W^{1,2}(\Oeta))}+
\Vert\mathbf{u}\Vert^2_{W^{1,2}(I;W^{2,2}(\Oeta))}
+
\Vert   \eta\Vert_{W^{2,2}(I;W^{1,2}(\omega))}^{2}
%\\&\lesssim
%\mathcal{E}_*(\bff,g, \eta_0, \eta_\star, \bu_0)
%+
%\int_{I} 
%\Vert \mathbb{S}_\bq( \partial_t\hbar)
%\Vert_{W^{1,2}(\Oeta)}^2
%\dt
\end{aligned}
\end{equation}
in terms of
\begin{equation}
\begin{aligned}
\label{maximalThisPaper}
\mathcal{E}_{**}(\bff,g, \eta_0, \eta_\star, \bu_0)
+
\int_{I} 
\Big(
\Vert \mathbb{S}_\bq(\partial_t\hbar)
\Vert_{W^{1,2}(\Oeta)}^2
+ 
\Vert \mathbb{S}_\bq(\hbar)
\Vert_{W^{2,2}(\Oeta)}^2
\Big)
\dt,
\end{aligned}
\end{equation} 
where
\begin{equation}
\begin{aligned}
\label{bigDataNorm*E}
\mathcal{E}_{**}(\bff,g, \eta_0, \eta_\star, \bu_0)&:=
\Vert \eta_\star\Vert_{W^{3,2}(\omega)}^2
 +
 \Vert \eta_0\Vert_{W^{5,2}(\omega)}^2
  +
 \Vert
  \bu_0\Vert_{W^{3,2}(\Omega_{\eta_0})}^2
  +
  \Vert \mathbb{S}_\bq(\hbar(0))
  \Vert_{W^{2,2}(\Omega_{\eta_0})}^2
   \\&
    +
   \Vert \bff(0)\Vert_{W^{1,2}(\Omega_{\eta_0})}^2    +
 \Vert
  g(0)\Vert_{W^{1,2}(\omega)}^2
   \\&+
   \int_I \big(\Vert \partial_t g \Vert_{W^{1,2}(\omega)}^2
  +
 \Vert \partial_t\bff\Vert_{L^2(\Oeta)}^2
    \big)
\dt.
\end{aligned}
\end{equation}
However, by using \cite[(3.4)]{masmoudi2008well} (for $k>1$), we obtain 
\begin{equation}
\begin{aligned}
\label{aprioriFP1}
\int_{I} 
\Big(
\Vert
\mathbb{S}_\bq(\partial_t\overline{\hbar} )
\Vert_{W^{1,2}(\Oeta)}^2
&+
\Vert \mathbb{S}_\bq( \hbar)
\Vert_{W^{2,2}(\Oeta)}^2\Big)
\dt
\\&\lesssim
 \int_{I}
 \Vert  \partial_t\hbar
\Vert_{W^{1,2}(\Oeta;L^2_M(B))}^2
\dt
+
\int_{I}
 \Vert  \hbar
\Vert_{W^{2,2}(\Oeta;L^2_M(B))}^2
\dt\leq R^2.
% \lesssim T
% \sup_{I}
% \Vert \partial_t \hbar
%\Vert_{W^{1,2}(\Oeta;L^2_M(B))}^2
%+
%T
% \int_{I}
% \Vert \partial_t\hbar
%\Vert_{W^{2,2}(\Oeta;L^2_M(B))}^2
% \dt
%  \\&
% \leq
%cTR.
\end{aligned}
\end{equation}
%Now let
%\begin{align*}
%\Xi:=c\exp(c\,\mathcal{E}_{**}(\bff,g, \eta_0, \eta_\star, \bu_0))\Big( \|\widehat f_0\|^2_{W^{1,2}(\Omega_{\eta_0};L^2_M(B))}+
%\Vert \widetilde{f}_0
%\Vert_{W^{1,2}(\Omega_{\eta_0};L^2_M(B))}^2\Big).
%\end{align*}
For  small enough $T$ and for $R$ very large, we obtain from \eqref{aprioriFP}-- \eqref{aprioriFP1} that $\Vert \widehat{f}\Vert_X^2 \leq R$. Thus, $\mathtt{T}:X\rightarrow X$ maps $B_R$ into $B_R$.
\\
Next, we show that $\mathtt{T}$ is a contraction in the larger space $Y$. For this reason, on the one hand, 
we let $\widehat{f}^1$ and $\widehat{f}^2$ be two solutions of \eqref{eq:FP} with data $(\widehat{f}_0,\eta^1, \bu^1)$ and $(\widehat{f}_0,\eta^2, \bu^2)$, respectively.
On the other hand, we let
$(\eta^1,\bu^1,\pi^1)$ and $(\eta^2,\bu^2,\pi^2)$  be two solutions of \eqref{contEqAlone}-\eqref{momEqAlone} with dataset $(\bff, g, \eta_0, \eta_\star, \bu_0, \mathbb{S}_\bq(\hbar^1))$ and $(\bff, g, \eta_0, \eta_\star, \bu_0, \mathbb{S}_\bq(\hbar^2))$, respectively.
For the former, by setting $\widehat{f}^{12}=\widehat{f}^1-\widehat{f}^2$, $\bu^{12}=\bu^1-\bu^2$ and $\eta^{12}=\eta^1-\eta^2$, we wish to obtain a bound for $\widehat{f}^{12}$ in the norm $Y$ in terms of a suitable norm of  $\bu^{12}$ and $\eta^{12}$.
This bound is achieved by transforming the equation \eqref{eq:FP} from the moving domain to the fixed reference domain and obtaining the equivalent bound for $\overline{\widehat{f}^{12}}$ in the norm $\overline{Y}$ in terms of a suitable norm of   $\overline{\bu}^{12}$ and $\overline{\eta}^{12}$.
Here, $\overline{\widehat{f}^{12}}=\overline{\widehat{f}^1}-\overline{\widehat{f}^2}$ and $\overline{\bu}^{12}=\overline{\bu}^1-\overline{\bu}^2$ with $\overline{\widehat{f}^i}= \widehat{f}^i\circ\bm{\Psi}_{\eta^i}$ and $\overline{\bu}^i=\bu^i\circ\bm{\Psi}_{\eta^i} $, $i=1,2$. Before obtaining this latter bound, let's see how a single strong solution $\widehat{f}$ of   \eqref{eq:FP} with a data $(\widehat{f}_0,\eta, \bu)$ transforms to a fixed domain. The difference equation will then be deduced from the equation of the single equation.
\\
Let $\varphi=\overline{\varphi}\circ\bm{\Psi}_{\eta}^{-1}$, $\overline{\widehat{f}}= \widehat{f}\circ\bm{\Psi}_{\eta}$ and $\overline{\bu}=\bu\circ\bm{\Psi}_{\eta} $. Since $\widehat{f}$ is
a  strong solution of   \eqref{eq:FP} with a data $(\widehat{f}_0,\eta, \bu)$, it follows from \eqref{weakFokkerPlanckEq} that
\begin{align*}
\int_I  \frac{\mathrm{d}}{\dt}
\int_{\Oeta \times B}&M \overline{\widehat{f}}\circ\bm{\Psi}_{\eta}^{-1} \, \overline{\varphi}\circ\bm{\Psi}_{\eta}^{-1} \dq \dx \dt 
\\&=
\int_I\int_{\Oeta \times B}M \overline{\widehat{f}}\circ\bm{\Psi}_{\eta}^{-1} \big(\partial_t  \overline{\varphi}\circ\bm{\Psi}_{\eta}^{-1}
+
\nabx\overline{\varphi}\circ\bm{\Psi}_{\eta}^{-1}\cdot\partial_t\bm{\Psi}_{\eta}^{-1}
\big) \dq \dx \dt
\\&+
\int_I\int_{\Oeta \times B}\big(
M\overline{\bu}\circ\bm{\Psi}_{\eta}^{-1}\overline{\widehat{f}}\circ\bm{\Psi}_{\eta}^{-1} \cdot\nabx\bm{\Psi}_{\eta}^{-1} \nabx\overline{\varphi}\circ\bm{\Psi}_{\eta}^{-1}\big) \dq \dx \dt
\\&
+ \int_I\int_{ \Oeta \times B}
  M (\nabx \bm{\Psi}_{\eta}^{-1})^\top(\nabx\overline{\bu}\circ\bm{\Psi}_{\eta}^{-1})  \bq\overline{\widehat{f}}\circ\bm{\Psi}_{\eta}^{-1} \cdot \nabq\overline{\varphi}\circ\bm{\Psi}_{\eta}^{-1} \dq \dx \dt
\\&
- \int_I\int_{ \Oeta \times B}
  M (\nabx \bm{\Psi}_{\eta}^{-1})^\top \nabx \bm{\Psi}_{\eta}^{-1}\nabx\overline{\widehat{f}}\circ\bm{\Psi}_{\eta}^{-1} \cdot \nabx\overline{\varphi}\circ\bm{\Psi}_{\eta}^{-1} \dq \dx \dt
 \\&
- \int_I\int_{ \Oeta \times B}
  M \nabq \overline{\widehat{f}}\circ\bm{\Psi}_{\eta}^{-1}  \cdot \nabq\overline{\varphi}\circ\bm{\Psi}_{\eta}^{-1} \dq \dx \dt
\end{align*}
and thus,
\begin{align*}
\int_I  \frac{\mathrm{d}}{\dt}
\int_{\Omega \times B}J_\eta M \overline{\widehat{f}}\, \overline{\varphi}  \dq \dx \dt 
&=
\int_I\int_{\Omega \times B}J_\eta M \overline{\widehat{f}}  \big(\partial_t  \overline{\varphi} 
+
\nabx\overline{\varphi} \cdot\partial_t\bm{\Psi}_{\eta}^{-1}\circ\bm{\Psi}_{\eta}
\big) \dq \dx \dt
\\&+
\int_I\int_{\Omega \times B}J_\eta 
M\overline{\bu}\overline{\widehat{f}}  \cdot\nabx\bm{\Psi}_{\eta}^{-1}\circ\bm{\Psi}_{\eta} \nabx\overline{\varphi}   \dq \dx \dt
\\&
+ \int_I\int_{ \Omega \times B}
  J_\eta M (\nabx \bm{\Psi}_{\eta}^{-1}\circ\bm{\Psi}_{\eta})^\top(\nabx\overline{\bu})  \bq\overline{\widehat{f}}  \cdot \nabq\overline{\varphi}  \dq \dx \dt
\\&
- \int_I\int_{ \Omega \times B}
  J_\eta M (\nabx \bm{\Psi}_{\eta}^{-1}\circ\bm{\Psi}_{\eta})^\top \nabx \bm{\Psi}_{\eta}^{-1}\circ\bm{\Psi}_{\eta}\nabx\overline{\widehat{f}}  \cdot \nabx\overline{\varphi} \dq \dx \dt
 \\&
- \int_I\int_{ \Omega \times B}
  J_\eta M \nabq \overline{\widehat{f}}  \cdot \nabq\overline{\varphi} \dq \dx \dt.
\end{align*}
Replace $\overline{\varphi}$ with $\overline{\varphi}/J_\eta$ to obtain
\begin{align*}
\int_I  \frac{\mathrm{d}}{\dt}&
\int_{\Omega \times B}  M \overline{\widehat{f}}\, \overline{\varphi}  \dq \dx \dt 
=
\int_I\int_{\Omega \times B}\big( M \overline{\widehat{f}}  \partial_t  \overline{\varphi} 
+
J_\eta M \overline{\widehat{f}}  \partial_t (J_\eta^{-1}) \overline{\varphi} 
\big) \dq \dx \dt
\\&+
\int_I\int_{\Omega \times B}\big(M \overline{\widehat{f}}  
\nabx  \overline{\varphi} \cdot\partial_t\bm{\Psi}_{\eta}^{-1}\circ\bm{\Psi}_{\eta}
+
J_\eta M \overline{\widehat{f}} 
(\nabx J_{\eta}^{-1})\overline{\varphi} \cdot\partial_t\bm{\Psi}_{\eta}^{-1}\circ\bm{\Psi}_{\eta}
\big) \dq \dx \dt
\\&+
\int_I\int_{\Omega \times B}\big( 
M\overline{\bu}\overline{\widehat{f}}  \cdot\nabx\bm{\Psi}_{\eta}^{-1}\circ\bm{\Psi}_{\eta} \nabx\overline{\varphi} 
+
J_\eta 
M\overline{\bu}\overline{\widehat{f}}  \cdot\nabx\bm{\Psi}_{\eta}^{-1}\circ\bm{\Psi}_{\eta} (\nabx J_\eta^{-1})\overline{\varphi} 
\big)  \dq \dx \dt
\\&
- \int_I\int_{ \Omega \times B}
  M (\nabx \bm{\Psi}_{\eta}^{-1}\circ\bm{\Psi}_{\eta})^\top \nabx \bm{\Psi}_{\eta}^{-1}\circ\bm{\Psi}_{\eta}\nabx\overline{\widehat{f}}  \cdot (\nabx\overline{\varphi} 
  +
  J_\eta (\nabx J_\eta^{-1})\overline{\varphi} 
  )\dq \dx \dt
\\&
+ \int_I\int_{ \Omega \times B}
   M (\nabx \bm{\Psi}_{\eta}^{-1}\circ\bm{\Psi}_{\eta})^\top(\nabx\overline{\bu})  \bq\overline{\widehat{f}}  \cdot \nabq\overline{\varphi}  \dq \dx \dt
 \\&
- \int_I\int_{ \Omega \times B}
   M \nabq \overline{\widehat{f}}  \cdot \nabq\overline{\varphi} \dq \dx \dt.
\end{align*}
Set $\overline{\widehat{f}^{12}}=\overline{\widehat{f}^1}-\overline{\widehat{f}^2}$, $\eta^{12}=\eta^1-\eta^2$ and $\overline{\bu}^{12}=\overline{\bu}^1-\overline{\bu}^2$. Then we obtain
\begin{align*}
\frac{1}{2}
\int_I  \frac{\mathrm{d}}{\dt}&
\int_{\Omega \times B}  M \overline{\widehat{f}^{12}}\, \overline{\varphi}  \dq \dx \dt 
+
\int_I\int_{ \Omega \times B}
   M \nabq \overline{\widehat{f}^{12}}  \cdot \nabq\overline{\varphi} \dq \dx \dt
  \\&
+ \int_I\int_{ \Omega \times B}
  M (\nabx \bm{\Psi}_{\eta^2}^{-1}\circ\bm{\Psi}_{\eta^2})^\top \nabx \bm{\Psi}_{\eta^2}^{-1}\circ\bm{\Psi}_{\eta^2}\nabx\overline{\widehat{f}^{12}}  \cdot \nabx\overline{\varphi} 
  \dq \dx \dt
\\&=
\int_I\int_{\Omega \times B} 
h^1_2(\overline{\varphi})  \dq \dx \dt
+
\int_I\int_{\Omega \times B}
J_{\eta^2} M \overline{\widehat{f}^{12}} \partial_t (J_{\eta^1}^{-1}) \overline{\varphi} 
 \dq \dx \dt
\\&+
\int_I\int_{\Omega \times B}M \overline{\widehat{f}^{12}} \big[ 
\nabx  \overline{\varphi} 
+
J_{\eta^2}   
(\nabx J_{\eta^1}^{-1})\overline{\varphi} \big]\cdot\partial_t\bm{\Psi}_{\eta^1}^{-1}\circ\bm{\Psi}_{\eta^1}
 \dq \dx \dt
\\&+
\int_I\int_{\Omega \times B} 
M\big[\overline{\bu}^{12}\overline{\widehat{f}^1}  
+
\overline{\bu}^2\overline{\widehat{f}^{12}}  \big]\cdot\nabx\bm{\Psi}_{\eta^1}^{-1}\circ\bm{\Psi}_{\eta^1} \nabx\overline{\varphi} 
  \dq \dx \dt
\\&+
\int_I\int_{\Omega \times B}  
J_{\eta^2} 
M\big[\overline{\bu}^{12}\overline{\widehat{f}^1}  
+
\overline{\bu}^2\overline{\widehat{f}^{12}} \big] \cdot\nabx\bm{\Psi}_{\eta^1}^{-1}\circ\bm{\Psi}_{\eta^1} (\nabx J_{\eta^1}^{-1})\overline{\varphi} 
  \dq \dx \dt
  \\&
- \int_I\int_{ \Omega \times B}
 M (\nabx \bm{\Psi}_{\eta^2}^{-1}\circ\bm{\Psi}_{\eta^2})^\top \nabx \bm{\Psi}_{\eta^2}^{-1}\circ\bm{\Psi}_{\eta^2}\nabx\overline{\widehat{f}^{12}} \cdot 
  J_{\eta^1} (\nabx J_{\eta^1}^{-1})\overline{\varphi} 
  \dq \dx \dt
   \\&
+ \int_I\int_{ \Omega \times B}
  M (\nabx \bm{\Psi}_{\eta^2}^{-1}\circ\bm{\Psi}_{\eta^2})^\top \big[(\nabx\overline{\bu}^{12})  \bq\overline{\widehat{f}^1}  
  +
  (\nabx\overline{\bu}^2)  \bq\overline{\widehat{f}^{12}}   \big] \cdot \nabq\overline{\varphi}
   \dq \dx \dt,
\end{align*}
where
\begin{align*}
h^1_2(\overline{\varphi})&:=(J_{\eta^1}-J_{\eta^2}) M \overline{\widehat{f}^{1}} \partial_t (J_{\eta^1}^{-1}) \overline{\varphi} 
+
J_{\eta^2} M \overline{\widehat{f}^{2}} \partial_t (J_{\eta^1}^{-1}-J_{\eta^2}^{-1}) \overline{\varphi} 
\\&+
M \overline{\widehat{f}^2}  
\nabx  \overline{\varphi} \cdot\big[\partial_t(\bm{\Psi}_{\eta^1}^{-1}-\bm{\Psi}_{\eta^2}^{-1})\circ\bm{\Psi}_{\eta^1}
+
 \cdot\partial_t\bm{\Psi}_{\eta^2}^{-1}\circ(\bm{\Psi}_{\eta^1}
-\bm{\Psi}_{\eta^2})
\big]
\\&+
\big[(J_{\eta^1}-J_{\eta^2}) M \overline{\widehat{f}^1} 
(\nabx J_{\eta^1}^{-1})
+
J_{\eta^2} M \overline{\widehat{f}^2} 
(\nabx (J_{\eta^1}^{-1}-J_{\eta^2}^{-1}))\big]\overline{\varphi} \cdot\partial_t\bm{\Psi}_{\eta^1}^{-1}\circ\bm{\Psi}_{\eta^1}
\\&+
J_{\eta^2} M \overline{\widehat{f}^2} 
(\nabx J_{\eta^2}^{-1})\overline{\varphi} \cdot\big[\partial_t(\bm{\Psi}_{\eta^1}^{-1}-\bm{\Psi}_{\eta^2}^{-1})\circ\bm{\Psi}_{\eta^1}
+
\partial_t\bm{\Psi}_{\eta^2}^{-1}\circ(\bm{\Psi}_{\eta^1}-\bm{\Psi}_{\eta^2})\big]
\\&+
M\overline{\bu}^2\overline{\widehat{f}^2}  \cdot\big[\nabx(\bm{\Psi}_{\eta^1}^{-1}-\bm{\Psi}_{\eta^2}^{-1})\circ\bm{\Psi}_{\eta^1} 
+
\nabx\bm{\Psi}_{\eta^2}^{-1}\circ(\bm{\Psi}_{\eta^1}-\bm{\Psi}_{\eta^2}) \big]
\nabx\overline{\varphi}
\\&
+
M\big[(J_{\eta^1}-J_{\eta^2}) 
\overline{\bu}^1\overline{\widehat{f}}^1  \cdot\nabx\bm{\Psi}_{\eta^1}^{-1} 
+
J_{\eta^2} 
\overline{\bu}^2\overline{\widehat{f}^2}  \cdot\nabx(\bm{\Psi}_{\eta^1}^{-1}-\bm{\Psi}_{\eta^2}^{-1})
\big]
\circ\bm{\Psi}_{\eta^1} (\nabx J_{\eta^1}^{-1})\overline{\varphi} 
\\&
+
J_{\eta^2} 
M\overline{\bu}^2\overline{\widehat{f}^2}  \cdot
\big[\nabx\bm{\Psi}_{\eta^2}^{-1}\circ(\bm{\Psi}_{\eta^1}-\bm{\Psi}_{\eta^2}) (\nabx J_{\eta^1}^{-1})  
+
\nabx\bm{\Psi}_{\eta^2}^{-1}\circ\bm{\Psi}_{\eta^2} (\nabx (J_{\eta^1}^{-1}-J_{\eta^2}^{-1}))
\big]\overline{\varphi} 
  \\&
 + M (\nabx \bm{\Psi}_{\eta^2}^{-1}\circ\bm{\Psi}_{\eta^2})^\top\big[ \nabx (\bm{\Psi}_{\eta^1}^{-1}-\bm{\Psi}_{\eta^2}^{-1})\circ\bm{\Psi}_{\eta^1}
  +
   \nabx \bm{\Psi}_{\eta^2}^{-1}\circ(\bm{\Psi}_{\eta^1}^{-1}-\bm{\Psi}_{\eta^2})
   \big]\nabx\overline{\widehat{f}^1}  \cdot \nabx\overline{\varphi} 
    \\&
 + M \big[ (\nabx (\bm{\Psi}_{\eta^1}^{-1}-\bm{\Psi}_{\eta^2}^{-1})\circ\bm{\Psi}_{\eta^1})^\top
  +
   (\nabx \bm{\Psi}_{\eta^2}^{-1}\circ(\bm{\Psi}_{\eta^1}^{-1}-\bm{\Psi}_{\eta^2}))^\top
   \big]\nabx \bm{\Psi}_{\eta^1}^{-1}\circ\bm{\Psi}_{\eta1}\nabx\overline{\widehat{f}^1}  \cdot \nabx\overline{\varphi} 
  \\&
 + M  (\nabx \bm{\Psi}_{\eta^2}^{-1}\circ\bm{\Psi}_{\eta^2})^\top
  \nabx \bm{\Psi}_{\eta^2}^{-1}\circ\bm{\Psi}_{\eta^2}\nabx\overline{\widehat{f}^2}  \cdot \big[ (J_{\eta^1}-J_{\eta^2})\nabx J_{\eta^1}^{-1}+
  J_{\eta^2} \nabx (J_{\eta^1}^{-1}-J_{\eta^2}^{-1})\big]\overline{\varphi}  
    \\&
 + M (\nabx \bm{\Psi}_{\eta^2}^{-1}\circ\bm{\Psi}_{\eta^2})^\top\big[ \nabx (\bm{\Psi}_{\eta^1}^{-1}-\bm{\Psi}_{\eta^2}^{-1})\circ\bm{\Psi}_{\eta^1}
  +
   \nabx \bm{\Psi}_{\eta^2}^{-1}\circ(\bm{\Psi}_{\eta^1}^{-1}-\bm{\Psi}_{\eta^2})
   \big]\nabx\overline{\widehat{f}^1}  \cdot  J_{\eta^1} (\nabx J_{\eta^1}^{-1})\overline{\varphi} 
    \\&
 + M \big[ (\nabx (\bm{\Psi}_{\eta^1}^{-1}-\bm{\Psi}_{\eta^2}^{-1})\circ\bm{\Psi}_{\eta^1})^\top
  +
   (\nabx \bm{\Psi}_{\eta^2}^{-1}\circ(\bm{\Psi}_{\eta^1}^{-1}-\bm{\Psi}_{\eta^2}))^\top
   \big]\nabx \bm{\Psi}_{\eta^1}^{-1}\circ\bm{\Psi}_{\eta^1}\nabx\overline{\widehat{f}^1}  \cdot  J_{\eta^1} (\nabx J_{\eta^1}^{-1})\overline{\varphi}  
\\&+
M \big[(\nabx (\bm{\Psi}_{\eta^1}^{-1}-\bm{\Psi}_{\eta^2}^{-1})\circ\bm{\Psi}_{\eta^1})^\top 
  +
    (\nabx \bm{\Psi}_{\eta^2}^{-1}\circ(\bm{\Psi}_{\eta^1}-\bm{\Psi}_{\eta^2}))^\top\big]
    (\nabx\overline{\bu}^1)  \bq\overline{\widehat{f}^1}  \cdot \nabq\overline{\varphi}.
\end{align*}
Take $\overline{\varphi} =\overline{\widehat{f}^{12}}$ as test function. Note that the estimate from Theorem 
\ref{thm:mainFP} does not yields boundedness of $\widehat f^i$, $i=1,2$ in space-time. However, the maximum principle from Remark \ref{rem:max} does the job and will be used repeatedly in the following.
To estimate $h^1_2$, we need the following  estimates for the critical terms
\begin{equation}
\begin{aligned}
\label{preH12}
\int_I&\int_{ \Omega \times B}
M \overline{\widehat{f}^2}  
\nabx\overline{\widehat{f}^{12}}\cdot\partial_t(\bm{\Psi}_{\eta^1}^{-1}-\bm{\Psi}_{\eta^2}^{-1})\circ\bm{\Psi}_{\eta^1}
 \dq \dx \dt
 \\&+
 \int_I\int_{ \Omega \times B}
M \overline{\bu}^2\overline{\widehat{f}^2}  
\cdot \nabx(\bm{\Psi}_{\eta^1}^{-1}-\bm{\Psi}_{\eta^2}^{-1})\circ\bm{\Psi}_{\eta^1}\nabx\overline{\widehat{f}^{12}}
 \dq \dx \dt
\\&
\leq
 \delta\int_I  
 \Vert \nabx\overline{\widehat{f}^{12}}
\Vert_{L^{2}(\Omega;L^2_M(B))}^2
 \dt
+  c(\delta,\eta^1, \overline{\widehat{f}^{2}} )
\int_I\Vert \partial_t(\bm{\Psi}_{\eta^1}^{-1}-\bm{\Psi}_{\eta^2}^{-1})
\Vert_{L^{2}(\Omega)}^2
  \dt
  \\&
+
\delta
\int_I 
 \Vert \nabx\overline{\widehat{f}^{12}}
\Vert_{L^{2}(\Omega;L^2_M(B))}^2
 \dt
+   c(\delta,\eta^1, \overline{\widehat{f}^{2}},\overline{\bu}^2)
\sup_I
\Vert \nabx(\bm{\Psi}_{\eta^1}^{-1}-\bm{\Psi}_{\eta^2}^{-1})
\Vert_{L^{2}(\Omega)}^2
.
\end{aligned}
\end{equation}
Subsequently, by using \eqref{211and213}, we obtain
\begin{equation}
\begin{aligned}
\int_I&\int_{ \Omega \times B}
 h^1_2(\overline{\widehat{f}^{12}})
   \dq \dx \dt
   \leq
    \delta
   \int_I 
 \Vert \nabx\overline{\widehat{f}^{12}}
\Vert_{L^2(\Omega;L^2_M(B))}^2
 \dt
 +
 \delta
   \int_I 
 \Vert \overline{\widehat{f}^{12}}
\Vert_{L^2(\Omega;H^1_M(B))}^2
 \dt
   \\&
 +
   c(\delta, \eta^1, \eta^2,\overline{\widehat{f}^{1}},\overline{\widehat{f}^{2}},\overline{\bu}^1,\overline{\bu}^2)
 \bigg( \int_I  \Vert \overline{\widehat{f}^{12}}
\Vert_{L^2(\Omega;L^2_M(B))}^2 \dt
+
\int_I\Vert \partial_t\eta^{12}
\Vert_{W^{1,2} (\omega)}^2
+
\sup_I
\Vert \eta^{12}
\Vert_{ W^{2,2}(\omega)}^2
 \bigg)
\end{aligned}
\end{equation}
%Next, by Reynolds transport theorem and the interface condition, we obtain
%\begin{equation}
%\begin{aligned}
%\int_I\int_{\Omega \times B} M \overline{\widehat{f}^{12}} \partial_t  \overline{\widehat{f}^{12}} 
% \dq \dx \dt
% &=
% \frac{1}{2}
% \int_I\frac{\dd}{\dt}
% \Vert \overline{\widehat{f}^{12}}
%\Vert_{L^2(\Omega;L^2_M(B))}^2
% \dt
% -
% \frac{1}{2}
% \int_I\int_{\omega}
% (\partial_t\eta^{12})\bn\circ\bm{\varphi}^{-1}\cdot\bn
% \Vert \overline{\widehat{f}^{12}}
%\Vert_{L^2_M(B)}^2
%\dy
% \dt
%\end{aligned}
%\end{equation}
%where 
%\begin{equation}
%\begin{aligned}
% \frac{1}{2}
% \int_I\int_{\omega}&
% (\partial_t\eta^{12})\bn\circ\bm{\varphi}^{-1}\cdot\bn
% \Vert \overline{\widehat{f}^{12}}
%\Vert_{L^2_M(B)}^2
%\dy
% \dt
% \lesssim
%  \int_I
% \Vert\partial_t\eta^{12}\Vert_{L^\infty(\omega)}
% \Vert \overline{\widehat{f}^{12}}
%\Vert_{L^2(\omega;L^2_M(B))}^2
% \dt
% \\&
% \lesssim
%  \int_I
%    \Vert\partial_t\eta^{12}\Vert_{W^{1,2}(\omega)}
% \Vert \overline{\widehat{f}^{12}}
%\Vert_{W^{1,2}(\Omega;L^2_M(B))}^2
% \dt
% \\&
% \lesssim
%  \int_I
%  \Vert\partial_t \eta^{12}\Vert_{W^{1,2}(\omega)}
% \Vert \overline{\widehat{f}^{12}}
%\Vert_{L^2(\Omega;L^2_M(B))} 
%\Vert (\overline{\widehat{f}^1},\overline{\widehat{f}^2})
%\Vert_{W^{2,2}(\Omega;L^2_M(B))} 
% \dt 
% \\&
% \leq c(\eta^1, \eta^2, \overline{\widehat{f}^1},\overline{\widehat{f}^2})
%  \int_I
%  \Big(\Vert\partial_t \eta^{12}\Vert_{W^{1,2}(\omega)}^2+
% \Vert \overline{\widehat{f}^{12}}
%\Vert_{L^2(\Omega;L^2_M(B))}^2
%\Big)
% \dt
%\end{aligned}
%\end{equation}
for any $\delta>0$. Next,
\begin{align*}
\int_I\int_{\Omega \times B}M \overline{\widehat{f}^{12}} 
\nabx  \overline{\widehat{f}^{12}} \cdot\partial_t\bm{\Psi}_{\eta^1}^{-1}\circ\bm{\Psi}_{\eta^1}
 \dq \dx \dt\lesssim \delta    \int_I 
 \Vert \overline{\widehat{f}^{12}}
\Vert_{L^2(\Omega;H^1_M(B))}^2
 \dt+   c(\delta, \eta^1)
 \int_I 
 \Vert \overline{\widehat{f}^{12}}
\Vert_{L^2(\Omega;L^2_M(B))}^2\dt
\end{align*}
as well as
\begin{equation}
\begin{aligned}
\int_I&\int_{\Omega \times B}\big[
J_{\eta^2} M \overline{\widehat{f}^{12}} \partial_t (J_{\eta^1}^{-1}) \overline{\widehat{f}^{12}}  
+
M \overline{\widehat{f}^{12}} 
J_{\eta^2}   
(\nabx J_{\eta^1}^{-1})\overline{\widehat{f}^{12}}  \cdot\partial_t\bm{\Psi}_{\eta^1}^{-1}\circ\bm{\Psi}_{\eta^1}\big]
 \dq \dx \dt
   \\&
   \leq
   c(\eta^1, \eta^2)\int_I 
 \Vert \overline{\widehat{f}^{12}}
\Vert_{L^2(\Omega;L^2_M(B))}^2
 \dt.
\end{aligned}
\end{equation}
Similar to \eqref{preH12}, we have that
\begin{equation}
\begin{aligned}
\int_I&\int_{\Omega \times B} 
M \overline{\bu}^{12}\overline{\widehat{f}^1}  
\cdot\nabx\bm{\Psi}_{\eta^1}^{-1}\circ\bm{\Psi}_{\eta^1} \nabx\overline{\widehat{f}^{12}} 
  \dq \dx \dt
   \\&
   \leq
   c(\delta, \eta^1,\overline{\widehat{f}^{1}})
   \int_I
 \Vert \overline{\bu}^{12}
\Vert_{W^{1,2}(\Omega )}^2
 \dt
 +
 \delta
   \int_I 
 \Vert \overline{\widehat{f}^{12}}
\Vert_{W^{1,2}(\Omega;L^2_M(B))}^2
 \dt.
\end{aligned}
\end{equation}
We also  have that
\begin{equation}
\begin{aligned}
\int_I&\int_{\Omega \times B}  
J_{\eta^2} 
M\big[\overline{\bu}^{12}\overline{\widehat{f}^1}  
+
\overline{\bu}^2\overline{\widehat{f}^{12}} \big] \cdot\nabx\bm{\Psi}_{\eta^1}^{-1}\circ\bm{\Psi}_{\eta^1} (\nabx J_{\eta^1}^{-1})\overline{\widehat{f}^{12}}
  \dq \dx \dt
   \\&
   \leq
   c(  \eta^1,\eta^2,\overline{\widehat{f}^{1}})
      \int_I
 \Vert \overline{\bu}^{12}
\Vert_{W^{1,2}(\Omega )}^2
 \dt
 +
 c( \eta^1,\eta^2,\overline{\widehat{f}^{1}},\overline{\bu}^2)
   \int_I 
 \Vert \overline{\widehat{f}^{12}}
\Vert_{L^2(\Omega;L^2_M(B))}^2
 \dt
\end{aligned}
\end{equation}
as well as
\begin{equation}
\begin{aligned}
\int_I&\int_{ \Omega \times B}
  M (\nabx \bm{\Psi}_{\eta^2}^{-1}\circ\bm{\Psi}_{\eta^2})^\top \big[(\nabx\overline{\bu}^{12})  \bq\overline{\widehat{f}^1}  
  +
  (\nabx\overline{\bu}^2)  \bq\overline{\widehat{f}^{12}}   \big] \cdot \nabq\overline{\widehat{f}^{12}}
   \dq \dx \dt
   \\&
   \leq
   c(\delta, \eta^2,\overline{\widehat{f}^{1}})
      \int_I
 \Vert \overline{\bu}^{12}
\Vert_{W^{1,2}(\Omega )}^2
 \dt
 +
 c(\delta, \eta^2)
   \int_I  \Vert\nabx\overline\bu^{2}\Vert_{L^\infty(\Omega)}
 \Vert \overline{\widehat{f}^{12}}
\Vert_{L^2(\Omega;L^2_M(B))}^2
 \dt
 \\&+
   \delta
   \int_I 
 \Vert \overline{\widehat{f}^{12}}
\Vert_{L^2(\Omega;H^1_M(B))}^2
 \dt.
\end{aligned}
\end{equation}
Next, we   obtain
\begin{equation}
\begin{aligned}
\int_I&\int_{\Omega \times B}M \overline{\widehat{f}^{12}}  
\big( 
\nabx  \overline{\widehat{f}^{12}}\cdot
 \partial_t\bm{\Psi}_{\eta^1}^{-1}\circ\bm{\Psi}_{\eta^1}
 +
  \overline{\bu}^2 \cdot\nabx\bm{\Psi}_{\eta^1}^{-1}\circ\bm{\Psi}_{\eta^1} \nabx\overline{\widehat{f}^{12}} 
\big)
  \dq \dx \dt
  \\&
  +
  \int_I\int_{ \Omega \times B}
 M (\nabx \bm{\Psi}_{\eta^2}^{-1}\circ\bm{\Psi}_{\eta^2})^\top \nabx \bm{\Psi}_{\eta^2}^{-1}\circ\bm{\Psi}_{\eta^2}\nabx\overline{\widehat{f}^{12}} \cdot 
  J_{\eta^1} (\nabx J_{\eta^1}^{-1})\overline{\widehat{f}^{12}}
  \dq \dx \dt
   \\&
   \leq 
   \delta
   \int_I 
 \Vert \nabx\overline{\widehat{f}^{12}}
\Vert_{L^2(\Omega;L^2_M(B))}^2
 \dt
 +
 c(\delta,\eta^1,\eta^2,\overline{\widehat{f}^{1}},\overline{\widehat{f}^{2}},\overline{\bu}^2)
   \int_I 
 \Vert \overline{\widehat{f}^{12}}
\Vert_{L^2(\Omega;L^2_M(B))}^2
 \dt.
\end{aligned}
\end{equation}
By combining the above with the ellipticity of $(\nabx \bm{\Psi}_{\eta^2}^{-1}\circ\bm{\Psi}_{\eta^2})^\top \nabx \bm{\Psi}_{\eta^2}^{-1}\circ\bm{\Psi}_{\eta^2}$ and applying Gr\"onwall's lemma, we have shown that
\begin{equation}
\begin{aligned}
\label{weakFPdifference1}
\Vert \overline{\widehat{f}^{12}}
\Vert_{\overline{Y}}^2
&=
\sup_{I}
\Vert \overline{\widehat{f}^{12}}
\Vert_{L^2(\Omega;L^2_M(B))}^2
+
\int_I
\Vert \overline{\widehat{f}^{12}}
\Vert_{W^{1,2}(\Omega;L^2_M(B))}^2
\dt
+
\int_I
\Vert \overline{\widehat{f}^{12}}
\Vert_{L^2(\Omega;H^1_M(B))}^2
\dt
\\&
   \leq e^{cT}
   \bigg(
   \int_I 
 \Vert \overline{\bu}^{12}
\Vert_{W^{1,2}(\Omega )}^2
 \dt
 +
\int_I\Vert \partial_t\eta^{12}
\Vert_{W^{1,2} (\omega)}^2\dt
+
\sup_I\Vert \eta^{12}
\Vert_{W^{2,2} (\omega)}^2
\bigg)
\end{aligned}
\end{equation}
with a constant $c=c( \eta^1, \eta^2,\overline{\widehat{f}^{1}},\overline{\widehat{f}^{2}},\overline{\bu}^1,\overline{\bu}^2)$. 
\\ \\ 
Now let us consider the two solutions $(\eta^1,\bu^1,\pi^1)$ and $(\eta^2,\bu^2,\pi^2)$  of \eqref{contEqAlone}--\eqref{momEqAlone} with datasets $$(\bff, g, \eta_0, \eta_\star, \bu_0, \mathbb{S}_\bq(\hbar^1))\quad\text{and}\quad(\bff, g, \eta_0, \eta_\star, \bu_0, \mathbb{S}_\bq(\hbar^2)),$$ respectively,
and let 
$\overline{\bu}^i=\bu^i\circ \bm{\Psi}_{\eta^i}$ and $\overline{\hbar}^i=\hbar^i\circ \bm{\Psi}_{\eta^i}$, for $i=1,2$, be the transformations onto the fixed reference domain.  Since $\nabx \overline{\bu}^i = \nabx\bu^i\circ \bm{\Psi}_{\eta^i}\nabx \bm{\Psi}_{\eta^i}$, $i=1,2$, by setting $\eta^{12}:=\eta^1-\eta^2$, $\overline{\bu}^{12}:=\overline{\bu}^1-\overline{\bu}^2$, $\overline{\pi}^{12}:=\overline{\pi}^1-\overline{\pi}^2$ and $\overline{\hbar}^{12}:=\overline{\hbar}^1-\overline{\hbar}^2$, we obtain
\begin{equation}
\begin{aligned}\label{eq:523}
\int_I\Vert
\nabx \overline{\bu}^{12}
\Vert_{L^2(\Omega)}^2
\dt
&\lesssim
\int_I\Vert
\nabx\bu^1\nabx \bm{\Psi}_{\eta^1}\circ \bm{\Psi}_{\eta^1}^{-1}-(\nabx\bu^2\circ \bm{\Psi}_{\eta^2})\circ \bm{\Psi}_{\eta^1}^{-1}\nabx \bm{\Psi}_{\eta^2}\circ \bm{\Psi}_{\eta^1}^{-1}
\Vert_{L^2(\Omega_{\eta^1})}^2\dt
\\&
\lesssim
\int_I\Vert
\nabx\bu^1(\nabx \bm{\Psi}_{\eta^1}\circ \bm{\Psi}_{\eta^1}^{-1}-\nabx \bm{\Psi}_{\eta^2}\circ \bm{\Psi}_{\eta^1}^{-1})
\Vert_{L^2(\Omega_{\eta^1})}^2\dt
\\&+
\int_I\Vert
(\nabx\bu^1-(\nabx\bu^2\circ \bm{\Psi}_{\eta^2})\circ \bm{\Psi}_{\eta^1}^{-1})\nabx \bm{\Psi}_{\eta^2}\circ \bm{\Psi}_{\eta^1}^{-1}
\Vert_{L^2(\Omega_{\eta^1})}^2\dt
\\&
=:K_1+K_2,
\end{aligned}
\end{equation}
where
\begin{equation}
\begin{aligned}
K_1&
\lesssim
\int_I\Vert \nabx\bu^1\Vert_{L^4(\Omega_{\eta^1})}^2\Vert
J_{\eta^1}
\nabx (\bm{\Psi}_{\eta^1}
-
\bm{\Psi}_{\eta^2})
\Vert_{L^4(\Omega)}^2\dt\\
&
\lesssim
\int_I\Vert \bu^1\Vert_{W^{2,2}(\Omega_{\eta^1})}^2\Vert
\bm{\Psi}_{\eta^1}
-
\bm{\Psi}_{\eta^2}
\Vert_{W^{2,2}(\Omega)}^2\dt\\
&
\lesssim
T\sup_I\Vert \bu^1\Vert_{W^{2,2}(\Omega_{\eta^1})}^2\sup_I\Vert
\eta^1 - \eta^2
\Vert_{W^{2,2}(\omega)}^2\\&\lesssim
T\sup_I\Vert
\eta^1 - \eta^2
\Vert_{W^{2,2}(\omega)}^2
%\lesssim
%\sup_I\Vert
%\eta^1 - \eta^2
%\Vert_{W^{1,2}(\omega)}^2\\
%&\lesssim T\int_I\|\partial_t(\eta^1-\eta^2)\|_{W^{1,2}(\omega)}^2\dt
\end{aligned}
\end{equation}
and
\begin{equation}
\begin{aligned}
K_2
\lesssim
\int_I\Vert
\nabx\bu^1- \nabx\overline{\bu}^2 \circ \bm{\Psi}_{\eta^1}^{-1} 
\Vert_{L^2(\Omega_{\eta^1})}^2\dt.
\end{aligned}
\end{equation}
Also,
\begin{equation}
\begin{aligned}\label{eq:2303}
&\int_I\Vert
 \overline{\bu}^{12}
\Vert_{L^2(\Omega)}^2
\dt
=
\int_I\Vert
\bu^1- \overline{\bu}^2\circ \bm{\Psi}_{\eta^1}^{-1} 
\Vert_{L^2(\Omega_{\eta^1})}^2
\dt.
\end{aligned}
\end{equation}
It, therefore, follows  from \eqref{eq:523}--\eqref{eq:2303} that for $T\ll1$,
\begin{equation}
\begin{aligned}
\label{weakFluidDifference1}
\int_I &
 \Vert \overline{\bu}^{12}
\Vert_{W^{1,2}(\Omega )}^2
 \dt
 +
\int_I\Vert \partial_t\eta^{12}
\Vert_{W^{1,2} (\omega)}^2\dt
+
\sup_I\Vert \eta^{12}
\Vert_{W^{2,2} (\omega)}^2
\\
&\lesssim 
\int_{I }\Vert  \bu^1-
\overline{\bu}^2\circ \bm{\Psi}_{\eta^1}^{-1}\Vert_{W^{1,2}(\Omega_{\eta^1(t)})}^2 \dt
 +
\int_I\Vert \partial_t\eta^{12}
\Vert_{W^{1,2} (\omega)}^2\dt
+
\sup_I\Vert \eta^{12}
\Vert_{W^{2,2} (\omega)}^2
\end{aligned}
\end{equation}
However, by \cite[Remark 5.2]{BMSS},
\begin{equation}
\begin{aligned}
\label{weakFluidDifference0}
\int_{I }\Vert  \bu^1-
\overline{\bu}^2\circ \bm{\Psi}_{\eta^1}^{-1}\Vert_{W^{1,2}(\Omega_{\eta^1(t)})}^2 \dt
 &+
\int_I\Vert \partial_t\eta^{12}
\Vert_{W^{1,2} (\omega)}^2\dt
+
\sup_I\Vert \eta^{12}
\Vert_{W^{2,2} (\omega)}^2
\\&\lesssim 
\int_{I }
 \Vert\mathbb{S}_\bq(\hbar^1-\overline{\hbar}^2\circ \bm{\Psi}_{\eta^1}^{-1}) 
\Vert_{L^2(\Omega_{\eta^1(t)})}^2\dt
\end{aligned}
\end{equation}
where
\begin{align*}
\int_{I }
 \Vert\mathbb{S}_\bq(\hbar^1-\overline{\hbar}^2\circ \bm{\Psi}_{\eta^1}^{-1}) 
\Vert_{L^2(\Omega_{\eta^1(t)})}^2\dt
=
\int_{I }
 \Vert\mathbb{S}_\bq(\overline{\hbar}^{12}) \Vert_{L^2(\Omega)}^2
\dt.
\end{align*}
If we now combine \eqref{weakFPdifference1} with \eqref{weakFluidDifference1}-\eqref{weakFluidDifference0} and the fact that
\begin{equation}
\begin{aligned}
\int_{I }
 \Vert\mathbb{S}_\bq(\overline{\hbar}^{12}) \Vert_{L^2(\Omega)}^2
\dt
&\leq
c(\delta)T 
\sup_I \Vert \overline{\hbar}^{12}
\Vert_{L^2(\Omega;L^2_M(B))}^2
+
\delta
 \int_I 
 \Vert \overline{\hbar}^{12}
\Vert_{L^2(\Omega;H^1_M(B))}^2
 \dt,
\end{aligned}
\end{equation}
we obtain
\begin{equation}
\begin{aligned}
\Vert \overline{\widehat{f}^{12}}
\Vert_{\overline{Y}}^2
&\leq c\,e^{cT}(c(\delta)T +\delta)\Vert \overline{\hbar}^{12}
\Vert_{\overline{Y}}^2 
\leq
\tfrac{1}{2}\Vert \overline{\hbar}^{12}
\Vert_{\overline{Y}}^2 
\end{aligned}
\end{equation}
choosing first $\delta$ and then $T$ accordingly. The existence of the desired fixed point now follows.

\section{The 2D co-rotational model}
\subsection{Solving the equation for the solute}
\label{sec:FP}
\noindent
In this section, for a known moving domain $\Ozeta$ and a known solenoidal velocity field $\mathbf{w}$ with skwe-symmetric gradient $\mathcal W(\mathbf w)=\frac{1}{2}(\nabla\mathbf w-\nabla\mathbf w^\top)$, we aim to construct a strong solution of the Fokker--Planck equation
\begin{align}
\label{eq:FPy}
M\big(\partial_t \widehat{f} + (\mathbf{w}\cdot \nabx) \widehat{f})
+
 \divq  \big(  \mathcal W(\mathbf{w}) \bq M\widehat{f} \big) 
=
\Delx(M\widehat{f})
+
 \divq  \big( M \nabq  \widehat{f}
\big)
\end{align} 
in $I\times\Omega_\zeta\times B$,
where the Maxwellian $M$ is given by
\begin{align*}
M(\mathbf{q}) = \frac{e^{-U \left(\frac{1}{2}\vert \mathbf{q} \vert^2 \right) }}{\int_Be^{-U \left(\frac{1}{2}\vert \mathbf{q} \vert^2 \right) }\,\mathrm{d}\mathbf{q}},\quad  U (s) = -\frac{b}{2} \log \bigg(1-  \frac{2s}{b} \bigg), \quad s\in [0,b/2)
\end{align*}
with $b>2$.
 Equation \eqref{eq:FPy} is complemented
with the conditions
\begin{align}
&\widehat{f}(0, \cdot, \cdot) =\widehat{f}_0 \geq 0
& \quad \text{in }\Omega_{\zeta_0} \times B,
\label{fokkerPlankIintialxsec4Y}
\\
&
\nabx\widehat{f}\cdot \bn_\zeta =0
&\quad \text{on }I \times \partial\Omega_\zeta \times B,
\label{fokkerPlankBoundarySpacexsc4Y}
\\
&M\big(\nabq\widehat{f}   -  \mathcal W(\mathbf w) \bq \widehat{f}
 \big) \cdot \frac{\bq}{\vert\bq\vert} =0
&\quad \text{on }I \times \Omega_\zeta \times \partial \overline{B}.
\label{fokkerPlankBoundaryxsc4Y}
\end{align}
Let us start with a precise definition of what we mean by a strong solution.  
\begin{definition}
%[Strong solution] 
\label{def:strsolmartFP}
Assume that the triplet $(\widehat{f}_0,\zeta, \mathbf{w})$ satisfies
\begin{align}
\label{fokkerPlanckDataAloneY}
\widehat{f}_0\in  W^{1,2}\big( \Omega_{\zeta(0)}; L^2_M(B) \big)   ,
\qquad
\mathbf{w}\in L^2(I; W^{2,2}_{\divx}(\Ozeta)),
%\cap  W^{1,\infty}(I;W^{1,2} ( \Ozeta))\cap  W^{1,2}(I;W^{2,2} ( \Ozeta)),
\\
\zeta\in W^{1,\infty}(I;W^{1,2}(\omega))
\cap 
L^{\infty}(I;W^{3,2}(\omega))
,
\label{fokkerPlanckDataAloneAY}
\\\mathbf w  \circ \bm{\varphi}_{\zeta} =(\partial_t\zeta)\bn
\quad \text{on }I \times \omega,  \quad\|\zeta\|_{L^\infty(I\times\omega)}<L.
\label{fokkerPlanckDataAloneBY}
\end{align}
We call
$\widehat{f}$
a \textit{strong solution} of   \eqref{eq:FPy} with data $(\widehat{f}_0,\zeta, \mathbf{w})$ if 
\begin{itemize}
\item[(a)] $\widehat{f}$ satisfies
\begin{align*}
\widehat{f}&\in   L^{\infty}\big(I;W^{1,2}(\Oeta;L^2_M(B))  \big)
\cap 
L^{2}\big(I;W^{2,2}(\Oeta;L^2_M(B))  \big)
\\&
\qquad\qquad \qquad\cap L^{2}\big(I;W^{1,2}(\Oeta;H^1_M(B))  \big);
\end{align*}
\item[(b)] for all  $  \varphi  \in C^\infty (\overline{I}\times \R^3 \times \overline{B} )$, we have
\begin{equation}
\begin{aligned}
\label{weakFokkerPlanckEq}
\int_I  \frac{\mathrm{d}}{\dt}
\int_{\Ozeta \times B}M \widehat{f} \, \varphi \dq \dx \dt 
&=\int_I\int_{\Ozeta \times B}\big(M \widehat{f} \,\partial_t \varphi 
+
M\mathbf{w} \widehat{f} \cdot \nabx \varphi
-
\nabx \widehat{f} \cdot \nabx \varphi
\big) \dq \dx \dt
\\&
+ \int_I\int_{ \Ozeta \times B}
 \big( M \mathcal W(\mathbf{w})  \bq\widehat{f}-
 M \nabq  \widehat{f} \big) \cdot \nabq\varphi \dq \dx \dt.
\end{aligned}
\end{equation}
\end{itemize}
\end{definition}
\noindent We now formulate our result on the existence of a  unique strong solution of \eqref{eq:FPy}.
\begin{theorem}\label{thm:mainFP}
Let $(\widehat{f}_0,\zeta, \mathbf{w})$ satisfy  \eqref{fokkerPlanckDataAloneY}--\eqref{fokkerPlanckDataAloneBY}
% and  
%suppose further that $\widetilde f_0\in L^{2}(\Omega_{\zeta(0)};L^2_M(B))$, where $\widetilde f_0$ is given by 
%\begin{equation}
%\begin{aligned}
%\label{eq:FP1Approx1Initial}
%M   \widetilde{f}_0 
%&=
%\Delx(M  \widehat{f}_0)
%+
% \divq  \big( M \nabq    \widehat{f}_0
%\big)
%-
%M  (\mathbf{w}_0\cdot \nabx)   \widehat{f}_0
%-
% \divq  \big(   \mathcal W(\mathbf{w}_0) \bq M\widehat{f}_0 \big) .
%\end{aligned}
%\end{equation}
Then there is a unique strong solution $\widehat{f}$ of \eqref{eq:FPy}--\eqref{fokkerPlankBoundaryxsc4Y}, in the sense of Definition \ref{def:strsolmartFP},
 such that
\begin{equation}
\begin{aligned} 
%&
% \sup_I \Vert \partial_t\widehat{f}(t)
%\Vert_{L^{2}(\Ozeta;L^2_M(B))}^2+
% \int_I 
% \Vert \partial_t\widehat{f}
%\Vert_{W^{1,2}(\Ozeta;L^2_M(B))}^2
% \dt
% +
% \int_I 
% \Vert \partial_t\widehat{f}
%\Vert_{L^{2}(\Ozeta;H^1_M(B))}^2
% \dt
% \\
 &+
 \sup_I \Vert \widehat{f}(t)
\Vert_{W^{1,2}(\Ozeta;L^2_M(B))}^2+
 \int_I 
 \Vert \widehat{f}
\Vert_{W^{2,2}(\Ozeta;L^2_M(B))}^2
 \dt
 +
 \int_I 
 \Vert \widehat{f}
\Vert_{W^{1,2}(\Ozeta;H^1_M(B))}^2
 \dt
 \\ 
&\quad\lesssim  \exp\bigg(c\int_I \Vert  \mathbf{w} \Vert_{W^{2,2}(\Ozeta)}^2 \dt
 +
c \int_I\Vert\partial_t\zeta\Vert_{W^{1/3,2}(\omega)}^{12/5}\dt
\bigg) 
\Vert \widehat{f}_0
\Vert_{W^{1,2}(\Omega_{\zeta(0)};L^2_M(B))}^2
\label{eq:thm:mainFP}
\end{aligned}
\end{equation} 
%holds for any $\kappa\in(0,1/2)$ with a constant depending on the $L^\infty(I;W^{1,\infty}(\omega))$-norm of $\zeta$ but otherwise being independent of the data.
holds with a constant depending on $\sup_I\|\mathbf w\|_{L^2(\Ozeta)}$, $\sup_I\|\zeta \|_{W^{1,\infty}(\omega)}$ and $\|\widehat f_0\|_{L^\infty(\Omega_{\zeta(0)};L^2_M(B))}$.
\end{theorem}
\noindent We will obtain a solution of \eqref{eq:FPy} by way of a limit to the following approximation
%\todo{Is $n$ the Galerkin paramter?}
\begin{equation}
\begin{aligned}
\label{eq:FP1ApproxY}
M\big(\partial_t   \widehat{f}^n + (\mathbf{w}\cdot \nabx)   \widehat{f}^n\big)
+
 \divq  \big( \mathcal W(\mathbf{w})\bq M\widehat{f}^n \big) 
&=
\Delx(M  \widehat{f}^n)
+
 \divq  \big( M \nabq    \widehat{f}^n
\big).
\end{aligned}
\end{equation}
Here, we solve the equation under the boundary conditions
\begin{align*} 
&
\nabx\widehat{f}^n\cdot \bn_\zeta =0
&\quad \text{on }I \times \partial\Omega_\zeta \times B,
\\
&M\big(\nabq\widehat{f}^n   -   \mathcal W(\mathbf w) \bq \widehat{f}^n
 \big) \cdot \frac{\bq}{\vert\bq\vert} =0
&\quad \text{on }I \times \Omega_\zeta \times \partial \overline{B},
\end{align*}
and consider the same initial condition $\widehat{f}^n(0)=\widehat{f}_0$. 
\begin{remark}
Unlike the case where $\mathcal W(\mathbf w)$  is replaced by $\nabx\bu$ considered in Section \ref{sec:FP}, we do not require a cutoff  $\chi^n(\bq)\in C^1_c(B)$ here since $\divq  \big( \mathcal W(\mathbf{w})\bq M\widehat{f}^n \big)
=
 \mathcal W(\mathbf{w})\bq M\cdot\nabq\widehat{f}^n  \in L^2(B)$ holds for the co-rotational case.
\end{remark}
In the following couple of lemmas, we will derive  estimates for
\eqref{eq:FP1ApproxY} which are uniform with respect to $n$. They transfer directly to \eqref{eq:FPy} as the latter is linear. 
%A rigorous proof can be achieved by working with a Galerkin approximation as was done in \cite[Section 4]{breit2021incompressible}.
%We will combine the proof of Theorem into several lemmas. Let begin with uniqueness.
%\begin{lemma}\label{lem:mainFPunique}
%Consider two  strong solutions $\widehat{f}^i$, $i=1,2$ of   \eqref{eq:FP} with the same dataset $(\mathbf{w} ,\zeta,\widehat{f}_0)$ satisfying \eqref{fokkerPlanckDataAlone}. For $\widehat{f}^{12}=\widehat{f}^1-\widehat{f}^2$,
%\begin{equation}
%\begin{aligned}
%\sup_{t\in I} \Vert \widehat{f}^{12}
%\Vert_{L^2(\Ozeta;L^2_M(B))}^2
%&+
% \int_I 
% \Vert \widehat{f}^{12}
%\Vert_{W^{1,2}(\Ozeta;L^2_M(B))}^2
% \dt
% +
% \int_I 
% \Vert \widehat{f}^{12}
%\Vert_{L^2(\Ozeta;H^1_M(B))}^2
% \dt
%=0.
%\end{aligned}
%\end{equation}
%\begin{proof}
%For $\widehat{f}^{12}=\widehat{f}^1-\widehat{f}^2$, the difference $\widehat{f}^{12}$ solves
%\begin{equation}
%\begin{aligned}
%\label{eq:FP2}
%M\big(\partial_t  \widehat{f}^{12} &+ (\mathbf{w} \cdot \nabx)  \widehat{f}^{12} )
%+
% \divq  \big(  ( \nabx\mathbf{w} ) \bq M\widehat{f}^{12} \big) 
% =
% \Delx\big( M  \widehat{f}^{12}
%\big)
%+
% \divq  \big( M \nabq  \widehat{f}^{12}
%\big).
%\end{aligned}
%\end{equation} 
%Since \eqref{eq:FP2} retains the same structure as \eqref{eq:FP} but with zero initial condition (since both solution share the same data), we directly obtain \eqref{eq:FP2}  from \eqref{fokkerEnergyEst5}.
%\end{proof}
%\end{lemma}
The first of two results leading to the proof of Theorem \ref{thm:mainFP} is the following.
\begin{lemma}\label{lem:mainFP1Y}
Let $(\widehat{f}_0,\zeta, \mathbf{w})$ satisfy  \eqref{fokkerPlanckDataAloneY}--\eqref{fokkerPlanckDataAloneBY}  and let $\widehat{f}^n$ be the corresponding solution to \eqref{eq:FP1ApproxY}. Then we have
\begin{equation}
\begin{aligned}
\label{fokkerEnergyEst5Y}
\sup_I \Vert \widehat{f}^n(t)
\Vert_{L^2(\Ozeta;L^2_M(B))}^2
&+
 \int_I 
 \Vert \widehat{f}^n
\Vert_{W^{1,2}(\Ozeta;L^2_M(B))}^2
 \dt
 +
 \int_I 
 \Vert \widehat{f}^n
\Vert_{L^2(\Ozeta;H^1_M(B))}^2
 \dt
\\&\lesssim
\Vert \widehat{f}_0
\Vert_{L^2(\Omega_{\zeta(0)};L^2_M(B))}^2
%\\&\lesssim 1
\end{aligned}
\end{equation} 
uniformly in $n\in\mathbb{N}$. 
\end{lemma}
\begin{proof} 
Before we begin, we first note that since
\begin{itemize}
\item $\divq(\mathcal W( \mathbf{w}) \bq)=0$, 
\item $\nabq M=-M\nabq U=-M\frac{b\bq}{b-\vert\bq\vert^2}$,
\item $(\mathcal W( \mathbf{w}) \bq)\cdot\bq=0$,
\end{itemize}
for any $q\geq 1$, we have that
\begin{equation}
\begin{aligned}
\label{allq}
\int_B  \divq  \big(   \mathcal W(\mathbf{w})\bq M\widehat{f}^n \big)(\widehat{f}^n)^q\dq
&=
 \int_B    \mathcal W(\mathbf{w})\bq M\cdot \nabq\widehat{f}^n (\widehat{f}^n)^q\dq
\\&=
\frac{1}{q+1}
\int_B    \mathcal W(\mathbf{w})\bq M\cdot \nabq(\widehat{f}^n)^{q+1}\dq
\\&
=
\frac{1}{q+1}
\int_B  \divq  \big(  \mathcal W(\mathbf{w})\bq M(\widehat{f}^n)^{q+1} \big)\dq.
\end{aligned} 
\end{equation} 
We only require $q=1$ at this point. The case $q>1$ will be used in Remark \ref{rem:max} below.

Now, if we  test \eqref{eq:FP1ApproxY} with $  \widehat{f}^n $ and integrate the resulting equation over the ball $B$, we obtain by using the boundary conditions that
\begin{equation}
\begin{aligned}
\label{fokkerEnergyY}
\frac{1}{2}\partial_t\Vert    \widehat{f}^n \Vert_{L^2_M(B)}^2
+
\frac{1}{2}(\mathbf{w}\cdot \nabx)\Vert    \widehat{f}^n \Vert_{L^2_M(B)}^2
&+
\Vert   \nabx \widehat{f}^n \Vert_{L^2_M(B)}^2
+ 
\Vert   \widehat{f}^n\Vert_{H^1_M(B)}^2
=
0.
\end{aligned}
\end{equation}
%By skew-symmetry of $\mathcal W(\mathbf w)$ and modulus dependence of $M$ and $\chi^n$ one easily checks that the remaining integral vanishes.
If we now integrate \eqref{fokkerEnergyY} over space-time and apply Reynolds transport theorem (using also \eqref{fokkerPlanckDataAloneBY}), we obtain  
\eqref{fokkerEnergyEst5Y}.
\end{proof}
\begin{remark}\label{rem:max}
The proof of Lemma \ref{lem:mainFP1Y} can be repeated for powers $q\geq 2$ of $\widehat f^n$ obtaining (ignoring the dissipative terms and using \eqref{allq}) 
\begin{equation}
\begin{aligned}
\label{fokkerEnergyEst5'}
\sup_I \Vert \widehat{f}^n(t)
\Vert_{L^q(\Ozeta;L^2_M(B))}^q\lesssim
\Vert \widehat{f}_0
\Vert_{L^q(\Omega_{\zeta(0)};L^2_M(B))}^q
%\\&\lesssim 1
\end{aligned}
\end{equation} 
uniformly in $n\in\mathbb{N}$. Checking that the $q$-dependent constant does not explode, we obtain the maximum principle
\begin{equation}
\begin{aligned}
\label{fokkerEnergyEst5''}
\sup_I \Vert \widehat{f}^n(t)
\Vert_{L^\infty(\Omega_\eta;L^2_M(B))}\lesssim
\Vert \widehat{f}_0
\Vert_{L^\infty(\Omega_{\zeta(0)};L^2_M(B))};
%\\&\lesssim 1
\end{aligned}
\end{equation}
a minimum principle can be proved similarly, but it is not needed for our purposes. 
\end{remark}
Next, we show the following lemma.
\begin{lemma}\label{lem:mainFP2Y}
Let $(\widehat{f}_0,\zeta, \mathbf{w})$ satisfy  \eqref{fokkerPlanckDataAloneY}--\eqref{fokkerPlanckDataAloneBY}  and let $\widehat{f}^n$ be the corresponding solution to \eqref{eq:FP1ApproxY}.
Then we have 
\begin{equation}
\begin{aligned}
\label{fokkerEnergyNeg1Y}
&\sup_I \Vert \nabx\widehat{f}^n(t)
\Vert_{L^2(\Ozeta;L^2_M(B))}^2+
 \int_I 
 \Vert \Delx\widehat{f}^n
\Vert_{L^2(\Ozeta;L^2_M(B))}^2
 \dt
 +
 \int_I 
 \Vert \nabx\widehat{f}^n
\Vert_{L^2(\Ozeta;H^1_M(B))}^2
 \dt
 \\&
 \qquad\quad\lesssim \exp\bigg(c\int_I \Vert  \mathbf{w} \Vert_{W^{2,2}(\Ozeta)}^2 \dt
 +
c \int_I\Vert\partial_t\zeta\Vert_{W^{1/3,2}(\omega)}^{12/5}\dt
\bigg) 
\Vert \widehat{f}_0
\Vert_{W^{1,2}(\Omega_{\zeta(0)};L^2_M(B))}^2
%+c\int_I \Vert  \mathbf{w} \Vert_{W^{2,2}(\Ozeta)}^2 \dt
%\\&\leq  c\exp\bigg(c\int_I \big(\Vert  \mathbf{w} \Vert_{W^{4,2}(\Ozeta)}^2
%+
%\Vert  \partial_t\zeta \Vert_{W^{1+\epsilon/2,2}(\omega)}^{2/(1-2\epsilon)}
%\big) \dt \bigg)
%\\&
%\times
%\bigg( 
%\Vert \widehat{f}_0
%\Vert_{W^{1,2}(\Omega_{\zeta(0)};L^2_M(B))}^2
%+
%\int_I\Vert \mathbf{w}\Vert_{W^{4,2}(\Ozeta)}^2\Vert  \widehat{f}^n\Vert_{L^2(\Ozeta;L^2_M(B))}^2\dt
%\bigg)
%\\&
\end{aligned}
\end{equation}
%for any $\kappa\in(0,1/2)$ 
uniformly in $n\in\mathbb{N}$, where the hidden constant also depends on $\sup_I\|\mathbf w\|_{L^2(\Ozeta)}$, $\sup_I\|\zeta \|_{W^{1,\infty}(\omega)}$ and $\|\widehat f_0\|_{L^\infty(\Omega_{\zeta(0)};L^2_M(B))}$.
%\todo{Note that $L^\infty_t(L^2_\by)\cap L^2_t(W^{1/2,2})\hookrightarrow L^{2+\kappa}_t(W^{1/2-\kappa,2})$. This is available even without dissipation inn the shell equation!}
\end{lemma}
\begin{proof}
Now, we test \eqref{eq:FP1ApproxY} with $\Delx  \widehat{f}^n$. First of all, note that by \eqref{fokkerPlankBoundarySpacexsc4Y}, the Reynolds transport theorem and \eqref{fokkerPlanckDataAloneBY}, 
\begin{equation}
\begin{aligned}
\label{fokkerEnergyEst6}
\int_I\int_{\Ozeta\times B}M \partial_t   \widehat{f}^n  \Delx  \widehat{f}^n\dq\dx\dt
&=
\frac{1}{2}
\int_I\int_{\partial\Ozeta }
\bn_\zeta\cdot((\partial_t\zeta)\bn)\circ\bm{\varphi}_\zeta^{-1}
\Vert \nabx  \widehat{f}^n\Vert_{L^2_M(B)}^2\dd\mathcal{H}^1\dt
\\&-
\frac{1}{2}\int_I\frac{\dd}{\dt}
\Vert \nabx  \widehat{f}^n\Vert_{L^2(\Ozeta;L^2_M(B))}^2\dt,
\end{aligned}
\end{equation}
where, by interpolation, the trace theorem and Young's inequality,
\begin{equation}
\begin{aligned}
\label{fokkerEnergyEst7Y}
\bigg\vert\int_I&\int_{\partial\Ozeta }\bn_\zeta\cdot((\partial_t\zeta)\bn)\circ\bm{\varphi}_\zeta^{-1}\Vert \nabx  \widehat{f}^n\Vert_{L^2_M(B)}^2\dd\mathcal{H}^1\dt\bigg\vert
\\&\lesssim
\int_I\Vert \bn_\zeta\cdot((\partial_t\zeta)\bn)\circ\bm{\varphi}_\zeta^{-1}\Vert_{L^6(\partial\Ozeta)}\Vert \nabx  \widehat{f}^n\Vert_{L^{12/5}(\partial\Ozeta;L^2_M(B))}^2 \dt
\\&
\lesssim
\int_I\Vert\partial_t\zeta\Vert_{W^{1/3,2}(\omega)}\Vert \nabx  \widehat{f}^n\Vert_{W^{7/12,2}(\Ozeta;L^2_M( B))}^2 \dt
\\&
\lesssim
\int_I\Vert\partial_t\zeta\Vert_{W^{1/3,2}(\omega)}\Vert \nabx  \widehat{f}^n\Vert_{L^2(\Ozeta;L^2_M(B))}^{5/6}
\Vert \nabx  \widehat{f}^n\Vert_{W^{1,2}(\Ozeta;L^2_M( B))}^{7/6} \dt
\\&
\leq
\delta
\int_I 
\Vert \nabx  \widehat{f}^n\Vert_{W^{1,2}(\Ozeta;L^2_M(B))}^2 \dt
+
c(\delta)
\int_I\Vert\partial_t\zeta\Vert_{W^{1/3,2}(\omega)}^{12/5}\Vert \nabx  \widehat{f}^n\Vert_{L^2(\Ozeta;L^2_M(B))}^2 \dt.
\end{aligned}
\end{equation}
Next, by Ladyszenskaya's inequality,
\begin{equation}
\begin{aligned}
\label{fokkerEnergyEst8}
\bigg\vert\int_I\int_{\Ozeta\times B}&M  (\mathbf{w}\cdot \nabx)   \widehat{f}^n \Delx  \widehat{f}^n\dq\dx\dt
\bigg\vert
\\&\leq
\int_I \Vert \mathbf w \Vert_{L^4(\Ozeta)}
\Vert \nabx  \widehat{f}^n\Vert_{L^4(\Ozeta;L^2_M(B))} \Vert \nabx^2  \widehat{f}^n\Vert_{L^2(\Ozeta;L^2_M(B))}\dt
\\
&\lesssim \int_I \Vert \mathbf w \Vert^{1/2}_{L^2(\Ozeta)}\Vert\nabx\mathbf w \Vert^{1/2}_{L^2(\Ozeta)}
\Vert \nabx  \widehat{f}^n\Vert^{1/2}_{L^2(\Ozeta;L^2_M(B))} \Vert \nabx^2  \widehat{f}^n\Vert_{L^2(\Ozeta;L^2_M(B))}^{3/2}\dt\\
&\leq \delta \int_I\Vert \nabx^2  \widehat{f}^n\Vert_{L^2(\Ozeta;L^2_M(B))}^2 \dt+
c(\delta)
\int_I\Vert\nabx\mathbf{w}\Vert_{L^2(\Ozeta)}^2\Vert \nabx  \widehat{f}^n\Vert_{L^2(\Ozeta;L^2_M(B))}^2 \dt.
\end{aligned}
\end{equation}
For the dissipative term, we obtain
\begin{align}
\label{fokkerEnergyEst9}
\int_I\int_{\Ozeta\times B}  \Delx(M   \widehat{f}^n) \Delx  \widehat{f}^n\dq\dx\dt
=
\int_I\Vert\Delx  \widehat{f}^n\Vert_{L^2(\Ozeta;L^2_M(B))}\dt.
\end{align}
Finally,
% we use \eqref{fokkerPlankBoundarySpacexsc4}--\eqref{fokkerPlankBoundaryxsc4} and Sobolev embeddings and 
 we have by \eqref{fokkerPlankBoundarySpacexsc4Y}
\begin{align}
\int_I&\int_{\Ozeta\times B}\divq  \big(  \mathcal W(  \mathbf{w}) \bq M\widehat{f}^n\big)\Delx  \widehat{f}^n\dq\dx\dt\nonumber
\\&
=\sum_\gamma
\int_I\int_{\Ozeta\times B}\partial_\gamma  \big(  \mathcal W(\mathbf{w}) \bq M\widehat{f}^n\big):\partial_\gamma\nabq  \widehat{f}^n\dq\dx\dt\nonumber
\\
&
=\sum_\gamma
\int_I\int_{\Ozeta\times B}  \big( (  \partial_\gamma\mathcal W(\mathbf{w})) \bq M\widehat{f}^n\big):\partial_\gamma\nabq  \widehat{f}^n\dq\dx\dt\nonumber
\\
&
+\sum_\gamma
\int_I\int_{\Ozeta\times B} \big(  \mathcal W(\mathbf{w}) \bq M\partial_\gamma\widehat{f}^n\big):\partial_\gamma\nabq  \widehat{f}^n\dq\dx\dt\nonumber.
\end{align}
The last term vanishes again because of the skew-symmetry of $\mathcal W(\mathbf w)$, while the first one is bounded by (employing the maximum principle, cf. \eqref{fokkerEnergyEst5''}) 
\begin{align}
& 
\int_I\Vert\nabx^2\mathbf{w}\Vert_{L^2(\Ozeta)}\Vert  \widehat{f}^n\Vert_{L^\infty(\Ozeta;L^2_M(B))}
\Vert \nabx \widehat{f}^n\Vert_{L^2(\Ozeta;H^1_M(B))}\dt\nonumber
\\&
\leq
\delta\int_I\Vert\nabx \widehat{f}^n\Vert_{L^2(\Ozeta;H^1_M(B))}^2\dt
+c(\delta)
\int_I\Vert\mathbf{w}\Vert_{W^{2,2}(\Ozeta)}^2\dt,\label{fokkerEnergyEst10A}
\end{align}
where $\delta>0$. Finally, we note that
\begin{align}
-\int_I&\int_{\Ozeta\times B}\divq  \big( M \nabq    \widehat{f}^n
\big)\Delx  \widehat{f}^n\dq\dx\dt
=-\int_I\Vert\nabx \widehat{f}^n\Vert_{L^2(\Ozeta;H^1_M(B))}^2\dt\label{fokkerEnergyEst10Y}
\end{align}
as a consequence of \eqref{fokkerPlankBoundarySpacexsc4Y}.
By combining \eqref{fokkerEnergyEst7Y}--\eqref{fokkerEnergyEst10Y} we obtain
\eqref{fokkerEnergyNeg1Y}
uniformly in $n\in \mathbb{N}$.
\end{proof}

As far as the temporal regularity is concerned we have the following result.
\begin{lemma}\label{lem:mainFP4}
Let $(\widehat{f}_0,\zeta, \mathbf{w})$ satisfy  \eqref{fokkerPlanckDataAloneY}--\eqref{fokkerPlanckDataAloneBY}  and let $\widehat{f}^n$ be the corresponding solution to \eqref{eq:FP1ApproxY}.
Suppose further that  $\widetilde f_0\in L^{2}(\Omega_{\zeta(0)};L^2_M(B))$, where $\widetilde f_0$ is given by 
%\footnote{please change this back, otherwise we must explain how it depends on the data}
\begin{equation}
\begin{aligned}
\label{eq:FP1Approx1Initial2D}
M   \widetilde{f}_0 
&=
\Delx(M  \widehat{f}_0)
+
 \divq  \big( M \nabq    \widehat{f}_0
\big)
-
M  (\mathbf w(0)\cdot \nabx)   \widehat{f}_0
-
 \divq  \big(   (  \nabx\mathbf w(0)) \bq M\widehat{f}_0 \big) .
\end{aligned}
\end{equation} 
Then we have 
\begin{equation}
\begin{aligned}
\label{fokkerEnergyNeg3}
\sup_I \Vert \partial_t\widehat{f}^n(t)
&\Vert_{L^2(\Ozeta;L^2_M(B))}^2
+
 \int_I 
 \Vert \partial_t\widehat{f}^n
\Vert_{W^{1,2}(\Ozeta;L^2_M(B))}^2
 \dt
 +
 \int_I 
 \Vert \partial_t\widehat{f}^n
\Vert_{L^2(\Ozeta;H^1_M(B))}^2
 \dt\\
&\lesssim \sup_I\|\partial_t\zeta\|_{W^{1,2}(\omega)}^2
\int_I 
\Vert \nabx \widehat{f}^n\Vert_{W^{1,2}(\Ozeta;L^2_M(B))}^2\dt
+\int_I 
\Vert \nabx\partial_t \mathbf{w}\Vert_{L^2(\Ozeta)}^{2}\dt
\\&
+\sup_I \Vert \nabx \widehat{f}^n\Vert_{L^2(\Ozeta;L^2_M(B))}^2\int_I\Vert \nabx^2 \widehat{f}^n\Vert^{2}_{L^2(\Ozeta;L^2_M(B))}\dt
%\\
%&\lesssim c\Big(\sup_I\|\partial_t\zeta\|_{W^{1,2}(\omega)}^2+1\Big)\exp\bigg(c\int_I \Vert  \mathbf{w} \Vert_{W^{1,2}(\Ozeta)}^2 \dt
% +
%c \int_I\Vert\partial_t\zeta\Vert_{W^{1/2-\kappa,2}(\omega)}^{2+\kappa}\dt
%\bigg) 
%\Big(\Vert \widehat{f}_0
%\Vert_{W^{1,2}(\Omega_{\zeta(0)};L^2_M(B))}^3+1\Big)\\&+c\Big(\sup_I\|\partial_t\zeta\|_{W^{1,2}(\omega)}^2+1\Big)\bigg(\int_I \Vert  \mathbf{w} \Vert_{W^{2,2}(\Ozeta)}^2 \dt+\int_I \Vert  \partial_t\mathbf{w} \Vert_{W^{1,2}(\Ozeta)}^2 \dt\bigg)^2
\end{aligned}
\end{equation}
%for all $\kappa\in(0,1/2)$ 
uniformly in $n\in\mathbb{N}$, where the hidden constant also depends on $\sup_I\|\mathbf w\|_{L^2(\Ozeta)}$, $\sup_I\|\partial_t\mathbf w\|_{L^2(\Ozeta)}$, $\sup_I\|\zeta \|_{W^{1,\infty}(\omega)}$ and $\|\widehat f_0\|_{L^\infty(\Omega_{\zeta(0)};L^2_M(B))}$.
\end{lemma}
\begin{proof}
Now  set $\widetilde{f}^n:=\partial_t\widehat{f}^n$ and consider the following equation
\begin{equation}
\begin{aligned}
\label{eq:FP1Approx1Y}
M\big(\partial_t \widetilde{f}^n + (\mathbf{w}\cdot \nabx)  \widetilde{f}^n)
&+
 \divq  \big( \mathcal W(\mathbf{w}) \bq M\widetilde{f}^n \big) 
-
\Delx(M \widetilde{f}^n)
-
 \divq  \big( M \nabq \widetilde{f}^n
\big)
\\&
=
-
M (\partial_t\mathbf{w}\cdot \nabx)   \widehat{f}^n
+
 \divq  \big( \mathcal W(  \partial_t\mathbf{w}) \bq M\widehat{f}^n \big) 
\end{aligned}
\end{equation}
in $I\times\Omega_\zeta\times B$ subject to%\todo{Note \eqref{fokkerPlankBoundaryxsc4time} because of the cut-off!}
\begin{align}
&\widetilde{f}^n(0, \cdot, \cdot) =\widetilde{f}_0 \geq 0
& \quad \text{in }\Omega_{\zeta_0} \times B,
\label{fokkerPlankIintialxsec4time}
\\
&
\nabx\widetilde{f}^n\cdot \bn_\zeta =-\nabx\widehat{f}^n\cdot \partial_t\bn_\zeta
&\quad \text{on }I \times \partial\Omega_\zeta \times B,
\label{fokkerPlankBoundarySpacexsc4timeY}
\\
%&\big(\nabq\widetilde{f}^n -  (\nabx \mathbf w) \bq \widetilde{f}^n -  (\nabx \partial_t\mathbf w) \bq \widehat{f}^n
% \big) \cdot \bn =0
& M\nabq\widetilde{f}^n \cdot \frac{\bq}{\vert\bq\vert} =0
&\quad \text{on }I \times \Omega_\zeta \times \partial \overline{B}
\label{fokkerPlankBoundaryxsc4time}
\end{align}
and where $\widetilde{f}_0$ satisfies \eqref{eq:FP1Approx1Initial}. 
%\begin{equation}
%\begin{aligned}
%\label{eq:FP1Approx1Initial}
%M   \widetilde{f}_0 
%&=
%\Delx(M  \widehat{f}_0)
%+
% \divq  \big( M \nabq    \widehat{f}_0
%\big)
%-
%M  (\mathbf{w}_0\cdot \nabx)   \widehat{f}_0
%-
% \divq  \big(   (  \nabx\mathbf{w}_0) \bq M\widehat{f}_0 \big) .
%\end{aligned}
%\end{equation}
We now test \eqref{eq:FP1Approx1Y} with $\widetilde{f}^n$. Since the left-hand side of \eqref{eq:FP1Approx1Y} is of the same form as \eqref{eq:FP1ApproxY}, we obtain similarly to \eqref{fokkerEnergyY}
\begin{equation}
\begin{aligned}
\label{fokkerEnergy1}
\int_I\int_{\Ozeta}
&\Big(
\frac{1}{2}\partial_t\Vert   \widetilde{f}^n \Vert_{L^2_M(B)}^2
+
\frac{1}{2}(\mathbf{w}\cdot \nabx)\Vert    \widetilde{f}^n \Vert_{L^2_M(B)}^2
+
\Vert   \nabx \widetilde{f}^n \Vert_{L^2_M(B)}^2
+ 
\frac{1}{2}
\Vert  \widetilde{f}^n\Vert_{H^1_M(B)}^2
\Big)\dx\dt
\\&
=
-
\int_I\int_{\partial\Ozeta\times B}
    M\partial_t\bn_\zeta\cdot\nabx\widehat{f}^n 
  \widetilde{f}^n\dq\dx\dt
\\&
+
\int_I\int_{\Ozeta\times B}
 (  \partial_t\mathcal W(\mathbf{w})) \bq M\widehat{f}^n \cdot \nabq
  \widetilde{f}^n\dq\dx\dt
  -
\int_I\int_{\Ozeta\times B} M
(\partial_t\mathbf{w}\cdot \nabx)   \widehat{f}^n
  \widetilde{f}^n\dq\dx\dt,
\end{aligned}
\end{equation}
where the second term on the right-hand side is due to \eqref{fokkerPlankBoundarySpacexsc4timeY}. For the boundary term, we use the trace theorem and Lemma \ref{lem:mainFP2Y} to obtain
\begin{align*}
\bigg\vert\int_I&\int_{\partial\Ozeta\times B}
    M\partial_t\bn_\zeta\cdot\nabx\widehat{f}^n 
  \widetilde{f}^n\dq\dx\dt
  \bigg\vert
\\
&\leq
\int_I 
\Vert \widetilde{f}^n\Vert_{L^{4}(\partial\Ozeta;L^2_M(B))}\|\partial_t\zeta\|_{W^{1,2}(\omega)}
\Vert \nabx \widehat{f}^n\Vert_{L^{4}(\partial\Ozeta;L^2_M(B))}\dt
\\
&\leq
\delta
\int_I 
\Vert \widetilde{f}^n\Vert_{W^{1,2}(\Ozeta;L^2_M(B))}^2\dt
+c(\delta)\sup_I\|\partial_t\zeta\|_{W^{1,2}(\omega)}^2
\int_I 
\Vert \nabx \widehat{f}^n\Vert_{W^{1,2}(\Ozeta;L^2_M(B))}^2\dt.
%&\leq
%\delta
%\int_I 
%\Vert \widetilde{f}^n\Vert_{W^{1,2}(\Ozeta;L^2_M(B))}^2\dt
%\\
%&+  c(\delta)\sup_I\|\partial_t\zeta\|_{W^{1,2}(\omega)}^2\exp\bigg(c\int_I \Vert  \mathbf{w} \Vert_{W^{1,2}(\Ozeta)}^2 \dt
% +
%c \int_I\Vert\partial_t\zeta\Vert_{W^{1/2-\kappa,2}(\omega)}^{2+\kappa}\dt
%\bigg) 
%\Vert \widehat{f}_0
%\Vert_{W^{1,2}(\Omega_{\zeta(0)};L^2_M(B))}^2\\&+c(\delta)\sup_I\|\partial_t\zeta\|_{W^{1,2}(\omega)}^2\int_I \Vert  \mathbf{w} \Vert_{W^{2,2}(\Ozeta)}^2 \dt
\end{align*}
Next we use \eqref{fokkerPlanckDataAloneY} and the maximum principle \eqref{fokkerEnergyEst5''} to infer
%\todo{the bound $x\leq \exp(x)$ is trivial, shouldn't be mentioned}
\begin{equation}
\begin{aligned}
\bigg\vert\int_I&\int_{ \Ozeta\times B}
  \mathcal W(  \partial_t\mathbf{w}) \bq M\widehat{f}^n \cdot \nabq
  \widetilde{f}^n\dq\dx\dt
  \bigg\vert
\\&\leq
c(\delta)
\int_I\int_{ \Ozeta }
\vert\partial_t\nabx\mathbf{w}\vert^2
\Vert  \widehat{f}^n\Vert_{L^2_M(B)}^2\dx\dt+\delta
\int_I\Vert    \widetilde{f}^n\Vert_{L^2(\Ozeta;H^1_M(B))}^2\dt
\\&
\leq
c(\delta)
\int_I 
\Vert\partial_t\nabx\mathbf{w}\Vert_{L^2(\Ozeta )}^2\dt +\delta
\int_I\Vert    \widetilde{f}^n\Vert_{L^2(\Ozeta;H^1_M(B))}^2\dt.
\end{aligned}
\end{equation} 
Finally, we use Lemma \ref{lem:mainFP1Y} to also obtain
\begin{align*}
\bigg\vert\int_I\int_{ \Ozeta\times B}& M
(\partial_t\mathbf{w}\cdot \nabx)   \widehat{f}^n
  \widetilde{f}^n\dq\dx\dt
  \bigg\vert
\\
&\lesssim
\int_I 
\Vert \partial_t \mathbf{w}\Vert_{L^4(\Ozeta)}
\Vert \nabx \widehat{f}^n\Vert_{L^4(\Ozeta;L^2_M(B))}\Vert \widetilde{f}^n\Vert_{L^2(\Ozeta;L^2_M(B))}\dt\\
&\lesssim
\int_I 
\Vert \partial_t \mathbf{w}\Vert^{1/2}_{L^2(\Ozeta)}\Vert\partial_t \nabx \mathbf{w}\Vert^{1/2}_{L^2(\Ozeta)}\Vert \nabx \widehat{f}^n\Vert^{1/2}_{L^2(\Ozeta;L^2_M(B))}\Vert \nabx^2 \widehat{f}^n\Vert^{1/2}_{L^2(\Ozeta;L^2_M(B))}\Vert \widetilde{f}^n\Vert_{L^2(\Ozeta;L^2_M(B))}\dt
\\
&\lesssim
\int_I 
\Vert \partial_t\nabx \mathbf{w}\Vert^{1/2}_{L^2(\Ozeta)}\Vert \nabx \widehat{f}^n\Vert^{1/2}_{L^2(\Ozeta;L^2_M(B))}\Vert \nabx^2 \widehat{f}^n\Vert^{1/2}_{L^2(\Ozeta;L^2_M(B))}\Vert \widetilde{f}^n\Vert_{L^2(\Ozeta;L^2_M(B))}\dt
\\
&\lesssim 
\int_I 
\Vert \partial_t\nabx \mathbf{w}\Vert_{L^2(\Ozeta)}^{2}\dt
+
\sup_I \Vert \nabx \widehat{f}^n\Vert_{L^2(\Ozeta;L^2_M(B))}^2\int_I\Vert \nabx^2 \widehat{f}^n\Vert^{2}_{L^2(\Ozeta;L^2_M(B))}\dt
\\
&+\int_I\Vert \widetilde{f}^n\Vert_{L^2(\Ozeta;L^2_M(B))}^2\dt.
%&\lesssim 1+ \exp\bigg(c\int_I \Vert  \mathbf{w} \Vert_{W^{1,2}(\Ozeta)}^2 \dt
% +
%c \int_I\Vert\partial_t\zeta\Vert_{W^{1/2-\kappa,2}(\omega)}^{2+\kappa}\dt
%\bigg) 
%\Vert \widehat{f}_0
%\Vert_{W^{1,2}(\Omega_{\zeta(0)};L^2_M(B))}^3\\&+\bigg(\int_I \big(\Vert  \partial_t\mathbf{w} \Vert_{W^{1,2}(\Ozeta)}^2+\Vert  \mathbf{w} \Vert_{W^{2,2}(\Ozeta)}^2 \big)\dt\bigg)^{2}+\int_I\Vert \widetilde{f}^n\Vert_{L^2(\Ozeta;L^2_M(B))}^2\dt.
\end{align*} 
Subsequently we use Reynold's transport and Gronwall's lemma   yielding the claim. 
\end{proof}

\subsection{The fully coupled system}
\label{sec:fullyCoupled}
\noindent 

We consider now the set of equations
\begin{align}
\partial_t^2\eta -\partial_t\Dely \eta + \Dely^2\eta&=g-\bn^\top\mathbb{T}\circ\bm{\varphi}_\eta\bn_\eta  \det(\naby\bm{\varphi}_\eta) ,
\label{shellEq2D}
\\
\partial_t \bu  + (\mathbf{u}\cdot \nabx)\mathbf{u} 
&= 
\Delx \bu -\nabx\pi+ \bff+
\divx   \mathbb{S}_\bq(\widehat{f}),
\label{momEq2D}
\\
\divx \bu&=0,
\label{contEq2D}
\\
M\big(\partial_t \widehat{f} + (\mathbf{u}\cdot \nabx) \widehat{f})
+
 \divq  \big( (\mathcal W(\bu)) \bq M\widehat{f} \big) 
&=
\Delx(M \widehat{f})
+
\divq  \big( M \nabq  \widehat{f}
\big),
\label{fokkerPlank2D}
\end{align}
where
\begin{align*}
\mathbb{T}=(\nabx\bu+\nabx\bu^\top)-\pi\mathbb{I}_{2\times2}+\mathbb{S}_\bq(\widehat{f}),\quad 
\mathbb{S}_\bq(\widehat{f} ) 
=
 \int_B M(\bq) \widehat{f} (t, \mathbf{x},\mathbf{q})\nabq U(\tfrac{1}{2}\vert\mathbf{q}\vert^2 ) \otimes\mathbf{q} \dq, 
%= 
% \int_B  M(\bq) \widehat{f}  (t, \mathbf{x},\mathbf{q})\, \mathbf{F}(\mathbf{q}) \otimes \mathbf{q} \dq.
\end{align*}
subject to initial conditions $\bu_0,\eta_0,\eta_\star,\widehat f_0$ and boundary conditions
\begin{align}
&\bu  \circ \bm{\varphi}_{\eta} =(\partial_t\eta)\bn
&\quad \text{on }I \times \omega ,\\
&
\nabx\widehat{f}\cdot \bn_\eta =0
&\quad \text{on }I \times \partial\Omega_\eta \times B,
\label{fokkerPlankBoundarySpacexsc42D}
\\
&M\big(\nabq\widehat{f}   -  \mathcal W(\mathbf u) \bq \widehat{f}
 \big) \cdot \frac{\bq}{\vert\bq\vert} =0
&\quad \text{on }I \times \Omega_\eta \times \partial \overline{B}.
\label{fokkerPlankBoundaryxsc42D}
\end{align}
A weak solution to \eqref{shellEq2D}--\eqref{fokkerPlankBoundaryxsc42D} can be defined as in Defintion \ref{def:weakSolution} (simply replacing $\nabx \bu$ by $\mathcal W(\bu)$ in the last integral of (d)). Its existence follows again from \cite{breit2021incompressible}; indeed replacing $\nabx \bu$ by $\mathcal W(\bu)$ does not alter the arguments there.
We speak about a strong solution, if a weak solutions satisfies
\begin{align*}
\eta&\in W^{2,2}\big(I;L^{2}(\omega)  \big)\cap W^{1,2}\big(I;W^{2,2}(\omega)  \big)\cap L^{\infty}\big(I;W^{3,2}(\omega)  \big)\cap L^{2}\big(I;W^{4,2}(\omega)  \big),\\
\bu &\in
W^{1,2} \big(I; L^{2}(\Oeta) \big)\cap  L^{2}\big(I;W^{2,2}(\Oeta)  \big),\quad
\pi \in
L^2\big(I;L^{2}(\Oeta)  \big),
\\ 
\widehat{f} & 
\in   L^{\infty}\big(I;W^{1,2}(\Oeta;L^2_M(B))  \big)
\cap 
L^{2}\big(I;W^{2,2}(\Oeta;L^2_M(B))  \big)
\\
&\qquad\qquad\cap  L^{2}\big(I;W^{1,2}(\Oeta;H^1_M(B))  \big).
\end{align*}
and $(\bu,\pi)$ solves the mometum equation a.a. in $I\times\Omega_\eta$. We have the following result
\begin{theorem}\label{thm:main2D}
Let $(\bff, g, \eta_0, \eta_\star, \bu_0, \widehat{f}_0)$ be a dataset satisfying 
\begin{equation}
\begin{aligned}
\label{dataset2D} 
&\bff \in L^2\big(I; L^2_{\mathrm{loc}}(\mathbb{R}^3 )\big),\quad
g \in L^2\big(I; W^{1,2}(\omega)\big), \quad
\eta_0 \in W^{3,2}(\omega) \text{ with } \Vert \eta_0 \Vert_{L^\infty( \omega)} < L,
\\
&\eta_\star \in W^{1,2}(\omega), \quad  \widehat{f}_0\in L^2\big( \Omega_{\eta_0} ;H^1_M(B)\big)\cap L^\infty\big(\Omega_{\eta_{0}};L^2_M(B)\big), \\
&\bu_0\in W^{1,2}_{\mathrm{\divx}}(\Omega_{\eta_0} ) \text{ is such that }\bu_0 \circ \bm{\varphi}_{\eta_0} =\eta_\star \bn \text{ on } \omega.
\end{aligned}
\end{equation}  
There is a unique  strong solution $( \eta, \bu,  \pi, \widehat{f} )$ of \eqref{shellEq2D}--\eqref{fokkerPlankBoundaryxsc42D}. The interval of existence is of the form $I = (0, t)$, where $t < T$ only in case $\Omega_{\eta(s)}$ approaches a self-intersection when $s\rightarrow t$ or it degenerates\footnote{Self-intersection and degeneracy are excluded if $\|\eta\|_{L^\infty_{t,y}}<L$, cf. \eqref{eq:boundary1}.} (namely, if $\lim_{s\rightarrow t}\partial_y\bm{\varphi}_\eta(s,y)=0$ or $\lim_{s\rightarrow t}\bn(y)\cdot\bn_{\eta(s)}(y)=0$ for some $y\in\omega$).
\end{theorem}

\begin{proof}
Take a weak solution $( \eta, \bu, \widehat{f} )$ to \eqref{shellEq2D}--\eqref{fokkerPlankBoundaryxsc42D} which exists according to \cite{breit2021incompressible}. By \cite{breit2022regularity} with right-hand side $\divx\mathbb S(\widehat{f})\in L^2(I;L^2(\Omega_\eta))$ there is a strong solution
$( \eta, \bu,  \pi)$  to the fluid-structure system (which belongs to the correct function space).
 By weak-strong uniqueness (see \cite[Remark 5.2]{BMSS}) it must coincide with the weak solution. By Lemma \ref{lem:mainFP2Y} we also get spatial regularity of $\widehat f$ such that the constructed solution lives in the claimed function spaces and the proof is complete.
% Note that Lemma \ref{lem:mainFP2} even applies without diffusion in the shell equation, but \cite{breit2022regularity} does not.
\end{proof}

\begin{remark}[Temporal regularity]
Having (as in the proof of Theorem \ref{thm:main2D} above) the weak solution from \cite{breit2022regularity} at hand, we apply Lemma \ref{lem:mainFP2Y} and, eventually,
Lemma \ref{lem:mainFP4}. By  Lemma \ref{lem:mainFP4} we get
\begin{equation}
\begin{aligned}
\label{fokkerEnergyNeg3'}
\sup_I \Vert \partial_t\widehat{f}^n(t)
\Vert_{L^2(\Oeta;L^2_M(B))}^2&+
 \int_I 
 \Vert \partial_t\widehat{f}^n
\Vert_{W^{1,2}(\Oeta;L^2_M(B))}^2
 \dt
 +
 \int_I 
 \Vert \partial_t\widehat{f}^n
\Vert_{L^2(\Ozeta;H^1_M(B))}^2
 \dt\\
&\lesssim1+ \sup_I\|\partial_t\zeta\|_{W^{1,2}(\omega)}^2+\int_I 
\Vert \partial_t\nabx \mathbf{u}\Vert_{L^2(\Oeta)}^{2}\dt.
%\\
%&\lesssim c\Big(\sup_I\|\partial_t\zeta\|_{W^{1,2}(\omega)}^2+1\Big)\exp\bigg(c\int_I \Vert  \mathbf{w} \Vert_{W^{1,2}(\Ozeta)}^2 \dt
% +
%c \int_I\Vert\partial_t\zeta\Vert_{W^{1/2-\kappa,2}(\omega)}^{2+\kappa}\dt
%\bigg) 
%\Big(\Vert \widehat{f}_0
%\Vert_{W^{1,2}(\Omega_{\zeta(0)};L^2_M(B))}^3+1\Big)\\&+c\Big(\sup_I\|\partial_t\zeta\|_{W^{1,2}(\omega)}^2+1\Big)\bigg(\int_I \Vert  \mathbf{w} \Vert_{W^{2,2}(\Ozeta)}^2 \dt+\int_I \Vert  \partial_t\mathbf{w} \Vert_{W^{1,2}(\Ozeta)}^2 \dt\bigg)^2
\end{aligned}
\end{equation}
If a flat reference geometry is considered (the case of elastic plates),
by \cite[Theorem 4.4]{schwarzacherSu}\footnote{The result of \cite{schwarzacherSu} applies even without dissipation in the structure equation.} the right-hand side is controlled by the initial data and $\int_I\Vert \partial_t\widehat{f}^n(t)
\Vert_{L^2(\Ozeta;L^2_M(B))}^2\dt$ hence the estimate can be closed by Gronwall's lemma. Again by \cite[Theorem 4.4]{schwarzacherSu} one gets temporal regularity for the fluid. In conclusion, for elastic shells one obtains
\begin{align*}
 \widehat f&\in W^{1,\infty}(I;L^2(\Omega_\eta;L^2_M(B)))\cap W^{1,2}(I;L^2(\Omega_\eta;H^1_M(B))\cap W^{1,2}(I;L^2(\Omega_\eta;L^2_M(B))),\\
 \bu&\in W^{1,\infty}(I;L^2(\Omega_\eta))\cap W^{1,2}(I;W^{1,2}(\Omega_\eta)).
\end{align*}
The same result certainly applies when considering the problem in a fixed fluid-domain (as the estimate from \cite[Theorem 4.4]{schwarzacherSu} is well-known then). However, it remains open  whether the result from \cite{schwarzacherSu} holds for elastic shells to conclude in the same way.
\end{remark}

%\bibliographystyle{spmpsci}
%\bibliography{myBibliography}

\end{document}